\documentclass{article}

\usepackage[a4paper,hmargin={2cm,2cm},vmargin={3cm,3cm}]{geometry}

\usepackage[dvipdfm]{hyperref}

\usepackage[all]{xy}

\usepackage{pb-diagram}
\usepackage{pb-xy}

\usepackage{graphicx}

\usepackage{amsmath}
\usepackage{amsthm}
\usepackage{amssymb}

\usepackage{bm}

\usepackage{latexsym}

   {\begin{proof}[Sketch of Proof of #1]}%
   {\end{proof}}

\theoremstyle{plain}
  \newtheorem{theorem}{Theorem}[section]
  \newtheorem{corollary}[theorem]{Corollary}
  
  \newtheorem{lemma}[theorem]{Lemma}
  
  \newtheorem{proposition}[theorem]{Proposition}

\theoremstyle{definition}
  \newtheorem{definition}[theorem]{Definition}
  \newtheorem{assumption}[theorem]{Assumption}

  \newtheorem{convention}[theorem]{Convention}
  
  \newtheorem{ex}[theorem]{Example}

  \newtheorem{remark}[theorem]{Remark}

  \newenvironment{example}{\begin{ex}}{\qed\end{ex}}

\newcommand{\category}[1]{\operatorname{\mathbf{#1}}}
\newcommand{\categories}[1]{{#1}\textbf{-}\mathbf{Categories}}

  \newcommand{\Categories}{\category{Categories}}

  \newcommand{\Sets}{\operatorname{\mathbf{Sets}}}

  \newcommand{\Spaces}{\operatorname{\mathbf{Spaces}}}
  \newcommand{\Spectra}{\operatorname{\mathbf{Spectra}}}

  \newcommand{\lMod}[1]{{#1}\text{-}\mathbf{Mod}}
  
  \newcommand{\lComod}[1]{{#1}\text{-}\mathbf{Comod}}
  \newcommand{\rComod}[1]{\mathbf{Comod}\text{-}{#1}}
  \newcommand{\biComod}[2]{#1\text{-}\mathbf{Comod}\text{-}#2} % bicomodules

  \newcommand{\Mor}{\operatorname{Mor}}

  \newcommand{\Ima}{\operatorname{Im}}

  \DeclareMathOperator*{\hocolim}{\mathrm{hocolim}}

\newcommand{\lset}[2]{%
\left.\left\{#1 \ \right| \ #2\right\}
}

  \newcommand{\rarrow}[1]{\buildrel #1 \over \longrightarrow}
  \newcommand{\larrow}[1]{\buildrel #1 \over \longleftarrow}

  \newcommand{\pr}{\operatorname{\mathrm{pr}}}

  \newcommand{\transpose}[1]{\raisebox{1ex}{$\scriptstyle t$}\kern-0.2ex #1}

  \newcommand{\op}{{\operatorname{\mathit{op}}}} % opposite category
  \newcommand{\cof}{\operatorname{\mathrm{cof}}}
  \newcommand{\fib}{\operatorname{\mathrm{fib}}}
  
  \newcommand{\String}{\mathit{\String}}

  \newcommand{\Gr}{\operatorname{Gr}}

\newcommand{\sk}{\operatorname{sk}}

\newcommand{\fixhyperref}{%
\ifnum 42146=\euc"A4A2 \AtBeginDvi{}\else
\AtBeginDvi{}\fi}

\newcommand{\rOplax}{\overrightarrow{\category{Oplax}}}
\newcommand{\lOplax}{\overleftarrow{\category{Oplax}}}
\newcommand{\rLax}{\overrightarrow{\category{Lax}}}
\newcommand{\lLax}{\overleftarrow{\category{Lax}}}
\newcommand{\rFunct}{\overrightarrow{\category{Funct}}}
\newcommand{\lFunct}{\overleftarrow{\category{Funct}}}
\newcommand{\rcats}[1]{{#1}\textbf{-}\overrightarrow{\mathbf{Categories}}}
\newcommand{\lcats}[1]{{#1}\textbf{-}\overleftarrow{\mathbf{Categories}}}
\newcommand{\rGamma}{\overrightarrow{\Gamma}}
\newcommand{\lGamma}{\overleftarrow{\Gamma}}
\newcommand{\rrComod}[1]{\overrightarrow{\mathbf{Comod}}\textbf{-}{#1}}
\newcommand{\rlComod}[1]{\overleftarrow{\mathbf{Comod}}\textbf{-}{#1}}
\newcommand{\llComod}[1]{{#1}\textbf{-}\overleftarrow{\mathbf{Comod}}}
\newcommand{\lrComod}[1]{{#1}\textbf{-}\overrightarrow{\mathbf{Comod}}}
\newcommand{\bilComod}[2]{{#1}\textbf{-}\overleftarrow{\mathbf{Comod}}\textbf{-}{#2}}
\newcommand{\birComod}[2]{{#1}\textbf{-}\overrightarrow{\mathbf{Comod}}\textbf{-}{#2}}

\newcommand{\bibdir}{bib}

\title{\bfseries The Grothendieck Construction and Gradings for Enriched
Categories} 
\author{Dai Tamaki
\thanks{Department of Mathematical Sciences, Shinshu University,
Matsumoto, 390-8621, Japan}
\thanks{Partially supported by Grants-in-Aid for Scientific Research,
Ministry of Education, Culture, Sports, Science and Technology, Japan:
17540070}
}

\begin{document}
\maketitle

\begin{abstract}
 The Grothendieck construction is a process to form a single category
 from a diagram of small categories. In this paper, we extend the
 definition of the Grothendieck construction to diagrams of small
 categories enriched over a symmetric monoidal category satisfying
 certain conditions. Symmetric monoidal categories satisfying the
 conditions in this paper include the category of $k$-modules over a
 commutative ring $k$, the category of chain complexes, the category of
 simplicial sets, the category of topological spaces, and the category
 of modern spectra. 
 In particular, we obtain a generalization of the
 orbit category construction in \cite{math/0312214}. We also extend the
 notion of graded categories and show that the Grothendieck construction
 takes values in the category of graded categories. Our definition of
 graded category does not require any coproduct decompositions and
 generalizes $k$-linear graded categories indexed by small categories
 defined in \cite{Lowen08}. 

 There are two popular ways to construct functors from the category of
 graded categories to the category of oplax functors. One of them is the
 smash product construction defined and studied in
 \cite{math/0312214,0807.4706,0905.3884} for $k$-linear categories and
 the other one is the fiber functor. We construct extensions of these
 functors for enriched categories and show that they are ``right adjoint''
 to the Grothendieck construction in suitable senses. 

 As a byproduct, we obtain a new short description of small enriched
 categories. 
\end{abstract}

\tableofcontents

\section{Introduction}

\subsection{The Grothendieck Construction}

Given a diagram of small categories
\[
 X : I \longrightarrow \Categories
\]
indexed by a small category $I$, there is a way to form a single
category $\Gr(X)$, called the Grothendieck construction on $X$. The
construction first appeared in \S8 of Expos{\'e} VI in \cite{SGA1}. It
can be used to prove an equivalence of categories of prestacks and
fibered categories 
\begin{equation}
 \Gr : \category{Prestacks}(I) \longleftrightarrow
 \category{Fibered}(I) : \Gamma, 
 \label{prestack_and_prefibered}
\end{equation}
where $\category{Prestacks}(I)$ is the category of lax presheaves
(contravariant lax
functors satisfying certain conditions)
\[
 X : I^{\op} \longrightarrow \Categories
\]
and $\category{Fibered}(I)$ is the category of fibered categories
\[
 \pi : E \longrightarrow I
\]
over $I$, which is a full subcategory of the category of prefibered
categories $\category{Prefibered}(I)$ over $I$. See
\cite{math.AT/0110247, math.AG/0412512}, for example. 

It was Quillen who first realized the usefulness of (pre)fibered
categories in homotopy theory. In particular, he proved famous Theorem A
and B for the classifying spaces of small categories in
\cite{Quillen73}. The Grothendieck construction was used implicitly in
the proofs of these theorems. Subsequently, the classifying space of
the Grothendieck construction of a diagram of categories was studied by
Thomason \cite{Thomason79-1} who found a description in terms of the
homotopy colimit construction of Bousfield and Kan
\cite{Bousfield-KanBook} 
\begin{equation}
 B\Gr(X) \simeq \hocolim_{I} B\circ X,
  \label{Gr_and_hocolim}
\end{equation}
where 
\[
 B : \Categories \longrightarrow \Spaces
\]
is the classifying space functor described, for example, in
\cite{SegalBG}. 

Since then the Grothendieck construction has been 
one of the most indispensable tools in homotopy theory of classifying
spaces, as is exposed by Dwyer \cite{Dwyer-Henn}.  
The work of Quillen \cite{Quillen78} suggests the usefulness of the
classifying space techniques in combinatorics in which posets are one of
the central objects of study. When the indexing category $I$ is a poset and
the functor $X$ takes values in the category of posets, the Grothendieck
construction of $X$ is called the poset limit of $X$. Thomason's
homotopy colimit description (\ref{Gr_and_hocolim}) has been proved to be
useful in topological combinatorics. See
\cite{Welker-Ziegler-Zivaljevic99}, for example. 

When $G$ is a group regarded as a category with a single object, giving
a functor
\[
 X : G \longrightarrow \category{Categories}
\]
is equivalent to giving a category $X$ equipped with a left action of
$G$. The Grothendieck construction has been used implicitly in this
context. One of the most fundamental examples is the semidirect product
of groups. The translation groupoid $G\ltimes X$ for an action of a
topological group $G$ on a space $X$ used in the study of orbifolds
\cite{Moerdijk02} is another example. 
When $X$ takes values in posets,
Borcherds \cite{Borcherds98} called the Grothendieck construction the
homotopy quotient of $X$ by $G$. This terminology is based on Thomason's
homotopy equivalence (\ref{Gr_and_hocolim})
\[
 B\Gr(X) \simeq \hocolim_G BX = EG\times_G BX
\]
and the fact that the Borel construction $EG\times_G BX$ is regarded as
a ``homotopy theoretic quotient'' of $BX$ under the action of $G$.

Algebraists have been studying group actions on $k$-linear
categories
\[
 X : G \longrightarrow \categories{k}
\]
and several ``quotient category'' constructions are known in relation
to covering theory of $k$-linear categories. It turns out that some of
them \cite{math/0312214,0807.4706,0905.3884} are equivalent to a
$k$-linear version of the Grothendieck construction. More general
diagrams of $k$-linear categories are considered by Gerstenhaber and
Schack in their study of deformation theory. In particular, the
Grothendieck construction was used in \cite{Gerstenhaber-Schack83-2} in 
a process of assembling a diagram $A$ of $k$-linear category into an
algebra $A!$.

Thomason's theorem also suggests the Grothendieck construction can be
regarded as a kind of colimit construction. It is also stated in
Thomason's paper that the construction can be naturally extended to
oplax functors. (See Definition \ref{oplax_functor_definition} for oplax
functors.) In fact, category theorists 
studied the Grothendieck construction as a model of $2$-colimits in the
bicategory of small categories. It was proved by J.~Gray
\cite{J.Gray69} (also stated in \cite{Thomason79-1}) that the
Grothendieck construction regarded as a functor 
\[
 \Gr : \lOplax(I,\category{Categories}) \longrightarrow
 \Categories 
\]
is left adjoint to the diagonal (constant) functor
\[
 \Delta :  \Categories \longrightarrow
 \lOplax(I,\category{Categories}), 
\]
where $\lOplax(I,\category{Categories})$ is the $2$-category of oplax
functors. 

The original motivation for the Grothendieck construction in
\cite{SGA1} and the equivalence (\ref{prestack_and_prefibered}) suggest,
however, we should regard $\Gr$ as a functor
\[
 \Gr : \lOplax(I,\category{Categories}) \longrightarrow
 \overleftarrow{\Categories}\downarrow I,
\]
where $\overleftarrow{\Categories}\downarrow I$ is a $2$-category of
comma categories over $I$ whose
morphisms are relaxed by taking ``left natural transformations'' into
account. See \S\ref{2-category} for precise definitions.

On the other hand, the orbit category construction in
\cite{math/0312214,0807.4706} gives rise to a functor
\[
 \Gr : \lFunct(G,\categories{k}) \longrightarrow
 \lcats{k}_{G},
\]
where $\lcats{k}_{G}$ is the $2$-category of $G$-graded
$k$-linear categories. In a recent paper \cite{0905.3884} Asashiba
proved that this functor induces an equivalence of $2$-categories
\begin{equation}
 \Gr : \lFunct(G,\categories{k}) \longleftrightarrow
 \lcats{k}_{G} : \Gamma.
\label{G-category_and_graded_category} 
\end{equation}
As is suggested by the similarity between this equivalence and
(\ref{prestack_and_prefibered}), we should regard the orbit category
construction as a $k$-linear version of the Grothendieck construction.

This ubiquity of the Grothendieck construction suggests us to work in
a more general framework. For example, when we study the derived category of
a $k$-linear category equipped with a group action, it would be useful
if we have a general theory of the Grothendieck construction for
dg categories, i.e.\ categories enriched over the category of chain
complexes, which can be regarded as a model for enhanced 
triangulated categories according to Bondal and Kapranov
\cite{Bondal-Kapranov90}. Another important 
model of enhanced triangulated categories is the notion of stable
quasicategory (stable $(\infty,1)$-category) by Lurie
\cite{math.CT/0608228}, which is closely related 
to stable simplicial categories appeared in a work of To{\"e}n and
Vezzosi \cite{math/0210125}, i.e.\ categories 
enriched over the category of simplicial sets.

\subsection{Gradings of Categories}

In order to fully understand the meaning of the above similarities and
attack the problem of extending the Grothendieck construction to general 
enriched categories, we should find a unified way
to handle the comma category $\overleftarrow{\Categories}\downarrow I$
and the category 
$\lcats{k}_{G}$ of $G$-graded $k$-linear categories. It is
immediate to extend the definition of the Grothendieck construction to
diagrams of enriched categories, as we will see in
\S\ref{definition_for_oplax_functor}. 
The problem is to find the right co-domain category of the
Grothendieck construction for general enriched categories. 

The notion of group graded $k$-linear categories has been used as a
natural ``many objectification'' of that of group graded
$k$-algebras. They are often defined in terms of coproduct
decompositions, i.e.\ a $k$-linear category $A$ graded by a group $G$ is
a $k$-linear category whose module $A(x,y)$ of morphisms has a coproduct
decomposition  
\[
 A(x,y) = \bigoplus_{g\in G} A^g(x,y)
\]
for each pair of objects $x,y$ satisfying certain compatibility
conditions. 

For a (non-enriched) small category $X$, we may also define a
$G$-grading as a coproduct decomposition
\[
 \Mor_{X}(x,y) = \coprod_{g\in G} \Mor_{X}^g(x,y)
\]
of the set of morphisms for each $x,y$, which satisfies the analogous
compatibility conditions for group graded $k$-linear categories.
This coproduct approach can be extended to categories graded by a small
category, including $k$-linear categories. See, for example,
\cite{Lowen08}. 

Notice, however, that in the case of group graded small (non-enriched)
categories,  
a $G$-grading on a category $X$ can be defined simply as a functor
\[
 p : X \longrightarrow G.
\]
All the necessary compatibility conditions are encoded into the
functoriality of $p$. And this definition of grading is close to the
co-domain in (\ref{prestack_and_prefibered}). It is desirable to
redefine graded categories without referring to coproduct
decompositions in order to find a correct co-domain of the Grothendieck
construction, although the idea of describing a $G$-graded category as 
a $k$-linear functor from $A$ to $k[G]$ fails immediately.

In fact such an approach has already appeared in a classical work
on group graded algebras by Cohen and Montgomery  
\cite{Cohen-Montgomery84}. When translated into the language of
comodules, their observation can be stated as follows.

\begin{lemma}
 \label{Cohen-Montgomery}
 Let $A$ be an algebra over a commutative ring $k$ and $G$ be a
 group. Then there is a one-to-one correspondence between gradings of
 $A$ by $G$ and comodule algebra structure on $A$ over $k[G]$.
\end{lemma}

We pursue this observation and define graded $k$-linear categories as
follows. 

\begin{definition}
 Let $I$ be a small category. An $I$-grading on a $k$-linear category
 $A$ is a structure of comodule category on $A$ over the coalgebra
 category $I\otimes k$ generated by $I$.
\end{definition}

Undefined terminologies appearing in the above ``definition'' will be
explained in \S\ref{preliminaries}. This definition extends immediately
to categories enriched over more general symmetric monoidal
categories. A precise definition and basic properties of graded
categories will be given in \S\ref{graded_category} in the general
context of categories enriched over a symmetric monoidal category.

It turns out this comodule approach of Cohen and Montgomery is
appropriate for generalizing the Grothendieck construction and gradings
for enriched categories. The Grothendieck construction should be
regarded as a process of forming a comodule category from a diagram of
categories. 

As a byproduct, we also obtain a simple characterization of 
categories enriched over a symmetric monoidal category in terms of
comodules. 

\begin{definition}
 Let $\bm{V}$ be a symmetric monoidal category satisfying a certain mild
 condition and $S$ be a set. A
 category enriched over $\bm{V}$ with the set of objects $S$  is a
 monoid object in the monoidal category of $S$-$S$-bicomodules.
\end{definition}

The definition of $S$-$S$-bicomodule and the monoidal structure on the
category of $S$-$S$-bicomodules will be given in
\S\ref{enrichment_by_comodule}, including related definitions.

\subsection{Aim and Scope}

The aim of this article, therefore, is the following:
\begin{itemize}
 \item We define a grading of an enriched category by a small category
       $I$ in terms of comodule structures over a coalgebra category and
       investigate basic properties of the $2$-category of $I$-graded
       categories.   
 \item We extend the definition of the Grothendieck construction to
       enriched categories as a $2$-functor
       \[
       \Gr: \lOplax(I,\categories{\bm{V}}) \longrightarrow
       \lcats{\bm{V}}_I,
       \]
       where $\lOplax(I,\categories{\bm{V}})$ is the
       $2$-category of oplax functors and left transformations from $I$
       to $\bm{V}$-enriched categories and $\lcats{\bm{V}}_I$ is the
       $2$-category of left $I$-graded categories.
 \item We show that $\Gr$ has a right adjoint
       \[
       \lGamma : \lcats{\bm{V}}_I \longrightarrow
       \lOplax(I,\categories{\bm{V}}).
       \]
       We also study right versions. In particular we define $2$-functors
       \begin{eqnarray*}
	\Gr & : & \rLax(I^{\op},\categories{\bm{V}}) \longrightarrow
	 \rcats{\bm{V}}_I \\
	\rGamma & : & \rcats{\bm{V}}_I \longrightarrow
	 \rLax(I^{\op},\categories{\bm{V}}). 
       \end{eqnarray*}
 \item We extend the notion of (pre)fibered and (pre)cofibered
       categories by Grothendieck and graded (pre)fibered categories by
       Lowen \cite{Lowen08} to enriched graded (pre)fibered  and
       (pre)cofibered categories. We define $2$-functors
       \begin{eqnarray*}
	\Gamma_{\cof} & : & \category{Precofibered}(I)
	 \longrightarrow \lcats{\bm{V}}_I \\
	\Gamma_{\fib} & : & \category{Prefibered}(I)
	 \longrightarrow \rcats{\bm{V}}_I
       \end{eqnarray*}
       and investigate their relations to $\lGamma$ and $\rGamma$.
\end{itemize}

\subsection{Organization}

The paper is written as follows:
\begin{itemize}
 \item We set up a categorical framework in \S\ref{preliminaries}. After
       a brief summary on monoidal categories in
       \S\ref{monoidal_category}, we define and study 
       comonoids and comodules over them in \S\ref{comodule}.  We
       describe the standard definition of enriched categories in
       \S\ref{enriched_category}. $2$-categorical
       notions used in this paper are recalled in \S\ref{2-category}. We
       introduce notions of coalgebra categories and comodule
       categories over a coalgebra category and investigate their
       properties in \S\ref{comodule_category}.

 \item The Grothendieck construction for diagrams of enriched categories
       is defined in \S\ref{definitions}. The construction is given in
       \S\ref{definition_for_oplax_functor}.  We introduce the
       $2$-category of graded categories in 
       \S\ref{graded_category} and the Grothendieck construction is
       extended into a $2$-functor taking values in the $2$-category of  
       graded categories in \S\ref{Gr_as_2-functor}.

 \item \S\ref{comma_category} is devoted to definitions of taking
       ``fibers'' of graded enriched categories. We define three ways to
       take fibers over an object in the grading category in 
       \S\ref{enriched_fiber}. The notions of prefibered, precofibered,
       fibered, and cofibered categories are extended to enriched
       categories in \S\ref{fibered_category}. 

 \item In \S\ref{adjunctions}, we prove several properties of the
       Grothendieck construction.  We show that the comma category
       constructions defined in the previous section can be regarded as
       an extension of the smash product construction and prove that it
       is right adjoint to the Grothendieck construction. A main result
       is Theorem \ref{Gr_and_Gamma_are_adjoint}. We also study
       relations between $\lGamma$ and $\Gamma_{\cof}$ in
       \S\ref{fibered}. 

% \item We make a couple of remarks in \S\ref{remarks}.

 \item We include four appendices. We define categories enriched over
       a symmetric monoidal category in terms of comodules in
       \S\ref{enrichment_by_comodule}. \S\ref{left_adjoint_to_diagonal}
       is an enriched version of a well-known fact that the Grothendieck
       construction can be regarded as a $2$-colimit. We specialize the
       contents of \S\ref{definitions} to the case of categories
       enriched over a symmetric monoidal category whose unit object is
       terminal and the tensor product is given by the product in
       \S\ref{product_type}. Constructions corresponding to those in
       \S\ref{adjunctions} are described in \S\ref{comma_for_enriched}.       
\end{itemize}

\subsection*{Acknowledgements}

This work was initiated by a talk delivered by Hideto Asashiba at
Shinshu University in March 2009. The idea of extending the
Grothendieck construction and the notion of graded categories arose
during the e-mail correspondences with him thereafter. The author is
grateful for his intriguing talk, interests, and helpful comments. In
particular, the use of comodules, which is central in this paper, was
suggested by him. 

\section{Categorical Preliminaries}
\label{preliminaries}

\subsection{Monoidal Categories}
\label{monoidal_category}

Throughout this paper, we fix a symmetric monoidal category $\bm{V}$ and
work in $\bm{V}$. Before we begin our discussion, let us briefly
summarize basic definitions and properties of monoidal categories. A
convenient summary is the appendix of \cite{0711.1402}, for example.

\begin{definition}
 \label{monoidal_definition}
 Let $\bm{V}$ be a category. A monoidal structure on $\bm{V}$
 is a collection of the following data:
 \begin{description}
  \item[($1$)] an object $1$ of $\bm{V}$,
  \item[($\bullet\otimes\bullet$)] a (covariant) functor
	     \[
	     \otimes : \bm{V}\times\bm{V} \longrightarrow \bm{V},
	     \]
 \item[($\bullet\otimes\bullet\otimes\bullet$)] a natural isomorphism
	     \[
	      a_{A,B,C} : A\otimes(B\otimes C) \longrightarrow
	      (A\otimes B)\otimes C
	     \]
	    called the associator,
  \item[($1\otimes\bullet$)] a natural isomorphism
	     \[
	      \ell_A : 1\otimes A \longrightarrow A,
	     \]
  \item[($\bullet\otimes 1$)] a natural isomorphism
	     \[
	      r_A : A\otimes 1 \longrightarrow A.
	     \]
 \end{description}
 They are required to satisfy the following three conditions:
 \begin{description}
  \item[($\bullet\otimes\bullet\otimes\bullet\otimes\bullet$)] For any
	     objects $A, B, C, D$ in $\bm{V}$, the following diagram is
	     commutative
	     \[
	      \begin{diagram}
	       \node{} \node{A\otimes(B\otimes(C\otimes D))}
	       \arrow{se,t}{1\otimes a}
	       \arrow{sw,t}{a} \node{} \\ 
	       \node{(A\otimes B)\otimes(C\otimes D)} \arrow{s,l}{a} \node{}
	       \node{A\otimes((B\otimes C)\otimes D)} \arrow{s,r}{a} \\
	       \node{((A\otimes B)\otimes C)\otimes D} \node{}
	       \node{(A\otimes(B\otimes C))\otimes D}
	       \arrow[2]{w,t}{a\otimes 1}
	      \end{diagram}
	     \]
  \item[($\bullet\otimes\bullet\otimes 1$)] For any objects $A, B$ in
	     $\bm{V}$, the following diagram is commutative 
	     \[
	      \begin{diagram}
	       \node{A\otimes(B\otimes 1)} \arrow{e,t}{1\otimes r}
	       \arrow{s,l}{a} 
	       \node{A\otimes B} \\ 
	       \node{(A\otimes B)\otimes 1} \arrow{ne,b}{r}
	      \end{diagram}
	     \]
  \item[($1\otimes 1$)] $\ell_1 = r_1 : 1\otimes 1 \longrightarrow 1$.
 \end{description}
\end{definition}

\begin{remark}
 The commutativity of the following diagrams follows from the above
 axioms, according to \cite{Kelly64-1}.
 \begin{description}
  \item[($1\otimes\bullet\otimes\bullet$)] For any objects $A, B$ in
	     $\bm{V}$, the following diagram is commutative 
	     \[
	      \begin{diagram}
	       \node{1\otimes(A\otimes B)} \arrow{e,t}{\ell}
	       \arrow{s,l}{a} \node{A\otimes B} \\
	       \node{(1\otimes A)\otimes B.} \arrow{ne,b}{\ell}
	      \end{diagram}
	     \]
  \item[($\bullet\otimes 1\otimes\bullet$)] For any objects $A, B$ in
	     $\bm{V}$, the following diagram is commutative 
	     \[
	      \begin{diagram}
	       \node{A\otimes(1\otimes B)} \arrow{e,t}{1\otimes\ell}
	       \arrow{s,r}{a} 
	       \node{A\otimes B} \\ 
	       \node{(A\otimes 1)\otimes B.} \arrow{ne,b}{r\otimes 1}
	      \end{diagram}
	     \]
 \end{description}
\end{remark}

Our primary examples are the category $\Sets$ of sets, the category
$\lMod{k}$ of $k$-modules over a commutative ring $k$, the category
$C(k)$ of chain complexes over $k$, the category $\Spaces$ of
topological spaces, the category $\Sets^{\Delta^{\op}}$ of simplicial
sets, and the category $\Categories$ of small categories. We may also
consider one of models of modern category $\Spectra$ of spectra (in the
sense of algebraic topology)
\cite{EKMM,Hovey-Shipley-Smith00,Mandell-May02}.  All these 
categories have coproducts. 

\begin{definition}
 \label{monoidal_category_with_coproducts}
 We say a monoidal category $(\bm{V},\otimes, 1)$ is a monoidal category
 with coproducts if $\bm{V}$ is closed under coproducts and $\otimes$
 distributes with respect to coproducts. 
\end{definition}

$\lMod{k}$ and $C(k)$ are different from $\Sets$, $\Spaces$, and
$\Categories$.  

\begin{definition}
 \label{additive_monoidal_category}
 A monoidal category $(\bm{V},\otimes, 1)$ is an additive monoidal
 category if
 \begin{enumerate}
  \item $\bm{V}$ has a structure of additive category, i.e.\ it has a
	$0$-object, has finite coproducts, the set of
	morphisms $\Mor_{\bm{V}}(A,B)$ from $A$ to $B$ has a structure
	of Abelian group, and the composition of morphisms is bilinear.
  \item $\bm{V}$ is a monoidal category with finite coproducts,
  \item the monoidal structure
	\[
	 \otimes : \Mor_{\bm{V}}(A,B)\times \Mor_{\bm{V}}(C,D)
	\longrightarrow \Mor_{\bm{V}}(A\otimes C,B\otimes D)
	\]
	is bilinear.
 \end{enumerate}

 An additive monoidal category $\bm{V}$ is called an Abelian monoidal
 category if it is an Abelian category.
\end{definition}

\begin{example}
 $\lMod{k}$ and $C(k)$ are Abelian monoidal categories.
\end{example}

The remaining examples but $\Spectra$ have the following structure.

\begin{definition}
 \label{product_type_definition}
 We say a monoidal category is of product type if the unit object $1$ is
 terminal and $\otimes$ is given by the categorical product.
\end{definition}

\begin{example}
 $\Sets$, $\Spaces$, $\Sets^{\Delta^{\op}}$, and $\Categories$ are of
 product type.
\end{example}

The category $\Spectra$ of spectra is neither Abelian nor of product
type. But it can be fit into our work by replacing isomorphisms by weak
equivalences. See Remark \ref{decompostion_for_stable_model_category}.

All of our monoidal categories are symmetric.

\begin{definition}
 Let $\bm{V}$ be a monoidal category. A switching operation on $\bm{V}$
 is a natural transformation
 \[
 t_{A,B} : A\otimes B \longrightarrow B\otimes A,
 \]
 i.e.\ the diagram
 \[
  \begin{diagram}
   \node{A\otimes B} \arrow{e,t}{t_{A,B}} \arrow{s,l}{f\otimes g}
   \node{B\otimes A} \arrow{s,r}{g\otimes f} \\
   \node{C\otimes D} \arrow{e,t}{t_{C,D}} \node{D\otimes C}
  \end{diagram}
 \]
 is commutative for any morphisms $f$ and $g$.
\end{definition}

\begin{definition}
 A monoidal category $\bm{V}$ with a switching operation $t$ is called a
 symmetric monoidal category if
 \begin{enumerate}
  \item $t_{A,B}\circ t_{B,A}=1$ for any objects $A, B$.
  \item The following diagram is commutative for any
	 triple of objects $A,B, C$:
	 \[
	 \begin{diagram}
	  \node{} \node{A\otimes (B\otimes C)}
	  \arrow{sw,t}{t_{A,B\otimes C}} \arrow{se,t}{a_{A,B,C}}
	  \node{} \\ 
	  \node{(B\otimes C)\otimes A} \node{} 
	  \node{(A\otimes B)\otimes C}
	  \arrow{s,r}{t_{A, B}\otimes 1} \\
	  \node{B\otimes (C\otimes A)} \arrow{n,l}{a_{B,C,A}}
	  \node{} \node{(B\otimes A)\otimes C}
	  \arrow{sw,b}{a_{B,A,C}} \\ 
	  \node{} \node{B\otimes (A\otimes C)} \arrow{nw,b}{1\otimes
	  t_{A,C}} 
	  \node{} 
	 \end{diagram}
	 \]
 \end{enumerate}
\end{definition}

We often have a left adjoint to the ``underlying set''
functor. 

\begin{definition}
 Let $(\bm{V},\otimes, 1)$ be a monoidal category. The functor
 \[
  \Mor_{\bm{V}}(1,-) : \bm{V} \longrightarrow \Sets
 \]
 is called the ``underlying set'' functor.
\end{definition}

For example, the forgetful functor from $\lMod{k}$ to $\Sets$ can be
expressed as $\Mor_{\lMod{k}}(k,-)$ and $k$ is the unit object in
$\lMod{k}$. In this case we have a left adjoint 
\[
 (-)\otimes k : \Sets \longrightarrow \lMod{k}
\]
which assigns to each set $S$ the free module generated by $S$.

We require this property as a fundamental assumption on our
symmetric monoidal category $\bm{V}$.

\begin{assumption}
 \label{condition_on_V}
 In the rest of this paper, we assume $\bm{V}$ is a symmetric monoidal
 category with coproducts satisfying the following conditions.
 \begin{enumerate}
  \item $\bm{V}$ is closed under finite limits.
%  \item The unit object $1$ is terminal.
  \item The ``underlying set'' functor 
	\[
	\Mor_{\bm{V}}(1,-) : \bm{V} \longrightarrow \Sets
	\]
	has a left adjoint
	\[
	(-)\otimes 1 : \Sets \longrightarrow \bm{V}.
	\]
%  \item $(-)\otimes 1$ is a lax monoidal functor. (See Definition
%	\ref{lax_monoidal_functor_definition}.) 
 \end{enumerate}
\end{assumption}

\begin{example}
 The categories $\Sets$, $\Sets^{\Delta^{\op}}$, $\Spaces$,
 $\lMod{k}$, $C(k)$, $\Categories$ and $\Spectra$ all satisfy the above 
 conditions. 
\end{example}

A monoidal category can be regarded as a $2$-category with a single
object. When we talk about functors between monoidal categories, we need
a notion of lax monoidal functor.

\begin{definition}
 \label{lax_monoidal_functor_definition}
 Let $(\bm{C},\otimes, 1_{\bm{C}})$ and $(\bm{D},\otimes, 1_{\bm{D}})$
 be monoidal categories. 
 \begin{enumerate}
  \item A lax monoidal functor\index{lax monoidal functor}
	\index{monoidal functor!lax ---} is a triple
	$(F,\mu, \eta)$, where 
  \begin{itemize}
   \item $F : \bm{C} \to \bm{D}$ is a functor,
  \item $\mu_{x,y} : F(x)\otimes F(y) \to F(x\otimes y)$ is a
	natural transformation,
  \item $\eta : 1_{\bm{D}} \to F(1_{\bm{C}})$ is a morphism in
	$\bm{D}$, 
  \end{itemize}
 which make the following diagrams commutative:
 \begin{enumerate}
  \item \[
	\begin{diagram}
	 \node{} \node{(F(x)\otimes F(y))\otimes F(z)}
	 \arrow{sw,t}{\mu\otimes 1} \arrow{se,t}{a}
	 \node{} \\ 
	 \node{F(x\otimes y)\otimes F(z)} \arrow{s,l}{\mu} \node{} 
	 \node{F(x)\otimes (F(y)\otimes F(z))}
	 \arrow{s,r}{1\otimes \mu} \\
	 \node{F((x\otimes y)\otimes z)} \arrow{se,b}{F(a)} \node{}
	 \node{F(x)\otimes F(y\otimes z)} 
	 \arrow{sw,b}{\mu} \\ 
	 \node{} \node{F(x\otimes (y\otimes z))} \node{}
	\end{diagram}
	\]
  \item \[
	\begin{diagram}
	 \node{1_{\bm{D}}\otimes F(x)} \arrow{s} \arrow{e}
	 \node{F(1_{\bm{C}})\otimes F(x)} \arrow{s} \\
	 \node{F(x)} \node{F(1_{\bm{C}}\otimes x)} \arrow{w}
	\end{diagram}
	\]
  \item \[
	\begin{diagram}
	 \node{F(x)\otimes 1_{\bm{D}}} \arrow{e} \arrow{s}
	 \node{F(x)\otimes F(1_{\bm{C}})} \arrow{s} \\
	 \node{F(x)} \node{F(x\otimes 1_{\bm{C}})} \arrow{w}
	\end{diagram}
	\]
 \end{enumerate}
  \item An oplax monoidal functor is a
	triple $(F,\mu,\eta)$, where
	\begin{enumerate}
	 \item $F : \bm{C} \to \bm{D}$ is a functor;
	 \item $\mu_{x,y} : F(x\otimes y) \to F(x)\otimes F(y)$ is a
	       natural transformation;
	 \item $\eta : F(1_{\bm{C}}) \to 1_{\bm{D}}$ is a morphism in
	       $\bm{D}$, 
	\end{enumerate}
	which make the following diagrams are commutative:
	\begin{enumerate}
	 \item \[
		\begin{diagram}
		 \node{} \node{(F(x)\otimes F(y))\otimes F(z)}		 
		 \arrow{se,t}{a} \node{} \\ 
		 \node{F(x\otimes y)\otimes F(z)}
		 \arrow{ne,t}{\mu\otimes 1} \node{} 
		 \node{F(x)\otimes (F(y)\otimes F(z))} 
		 \\
		 \node{F((x\otimes y)\otimes z)}
		 \arrow{n,l}{\mu} \arrow{se,b}{F(a)} \node{}
		 \node{F(x)\otimes F(y\otimes z)} \arrow{n,r}{1\otimes \mu} 
		 \\ 
		 \node{} \node{F(x\otimes (y\otimes z))}
		 \arrow{ne,b}{\mu} \node{}
		\end{diagram}
	       \]
	 \item \[
		\begin{diagram}
		 \node{1_{\bm{D}}\otimes F(x)} \arrow{s} 
		 \node{F(1_{\bm{C}})\otimes F(x)} \arrow{w} \\
		 \node{F(x)} \node{F(1_{\bm{C}}\otimes x)} \arrow{w} \arrow{n} 
		\end{diagram}
	       \]
	 \item \[
		\begin{diagram}
		 \node{F(x)\otimes 1_{\bm{D}}} \arrow{s}
		 \node{F(x)\otimes F(1_{\bm{C}})} \arrow{w} \\
		 \node{F(x)} \node{F(x\otimes 1_{\bm{C}})} \arrow{w} \arrow{n} 
		\end{diagram}
	       \]
	\end{enumerate}
%  \item A strong monoidal functor\index{strong monoidal functor}
%	\index{monoidal functor!strong ---} is a
%	lax monoidal functor 
%	$(F,\mu,\eta)$ such that $\mu_{x,y}$ and $\eta$ are isomorphisms
%	for all $x,y \in \bm{C}_0$.
 \end{enumerate}

\end{definition}

\subsection{Comonoids, Comodules, and Coproduct Decompositions}
\label{comodule}

In this section, we introduce and study the notion of comonoids in a
symmetric monoidal category and comodules over them, which play central
roles in this paper.

Let $\bm{V}$ be a symmetric monoidal category satisfying the conditions
in Assumption \ref{condition_on_V}.

\begin{definition}
 A comonoid object in $\bm{V}$ is an object $C$ equipped with morphisms
 \begin{eqnarray*}
  \Delta & : & C \longrightarrow C\otimes C \\
  \varepsilon & : & C \longrightarrow 1
 \end{eqnarray*}
 making the following diagrams commutative
 \[
 \begin{diagram}
  \node{} \node{C} \arrow{s,r}{\Delta}
  \arrow{sw,t}{\cong} \arrow{se,t}{\cong}
  \\ 
  \node{1\otimes C}
  \node{C\otimes  C}
  \arrow{w,t}{\varepsilon\otimes 1}
  \arrow{e,t}{1\otimes \varepsilon}
  \node{C\otimes 1}
 \end{diagram}
 \]
 \[
  \begin{diagram}
   \node{C} \arrow{e,t}{\Delta} \arrow[2]{s,l}{\Delta}
   \node{C\otimes C} 
   \arrow{s,r}{\Delta\otimes 1} \\ 
   \node{} \node{(C\otimes C)\otimes C}
   \\ 
   \node{C\otimes C} \arrow{e,t}{1\otimes \Delta}
   \node{C\otimes(C\otimes C),} \arrow{n,r}{a} 
  \end{diagram}
 \]
 where $a$ is the associator in $\bm{V}$.

 The morphisms $\Delta$ and $\varepsilon$ are called the coproduct and
 the counit of $C$, respectively.

 Morphisms of comonoids are defined in an obvious way. The category of
 comonoids in $\bm{V}$ is denoted by $\category{Comonoids}(\bm{V})$.
\end{definition}

\begin{example}
 \label{everybody_is_comonoid}
 When $1$ is terminal and $\otimes$ is the product, i.e.\ when $\bm{V}$
 is of product type (Definition \ref{product_type_definition}), any
 object in $\bm{V}$ has a canonical comonoid structure. 
\end{example}

A typical example of such a monoidal category is the category of
sets. Under our assumption on $\bm{V}$ (Assumption
\ref{condition_on_V}), we can compare the category of 
sets and $\bm{V}$ by the ``free object'' functor 
\[
 (-)\otimes 1 : \Sets \longrightarrow \bm{V},
\]
which is left adjoint to the underlying set functor.

\begin{lemma}
 \label{otimes1_is_lax_monoidal}
 $(-)\otimes 1$ is an oplax monoidal functor. 
\end{lemma}

\begin{proof}
 We need to define a natural transformation
 \[
  \theta : ((-)\times (-))\otimes 1 \Longrightarrow ((-)\otimes 1)\otimes
 ((-)\otimes 1).
 \]
 For sets $S, T$, consider the composition
 \begin{eqnarray*}
  S\times T & \rarrow{} & \Mor_{\bm{V}}(1,S\otimes 1)\times
   \Mor_{\bm{V}}(1, T\otimes 1) \\
  & \rarrow{\otimes} & \Mor_{\bm{V}}(1\otimes 1, (S\otimes 1)\otimes
   (T\otimes 1)) \\
  & \rarrow{\cong} & \Mor_{\bm{V}}(1,(S\otimes 1)\otimes (T\otimes 1)).
 \end{eqnarray*}
 By taking the adjoint, we obtain
 \[
  \theta_{S,T} : (S\times T)\otimes 1 \longrightarrow (S\otimes
 1)\otimes (T\otimes 1). 
 \]
 By taking the left adjoint to the map
 \[
  1_{1} : \{\ast\} \longrightarrow \Mor_{\bm{V}}(1,1)
 \]
 representing the identity morphism on the unit object, we obtain
 \[
  \eta : \{\ast\} \otimes 1 \longrightarrow 1.
 \]
 It is straightforward to check that $\theta$ and $\eta$ satisfy the
 condition for an oplax monoidal functor.
% makes the following diagrams commutative.
% \[
% \begin{diagram}
%  \node{} \node{((S\otimes 1)\otimes (T\otimes 1))\otimes (U\otimes 1)}
%  \arrow{sw,t}{\theta\otimes 1} \arrow{se,t}{a}
%  \node{} \\ 
%  \node{((S\times T)\otimes 1)\otimes (U\otimes 1)} \arrow{s,l}{\theta}
%  \node{}  
%  \node{(S\otimes 1)\otimes ((T\otimes 1)\otimes (U\otimes 1))}
%  \arrow{s,r}{1\otimes \theta} \\
%  \node{((S\times T)\times U)\otimes 1} \arrow{se,b}{a\otimes 1} \node{}
%  \node{(S\otimes 1)\otimes (T\otimes U)\otimes 1} 
%  \arrow{sw,b}{\theta} \\ 
%  \node{} \node{(S\times (T\times U))\times 1} \node{}
% \end{diagram}
% \]
% \[
% \begin{diagram}
%  \node{1_{\bm{V}}\otimes (S\otimes 1)} \arrow{s} \arrow{e}
%  \node{(\{\ast\}\otimes 1)\otimes (S\otimes 1)} \arrow{s} \\
%  \node{S\otimes 1} \node{(\{\ast\}\times S)\otimes 1} \arrow{w}
% \end{diagram}
% \]
% \[
% \begin{diagram}
%  \node{(S\otimes 1)\otimes 1_{\bm{V}}} \arrow{e} \arrow{s}
%  \node{(S\otimes 1)\otimes (\{\ast\}\otimes 1)} \arrow{s} \\
%  \node{S\otimes 1} \node{(S\times \{\ast\})\otimes 1.} \arrow{w}
% \end{diagram}
% \]
\end{proof}

\begin{example}
 \label{coproduct_on_free_object}
 For any set $S$, $S\otimes 1$ has a comonoid structure defined as
 follows. We have
 \[
  \theta : (S\times S)\otimes 1 \longrightarrow (S\otimes 1)\otimes
 (S\otimes 1)
 \]
 by the above Lemma. By composing with 
 \[
  \Delta\otimes 1 : S\otimes 1 \longrightarrow (S\times S)\otimes 1,
 \]
 we obtain a coproduct
 \[
  \Delta : S\otimes 1 \longrightarrow (S\otimes 1)\otimes (S\otimes 1).
 \]

 The counit is defined by
 \[
  \varepsilon : S\otimes 1 \longrightarrow \{\ast\}\otimes 1
 \rarrow{\eta} 1.
 \]
\end{example}

\begin{definition}
 Let $C$ be a comonoid in $\bm{V}$. A right coaction of $C$ on an object
 $M$ in $\bm{V}$ is a morphism
 \[
  \mu : M \longrightarrow M\otimes C
 \]
 making the following diagrams commutative
 \[
 \begin{diagram}
  \node{} \node{M} \arrow{s,r}{\mu}
  \arrow{sw,t}{\cong} \\ 
  \node{M\otimes 1}
  \node{M\otimes C}
  \arrow{w,t}{1\otimes \varepsilon,}
 \end{diagram}
 \]
 \[
  \begin{diagram}
   \node{M} \arrow{e,t}{\mu} \arrow[2]{s,l}{\mu} \node{M\otimes C}
   \arrow{s,r}{1\otimes \Delta} \\ 
   \node{} \node{M\otimes (C\otimes C)} \arrow{s,r}{a}
   \\ 
   \node{M\otimes C} \arrow{e,t}{\mu\otimes 1}
   \node{(M\otimes C)\otimes C.} 
  \end{diagram}
 \]

 An object $M$ equipped with a right coaction of $C$ is called a right
 $C$-comodule. Morphisms of right comodules are defined in an obvious
 way.

 Left coactions are defined analogously. An object $M$ equipped with both
 right coaction $\mu^R$ and left coaction $\mu^L$ is called a
 $C$-$C$-bimodule if the following diagram is commutative.
 \[
  \begin{diagram}
   \node{M} \arrow[2]{e,t}{\mu^L} \arrow{s,l}{\mu^R} \node{}
   \node{C\otimes M} \arrow{s,r}{1\otimes \mu^L} \\
   \node{M\otimes C} \arrow{se,b}{\mu^L\otimes 1} \node{}
   \node{C\otimes(M\otimes C)} \arrow{sw,b}{a} \\
   \node{} \node{(C\otimes M)\otimes C} \node{}
  \end{diagram}
 \]

 The categories of right $C$-comodules, left $C$-comodules, and
 $C$-$C$-bimodules are denoted by $\rComod{C}$, $\lComod{C}$, and
 $\biComod{C}{C}$, respectively. 
\end{definition}

\begin{example}
 \label{coaction_in_product_type}
 Suppose $\bm{V}$ is of product type. By
 Example \ref{everybody_is_comonoid}, any object $C$ can be regarded as
 a comonoid and any morphism
 \[
  f : M \longrightarrow C
 \]
 is a morphism of comonoids. Define
 \[
  \mu : M \rarrow{\Delta} M\otimes M \rarrow{1\otimes f} M\otimes C,
 \]
 then $\mu$ is a coaction of $C$ on $M$.

 Conversely any coaction $\mu$ is determined by the composition
 \[
  \pi : M \rarrow{\mu} M\otimes C \rarrow{\pr_2} C,
 \]
 since the composition of $\mu$ and the first projection is always the
 identity morphism by the counit condition. 
\end{example}

The following is a direct extension of an observation (Lemma
\ref{Cohen-Montgomery}) by Cohen and Montgomery in
\cite{Cohen-Montgomery84}. 

\begin{lemma}
 \label{decomposition_and_comodule}
 Suppose $\bm{V}=\lMod{k}$ for a commutative ring $k$. Let $M$ be an
 object of $\bm{V}$ and $S$ be a set. Then there is a one-to-one
 correspondence  between right $S\otimes 1$-comodule structures on $M$
 and coproduct decompositions on $M$ indexed by $S$
 \[
  M \cong \bigoplus_{s\in S} M^s.
 \]
\end{lemma}

\begin{proof}
 Suppose we have a coproduct decomposition
 \[
  M \cong \bigoplus_{s\in S} M^s.
 \]

 Define
 \[
 \mu : M \longrightarrow 
 M \otimes(S\otimes 1)
 \]
 on each component $M^s$ by the composition
 \[
  M^s \larrow{\cong} M^s \otimes 1 \rarrow{1\otimes s}
 M^s\otimes (S\otimes 1).
 \]

 The coassociativity follows from the commutativity of the following
 diagram  
 \[
  \begin{diagram}
   \node{M^s} \arrow{e} \arrow{s} \node{M^s\otimes 1} \arrow{s}
   \arrow{e,t}{1\otimes s} 
    \node{M^s\otimes (S\otimes 1)} \arrow{s} \\ 
    \node{M^s\otimes 1} \arrow{e} \arrow{s}
    \node{(M^s\otimes 1)\otimes 1} \arrow{s,r}{(1\otimes s)\otimes 1} 
    \arrow{e,t}{(1\otimes 1)\otimes s}
    \node{(M^s\otimes 1)\otimes (S\otimes 1)}
    \arrow{s,r}{(1\otimes s)\otimes 1} \\ 
    \node{M^s\otimes (1\otimes 1)} \arrow{se,t}{1\otimes (s\otimes 1)}
    \node{(M^s\otimes (S\otimes 1))\otimes 1}
    \arrow{e,t}{(1\otimes 1)\otimes s} \arrow{s} 
   \node{(M^s\otimes (S\otimes 1))\otimes (S\otimes 1)} \arrow{s} \\   
   \node{} 
   \node{M^s\otimes ((S\otimes 1)\otimes 1)}
    \arrow{e,t}{1\otimes (1\otimes s)}
   \node{M^s\otimes ((S\otimes 1)\otimes (S\otimes 1)).} 
  \end{diagram}
 \]
 The counitality is obvious.

 Conversely suppose we have a comodule structure
 \[
  \mu : M \longrightarrow M\otimes (S\otimes 1).
 \]
 For each $s\in S$, define a morphism 
 \[
  p_s : M \longrightarrow M
 \]
 by the composition
 \[
  \begin{diagram}
   \node{M} \arrow[2]{s,=} \arrow{e,t}{\mu}
   \node{M\otimes (S\otimes 1)} \arrow{e,t}{\cong} 
   \node{\bigoplus_{t\in S} M\otimes(t\otimes 1)}
   \arrow{s,r}{\text{projection}} \\ 
   \node{} \node{} \node{M\otimes (s\otimes 1)} \arrow{s,r}{\cong} \\
   \node{M} \arrow[2]{e,t}{p_s} \node{} \node{M.}
  \end{diagram}
 \]
 By the coassociativity of $\mu$ and the fact that the coproduct
 on $S\otimes 1$ is given by the diagonal on $S$, we have
 \[
  p_s\circ p_{t} = \begin{cases}
		     p_s, & s=t \\
		     0, & s\neq t.
		    \end{cases}
 \]
 By the definition of $p_s$ and the counitality of $\mu$, 
 \[
  \sum_{s\in S} p_s = 1_{M}
 \]
 and we have a coproduct decomposition
 \[
  M \cong \bigoplus_{s\in S} \Ima p_s.
 \] 
\end{proof}

The above discussion is based on the fact that $\lMod{k}$ is Abelian.
Even when $\bm{V}$ is not Abelian, we often obtain a coproduct
decomposition from a comodule structure. Suppose $\bm{V}$ is a product
type monoidal category. In such a monoidal category, a
coaction 
\[
 \mu : M \longrightarrow M\otimes (S\otimes 1)
\]
defines and is defined by a morphism
\[
 \pi : M \rarrow{\mu} M\otimes (S\otimes 1) \longrightarrow S\otimes 1 
\]
by Example \ref{coaction_in_product_type}.

\begin{example}
 \label{decomposition_in_product_type}
 Let $\bm{V}$ be the category $\Spaces$ of topological spaces. This is a
 product type monoidal category and thus any object $C$ is a comonoid
 and a coaction of $C$ on another object $M$ determines and is
 determined by a morphism
 \[
  p : M \longrightarrow C
 \]
 When $C$ has a discrete topology, i.e.\ $C=S\otimes 1$ for a set $S$,
 such a continuous map induces a coproduct decomposition
 \[
  M \cong \coprod_{s\in S} M^s
 \]
 by $M^{s} = p^{-1}(s)$.
 We have analogous decompositions when $\bm{V}$ is the category of sets,
 of simplicial sets, or of small categories. 
\end{example}

\begin{remark}
 \label{decompostion_for_stable_model_category}
 When $\bm{V}$ is one of models of symmetric monoidal category of
 spectra, we do not have such a correspondence between comodule
 structures and coproduct decompositions. However, we can obtain
 analogous coproduct decomposition from a comodule structure by
 replacing isomorphisms by weak equivalences.
\end{remark}

The following operation plays an important role.

\begin{definition}
 \label{cotensor_product}
 Let $C$ be a comonoid in $\bm{V}$ and $M$ and $N$ be a right and a left
 $C$-comodules. Define the cotensor product $M\Box_C N$ of $M$ and $N$
 over $C$ by the following equalizer diagram
 \[
 \xymatrix{%
  M\Box_C N \ar[r] & M\otimes N \ar@<1ex>[r]^{\mu_M\otimes 1_N}
 \ar@<-1ex>[r]_{1_M\otimes\mu_N} & M\otimes C\otimes N. 
 }
 \]
\end{definition}

We need the following condition on $C$.

\begin{definition}
 An object $C$ in $\bm{V}$ is called flat if $C\otimes(-)$ preserves
 equalizers. 
\end{definition}

\begin{remark}
 Since $\bm{V}$ is symmetric, if $C$ is flat, $(-)\otimes C$ also
 preserves equalizers.

 The condition that $C\otimes(-)$ preserves equalizers
 is satisfied when $\bm{V}$ is $\Sets$, $\Spaces$,
 $\Sets^{\Delta^{\op}}$ for any $C$ and when $\bm{V}$ is $\lMod{k}$
 and $C$ is flat. In particular, when $\bm{V}=\lMod{k}$ and
 $C=S\otimes 1$, the free $k$-module generated by a set $S$, the
 condition is satisfied.
\end{remark}

\begin{lemma}
 \label{monoidal_category_under_cotensor}
 Let $C$ be a flat comonoid object in $\bm{V}$. For $C$-$C$-bicomodules
 $M$ and $N$, $M\Box_C N$ has a structure of
 $C$-$C$-bicomodule under which the canonical morphism
 \[
  M\Box_C N \rarrow{} M\otimes N
 \]
 is a morphism of bicomodules. Furthermore $\Box_C$ defines a 
 monoidal structure on $\biComod{C}{C}$. The unit object is $C$.
\end{lemma}

\begin{proof}
 For $C$-$C$-bicomodules
 \begin{eqnarray*}
  \mu_M^L & : & M \longrightarrow C\otimes M, \\
  \mu_M^R & : & M \longrightarrow M\otimes C, \\
  \mu_N^L & : & N \longrightarrow C\otimes N, \\
  \mu_N^R & : & N \longrightarrow N\otimes C,
 \end{eqnarray*}
 we need to define a bimodule structure on $M\Box_C N$.
 
 Since $C\otimes(-)$ preserves equalizers, we can define a
 left coaction 
 \[
  \mu_M^L\Box_C 1 : M\Box_C N \longrightarrow C\otimes (M\Box_C N)
 \]
 by the following diagram
 \[
 \xymatrix{%
  M\Box_C N \ar[r] \ar@{..>}[d] & M\otimes N \ar[d]_{\mu_M^L\otimes 1_N}
 \ar@<1ex>[r]^{\mu_M^R\otimes 1_N} 
 \ar@<-1ex>[r]_{1_M\otimes\mu^L_N} & M\otimes C\otimes N
 \ar[d]^{\mu_M^L\otimes 1_C\otimes 1_N} \\ 
  C\otimes (M\Box_C N) \ar[r] & C\otimes M\otimes N
 \ar@<1ex>[r]^{1\otimes \mu_M^R\otimes 1_N} 
 \ar@<-1ex>[r]_{1\otimes 1_M\otimes\mu^L_N} & C\otimes M\otimes C\otimes N. 
 }
 \]
 The right comodule structure is defined analogously.

 The associator of $\otimes$ defines an associator for $\Box_{C}$ and
 $(\biComod{C}{C},\Box_{C},C)$ becomes a monoidal category.
\end{proof}

We also need to compare comodules over different comonoids.

\begin{definition}
 \label{morphism_of_comodules}
 Let $C$ and $D$ be comonoids in $\bm{V}$ and $M$ and $N$ be right
 comodules over $C$ and $D$, respectively. A morphism of right comodules
 from $M$ to $N$ 
 \[
  f : M \longrightarrow N
 \]
 is a pair $f=(f_0,f_1)$ morphisms
 \begin{eqnarray*}
  f_0 & : & C \longrightarrow D \\
  f_1 & : & M \longrightarrow M,
 \end{eqnarray*}
 where $f_0$ is a morphism of comonoids and $f_1$ makes the following
 diagram commutative
 \[
  \begin{diagram}
   \node{M} \arrow{s,l}{f_1} \arrow{e} \node{M\otimes C}
   \arrow{s,r}{f_1\otimes f_0} \\ 
   \node{N} \arrow{e} \node{N\otimes D.}
  \end{diagram}
 \]
 Morphisms of left comodules and bicomodules are defined analogously.
\end{definition}

\subsection{Enriched Categories}
\label{enriched_category}

In this section, we collect the standard definitions concerning
categories enriched over a symmetric monoidal
category. We can use the language developed in
\S\ref{comodule} to define (small) enriched categories and related
notions in a very compact way, as is shown in
\S\ref{enrichment_by_comodule}. We have chosen to use the traditional
definitions for the convenience of the reader. Our reference is Kelly's  
book \cite{KellyEnrichedCategory}. The reader is encouraged to compare
definitions in this section and the corresponding definitions in
\S\ref{enrichment_by_comodule}. 

\begin{definition}
 A category enriched over $\bm{V}$, or simply a $\bm{V}$-category $A$
 consists of
 \begin{itemize}
  \item a class of objects $A_0$;
  \item for two objects $a, b$ in $A$, an object
	${A}(a,b)$ in $\bm{V}$;
  \item for three objects $a, b, c$ in $A$, a morphism
	\[
	 \circ : {A}(b,c) \otimes {A}(a,b)
	\longrightarrow {A}(a,c)
	\]
	in $\bm{V}$;
  \item for an object $a$ in $A$, a morphism in $\bm{V}$
	\[
	 1_a : 1 \longrightarrow {A}(a,a)
	\]
 \end{itemize}
  satisfying the following conditions:
 \begin{enumerate}
  \item for any objects $a, b, c, d$, the following diagram is
	commutative 
	\[
	\begin{diagram}
	 \node{(A(c,d)\otimes
	 A(b,c))\otimes A(a, b)}
	 \arrow{s,l}{\circ\otimes 1}
	 \arrow[2]{e,t}{a} \node{}
	 \node{A(c,d)\otimes(A(b,c)\otimes A(a,b))}
	 \arrow{s,r}{1\otimes\circ} \\  
	 \node{A(b,d)\otimes A(a,b)}
	 \arrow{se,b}{\circ} 
	 \node{} \node{A(c,d)\otimes A(a,c)}
	 \arrow{sw,b}{\circ} \\
	 \node{} \node{A(a,d)} \node{}  
	\end{diagram}
	\]

  \item for any objects $a, b$, the following diagram is commutative
	     \[
	      \begin{diagram}
	       \node{A(b,b)\otimes A(a,b)}
	       \arrow{e,t}{\circ} \node{A(a,b)}
	       \node{A(a,b)\otimes A(a,a)}
	       \arrow{w,t}{\circ} \\  
	       \node{1\otimes A(a,b)} \arrow{n} \arrow{ne}
	       \node{} \node{A(a,b)\otimes 1} \arrow{n}
	       \arrow{nw} 
	      \end{diagram}
	     \]
	where $1$ is the unit object in $\bm{V}$.
 \end{enumerate}
\end{definition}

Any enriched category has its underlying category.

\begin{definition}
 Let $A$ be a $\bm{V}$-category. Define an ordinary category
 $\underline{A}$ as follows. Objects are the same as objects in $A$. The 
 set of morphisms from $a \in A_0$ to $b\in A_0$ is defined by
 \[
  \Mor_{\underline{A}}(a,b) = \Mor_{\bm{V}}(1,A(a,b)).
 \]
 The composition is given by
 \begin{eqnarray*}
  \Mor_{\underline{A}}(b,c)\times \Mor_{\underline{A}}(a,b) & = &
   \Mor_{\bm{V}}(1,A(b,c))\times 
   \Mor_{\bm{V}}(1,A(a,b)) \\
  & \longrightarrow & \Mor_{\bm{V}}(1\otimes 1,A(b,c)\otimes
   A(a,b)) \\
  & \larrow{\cong} & \Mor_{\bm{V}}(1,A(b,c)\otimes A(a,b)) \\
  & \rarrow{} & \Mor_{\bm{V}}(1,A(a,c)) \\
  & = & \Mor_{\underline{A}}(a,c).
 \end{eqnarray*}

 When it is obvious from the context, a $\bm{V}$-category and its
 underlying category are denoted by the same symbol.
\end{definition}

\begin{definition}
 Let $A$ be a $\bm{V}$-category. For a morphism in the underlying
 category $f \in \Mor_{\underline{A}}(a,b)$, define morphisms in
 $\bm{V}$
 \begin{eqnarray*}
  f_* & : & A(c,a) \longrightarrow A(c,b), \\
  f^* & : & A(b,c) \longrightarrow A(a,c)
 \end{eqnarray*}
 by the following compositions
 \begin{eqnarray*}
  f_* & : & A(c,a) \cong 1\otimes A(c,a) \rarrow{f\otimes 1}
   A(a,b)\otimes A(c,a) \rarrow{\circ} A(c,b), \\
  f^* & : & A(b,c) \cong A(b,c)\otimes 1 \rarrow{1\otimes f}
   A(b,c)\otimes A(a,b) \rarrow{\circ} A(a,c).
 \end{eqnarray*}
\end{definition}

The following is the standard description of $\bm{V}$-functors.

\begin{definition}
 \label{V-functor}
 Let $A$ and $B$ be small $\bm{V}$-categories. A $\bm{V}$-functor $f$
 from $A$ to $B$ consists of a map
 \[
  f_0 : A_0 \longrightarrow B_0
 \]
 and a morphism
 \[
  f_1 : A(a,b) \longrightarrow B(f(a),f(b))
 \]
 in $\bm{V}$ for each pair $a,b$ of objects in $A$, satisfying the
 following conditions:
 \begin{enumerate}
  \item the following diagram is commutative
	\[
	 \begin{diagram}
	  \node{A(b,c)\otimes {A}(a,b)} \arrow{e} \arrow{s}
	  \node{{A}(a,c)} \arrow{s} \\
	  \node{B(f(b),f(c))\otimes B(f(a),f(b))} \arrow{e}
	  \node{B(f(a),f(c)),} 
	 \end{diagram}
	\]
  \item and the following diagram is commutative
	\[
	 \begin{diagram}
	  \node{1} \arrow{e} \arrow{se} \node{A(a,a)} \arrow{s} \\
	  \node{} \node{B(f(a),f(a)).}
	 \end{diagram}
	\]
 \end{enumerate}
\end{definition}

\begin{definition}
 The category of small $\bm{V}$-categories and $\bm{V}$-functors is
 denoted by $\categories{\bm{V}}$.
\end{definition}

The monoidal structure on $\categories{\bm{V}}$ (Lemma
\ref{monoidal_category_of_V-categories}) can be described as follows. 

\begin{lemma}
 \label{V-categories_is_monoidal}
 For small $\bm{V}$-categories $A$ and $B$, define a $\bm{V}$-category
 $A\otimes B$ by
 \[
  (A\otimes B)_0 = A_0\times B_0
 \]
 and
 \[
  (A\otimes B)((a,b),(a',b')) = A(a,a')\otimes B(b,b').
 \]
 Define a $\bm{V}$-category $1$ with a single object $\ast$ by
 $1(*,*)=1$. 

 If $\bm{V}$ is symmetric monoidal, $(\categories{\bm{V}},\otimes, 1)$
 forms a symmetric monoidal category whose associator is
 defined by that of $\bm{V}$.  
\end{lemma}

\begin{example}
 When $\bm{V}$ is the category of $k$-modules for a commutative ring
 $k$, $\bm{V}$-categories are called $k$-linear categories and the
 category of small $k$-linear categories is denoted by
 $\categories{k}$.  When $\bm{V}$ is the category of (unbounded) chain
 complexes over a fixed commutative ring $k$, $\bm{V}$-categories are
 called dg (differential graded) categories over $k$.

 When $\bm{V}$ is the category of simplicial sets, topological spaces,
 or spectra (in the sense of algebraic topology), $\bm{V}$-categories
 are called simplicial categories, topological categories, or spectral
 categories. 
\end{example}

\begin{example}
 \label{2-functor_definition}
 Consider the case when $\bm{V}$ is the category $\Categories$ of small
 categories. 
 $\Categories$-enriched categories are usually called (strict)
 $2$-categories. $\Categories$-functors are called $2$-functors. See
 \S\ref{2-category} for more details on $2$-categories.
\end{example}

The following is the standard definition of $\bm{V}$-natural
transformation.

\begin{definition}
 Let
 \[
  f, g : A \longrightarrow B
 \]
 be $\bm{V}$-functors. A $\bm{V}$-natural transformation $\varphi$ from
 $f$ to $g$, denoted by
 \[
  \varphi : f \Longrightarrow g
 \]
 consists of a family of morphisms in $\bm{V}$
 \[
  \varphi(a) : 1 \longrightarrow B(f(a),g(a))
 \]
 indexed by objects in $A$ making the following diagram commutative for
 any pair of objects in $A$:
 \[
  \begin{diagram}
   \node{} \node{A(a,a')} \arrow{se,t}{\ell^{-1}} \arrow{sw,t}{r^{-1}}
   \node{} \\ 
   \node{A(a,a')\otimes 1} \arrow{s,l}{g\otimes\varphi(a)} \node{}
   \node{1\otimes A(a,a')} 
   \arrow{s,r}{\varphi(a')\otimes f} \\ 
   \node{B(g(a),g(a'))\otimes B(f(a),g(a))} \arrow{se,b}{\circ} \node{}
   \node{B(f(a'),g(a'))\otimes B((f(a),f(a')))} \arrow{sw,b}{\circ} \\
   \node{} \node{B(f(a),g(a'))}
  \end{diagram}
 \]
\end{definition}

The composition of $\bm{V}$-natural transformations is defined in an
obvious way.

\begin{definition}
 Let
 \[
  f,g,h : A \longrightarrow B
 \]
 be $\bm{V}$-functors and 
 \[
 \xymatrix{%
  f \ar@{=>}[r]^{\varphi} & g \ar@{=>}[r]^{\psi} & h}
 \]
 be $\bm{V}$-natural transformations. The composition
 \[
  \psi\circ \varphi : f \Longrightarrow h
 \]
 is defined by
 \[
  (\psi\circ\varphi)(a) : 1 \rarrow{\cong} 1\otimes 1
 \rarrow{\psi(a)\otimes \varphi(a)}
 B(g(a),h(a))\otimes B(f(a),g(a)) \rarrow{\circ} B(f(a),h(a)).
 \]
\end{definition}

%\begin{lemma}
% Small $\bm{V}$-categories, $\bm{V}$-functors, and $\bm{V}$-natural
% transformations form a $2$-category.
%\end{lemma}
\subsection{The Language of $2$-Categories}
\label{2-category}

We need the language of $2$-categories in order to fully understand the
role and meanings of the Grothendieck construction. Our reference is
\cite{Street72-2}. 

\begin{definition}
 A (strict) $2$-category $\bm{C}$ is a category enriched over the
 category of small categories. In other words, for each pair of objects
 $x$ and $y$ in $\bm{C}$, ${\bm{C}}(x,y)$ is a category. Objects in 
 ${\bm{C}}(x,y)$ are called $1$-morphisms and morphisms in
 ${\bm{C}}(x,y)$ are called $2$-morphisms.

 The category obtained from a $2$-category $\bm{C}$ by forgetting
 $2$-morphisms is denoted by $\sk_1\bm{C}$.
\end{definition}

\begin{example}
 \label{2-category_of_V-categories}
 $\bm{V}$-categories, $\bm{V}$-functors, and $\bm{V}$-natural
 transformations form a $2$-category $\categories{\bm{V}}$.
\end{example}

\begin{example}
 Any category can be regarded as a $2$-category whose $2$-morphisms are
 identities. 
\end{example}

\begin{example}
 \label{2-category_of_comma_categories}

 Let $B$ be a $\bm{V}$-category. Define a $2$-category
 $\lcats{\bm{V}}\downarrow B$ as 
 follows. Objects are $\bm{V}$-functors
 \[
  \pi : E \longrightarrow B.
 \]
 A left morphism from $\pi : E \to B$ to
 $\pi' : E' \to B$ is a pair $(F,\varphi)$ of a
 $\bm{V}$-functor
 \[
  F : E \longrightarrow E'
 \]
 and a $\bm{V}$-natural transformation
 \[
  \varphi : \pi'\circ F \Longrightarrow \pi.
 \]
 For left morphisms
 \[
  (F,\varphi), (G,\psi) : \pi \longrightarrow \pi',
 \]
 a $2$-morphism from $(F,\varphi)$ to $(G,\psi)$ is a $\bm{V}$-natural
 transformation
 \[
  \xi : F \Longrightarrow G
 \]
 making the following diagram commutative
 \[
 \xymatrix{%
 \pi'\circ F \ar@{=>}_{\varphi}[dr] \ar@{=>}^{\pi'\circ \xi}[rr] & &
 \pi'\circ G \ar@{=>}^{\psi}[dl] \\ 
  & \pi
 }
 \]

 By reversing the direction of $\varphi$ in left morphisms, we obtain
 right morphisms and another $2$-category $\rcats{\bm{V}}\downarrow B$
 with the same objects.
\end{example}

We can define and discuss ``functors'' between $2$-categories. One of
them is $2$-functors in Example \ref{2-functor_definition}. However,
we often encounter ``functors up to natural transformations''. The
notion of lax functor was introduced by Street in \cite{Street72-2} in
order to describe such ``functors''. For the
Grothendieck construction, we need oplax functors from ordinary
categories to $2$-categories.

\begin{definition}
 \label{oplax_functor_definition}
 Let $I$ be a category and $\bm{C}$ be a $2$-category. An oplax functor
 from $I$ to $\bm{C}$ consists of the following:
 \begin{itemize}
  \item a morphism of quivers
	\[
	 F : I \longrightarrow \sk_1\bm{C},
	\]
  \item for each object $i\in I_0$, a $2$-morphism
	\[
	 \eta_i : F(1_i) \Longrightarrow 1_{F(i)},
	\]
  \item for each pair of composable morphisms $i \rarrow{u} i'
	\rarrow{u'} i''$, a $2$-morphism
	\[
	 \theta_{u',u} : F(u'\circ u) \Longrightarrow F(u')\circ F(u),
	\]
 \end{itemize}
 satisfying the following conditions:
 \begin{enumerate}
  \item For any morphism $u : i \to j$ in $I$, the following diagram of
	$2$-morphisms are commutative
	\[
	\xymatrix{%
	F(u)\circ 1_{F(i)} & F(u)\circ F(1_i) \ar@{=>}[l] \\
	F(u) \ar@{=}[u] \ar@{=}[r] & F(u\circ 1_i) \ar@{=>}[u]
	}
	\]
	\[
	\xymatrix{%
	 F(1_j)\circ F(u) \ar@{=>}[r] & 1_{F(j)}\circ F(u) \\
	F(1_j\circ u) \ar@{=>}[u] & F(u) \ar@{=}[u] \ar@{=}[l]
	}
	\]
  \item For each sequence $a \rarrow{u} b \rarrow{v} c \rarrow{w} d$ in
	$I$, the following diagram of natural transformations is
	commutative
	\[
	\xymatrix{%
	F(w\circ v\circ u) \ar@{=>}[d] \ar@{=>}[r] & F(w\circ v)\circ
	F(u) \ar@{=>}[d] \\
	F(w)\circ F(v\circ u) \ar@{=>}[r] & F(w)\circ F(v)\circ F(u).
	}
	\]
 \end{enumerate}
\end{definition}

\begin{remark}
 There seem to be confusions on these terminologies. Our oplax functors
 are called lax functors by Goerss and Jardine
 \cite{Goerss-jardine}. Lax functors in the sense of Street are required
 to ``compose'' functors 
 \[
 \theta_{u',u} : F(u')\circ F(u) \Longrightarrow F(u'\circ u),  
 \]
 with respect to compositions of morphisms in $I$. We follow the
 original terminology in \cite{Street72-2}.
 A precise definition of lax functor can be found in this
 paper of Street's. 
\end{remark}

\begin{definition}
 Let $I$ and $\bm{C}$ be as above. Let
 \[
  X,Y : I \longrightarrow \bm{C}
 \]
 be oplax functors. A left transformation from $X$ to
 $Y$ consists of
 \begin{itemize}
  \item a family of $1$-morphisms
	\[
	 F(i) : X(i) \longrightarrow Y(i)
	\]
	indexed by objects $i\in I_0$,
  \item a family of $2$-morphisms $\varphi(u)$ indexed by morphisms in
	$I$ with
	\[
	 \xymatrix{
	X(i) \ar[r]^{F(i)} \ar[d]_{X(u)} & Y(i) \ar[d]^{Y(u)}
	\ar@{=>}[dl]_{\varphi(u)} \\ 
	X(j) \ar[r]_{F(j)} & Y(j)
	}
	\]
	if $u : i \to j$ in $I$,
 \end{itemize}
 satisfying the following conditions:
 \begin{enumerate}
  \item For any object $i\in I_0$, the following diagram is commutative
	\[
	\xymatrix{%
	Y(1_i)\circ F(i) \ar@{=>}[r] \ar@{=>}[d] & F(i)\circ X(1_i)
	\ar@{=>}[d] \\ 
	1_{Y(i)}\circ F(i) \ar@{=}[r] & F(i)\circ 1_{X(i)}
	}
	\]
  \item  For composable morphisms $i \rarrow{u} j \rarrow{v} k$ in $I$,
	 the following diagram is commutative
	 \[
	 \xymatrix{%
	 Y(v\circ u)\circ F(k) \ar@{=>}[r] \ar@{=>}[d] &  
	 Y(u)\circ Y(v)\circ F(k)  \ar@{=>}[r] &
	 Y(u)\circ F(j)\circ X(v) \ar@{=>}[d] \\
	 F(k)\circ X(v\circ u) 
	 \ar@{=>}[rr] & & 
	 F(k)\circ X(v)\circ X(u).
	 }
	 \]
 \end{enumerate}
 We denote
 \[
 (F,\varphi) : X \longrightarrow Y.
 \]

 For left transformations
 \[
  X \rarrow{(F,\varphi)} Y \rarrow{(G,\psi)} Z,
 \]
 the composition $(G\circ F, \psi\circ\varphi) : X \to Z$ is defined by
 \[
  (G\circ F)(i) = G(i)\circ F(i)
 \]
 for $i\in I_0$ and, for $u : i\to j$  in $I$,
 \[
  (\psi\circ\varphi)(u) : Z(u)\circ (G\circ F)(i) \Longrightarrow
 (G\circ F)(j)\circ X(u)
 \]
 is defined by the composition
 \begin{eqnarray*}
  Z(u)\circ (G\circ F)(i) & = & Z(u)\circ G(i)\circ
   F(i) \\
  & \Longrightarrow & G(j)\circ Y(u)\circ F(i) \\
  & \Longrightarrow & G(j)\circ F(j)\circ X(u) \\
  & = & (G\circ F)(j)\circ X(u).
 \end{eqnarray*}
\end{definition}

The reader might have noticed that the natural transformation
$\varphi(u)$ in the above definition goes in a wrong direction with
respect to $F$. But it is in the right direction with respect to $u$.

\begin{definition}
 Let $X$ and $Y$ be as above. A right transformation from $X$ to $Y$
 consists of a family of functors
 \[
 F(i) :  X(i) \longrightarrow Y(i)
 \]
 indexed by $i\in I_0$ and a family of $2$-morphisms $\varphi(u)$
 \[
 \xymatrix{
 X(i) \ar[r]^{F(i)} \ar[d]_{X(u)} & Y(i) \ar[d]^{Y(u)}
 \\ 
 X(j) \ar@{=>}[ur]^{\varphi(u)} \ar[r]_{F(j)} & Y(j)
 }
 \]
 in $\bm{C}$ satisfying conditions analogous to those of left
 transformations. 
\end{definition}

We may dualize and define left and right transformations for lax
functors. 

\begin{definition}
 Let $I$ and $\bm{C}$ be as above. Let
 \[
  X,Y : I \longrightarrow \bm{C}
 \]
 be lax functors. A left transformation of lax functors from $X$ to $Y$
 consists of 
 \begin{itemize}
  \item a family of $1$-morphisms
	\[
	 F(i) : X(i) \longrightarrow Y(i)
	\]
	indexed by $i\in I_0$,
  \item a family of $2$-morphisms
	\[
	 \varphi(u) :  Y(u)\circ F(i) \Longrightarrow F(j)\circ X(u)
	\]
	for $u : i\to j$,
 \end{itemize}
 satisfying the following conditions:
 \begin{enumerate}
  \item For any object $i\in I_0$, the following diagram is commutative
	\[
	 \xymatrix{
	Y(1_i)\circ F(i) \ar@{=>}[r] \ar@{=>}[d] & F(i)\circ X(1_i) \\
	1_{Y(i)}\circ F(i) \ar@{=>}[u] \ar@{=}[r] & F(i)\circ 1_{X(i)}
	\ar@{=>}[u]. 
	}
	\]
  \item For composable morphisms $i \rarrow{u} j \rarrow{v} k$ in $I$,
	the following diagram is commutative
	\[
	\xymatrix{
	Y(v)\circ F(j)\circ X(u) \ar@{=>}[d] & Y(v)\circ Y(u)\circ F(i)
	\ar@{=>}[r] \ar@{=>}[l] & Y(v\circ 
	u)\circ F(i) \ar@{=>}[d] \\
	F(k)\circ X(v)\circ X(u) \ar@{=>}[rr] & & F(k)\circ X(v\circ u).
	} 
	\]
 \end{enumerate}

 Right transformations of lax functors are defined by reversing the
 direction of $\varphi(u)$ and by changing the conditions accordingly. 
\end{definition}

%\begin{remark}
% The above terminologies are due to Street \cite{Street72-2}.
%\end{remark}

We have $2$-morphisms.

\begin{definition}
 Let
 \[
  X,Y : I \longrightarrow \bm{C}
 \]
 be oplax functors and
 \[
  (F,\varphi), (G,\psi) : X \longrightarrow Y
 \]
 be left transformations of oplax functors. A morphism of left
 transformations from 
 $(F,\varphi)$ to $(G,\psi)$ is a collection of $2$-morphisms in $\bm{C}$
 \[
  \theta(i) : F(i) \Longrightarrow G(i)
 \]
 indexed by $i\in I_0$ making the following diagram commutative
 \[
 \xymatrix{%
 Y(u)\circ F(i) \ar@{=>}_{\varphi(u)}[d] \ar@{=>}^{\theta(i)}[r] &
 Y(u)\circ G(i) 
 \ar@{=>}^{\psi(u)}[d] \\
 F(j)\circ X(u) \ar@{=>}_{\theta(j)}[r]  &
 G(j)\circ X(u) 
 }
 \]
 for each morphism $u : i \to j$ in $I$.

 The composition of morphisms of left transformations is
 defined by the composition of $2$-morphisms in $\bm{C}$. Morphisms
 between right transformations are defined by reversing arrows
 appropriately. We also have lax versions of left and right
 transformations. 
\end{definition}

Oplax functors and lax functors form $2$-categories but there are two
variations.

\begin{definition}
 Let $I$ be a small category and $\bm{C}$ be a $2$-category. 

 The $2$-category consisting of oplax functors from $I$ to $\bm{C}$,
 left transformations, and morphisms of left transformations
 is denoted by $\lOplax(I,\bm{C})$.
 By using right transformations instead of left transformations, we also 
 obtain a $2$-category, which is denoted by $\rOplax(I,\bm{C})$. 

 Lax functors, left transformations, morphisms of left transformations 
 form a $2$-category and it is denoted by $\lLax(I,\bm{C})$. The
 $2$-category obtained by replacing left transformations by
 right transformations is denoted by $\rLax(I,\bm{C})$.

 By restricting objects to strict functors, we obtain the full
 $2$-subcategories $\lFunct(I,\bm{C})$ and $\rFunct(I,\bm{C})$ whose
 $1$-morphisms are left and right transformations, respectively.
\end{definition}

We have the following diagonal functor.

\begin{definition}
 \label{2-diagonal}
 Define a functor
 \[
  \Delta : \bm{C} \longrightarrow \lOplax(I,\bm{C}) 
 \]
 by 
 \[
  \Delta(X)(i) = X
 \]
 on objects.
\end{definition}

It is useful to have an explicit description of a morphism of oplax
functors from $F$ to $\Delta(A)$.

\begin{lemma}
 Let $A$ be an object of a $2$-category $\bm{C}$ and $X : I \to \bm{C}$ be
 an oplax functor.  A morphism of oplax functors
 \[
 (F,\varphi) : X \longrightarrow \Delta(A) 
 \]
 consists of
 \begin{itemize}
  \item a family of $1$-morphisms
	\[
	 F(i) : X(i) \to A
	\]
	indexed by objects in $I$, and
  \item a family of $2$-isomorphisms
	\[
	 \varphi(u) : F(i) \Longrightarrow F(j)\circ X(u)
	\]
	indexed by morphisms in $I$,
 \end{itemize}
 satisfying the following conditions:
 \begin{enumerate}
  \item For each object $i\in I_0$, 
	\[
	 \eta_i\circ \varphi(1_i) = 1_{F(i)}.
	\]
  \item The following diagram of $2$-morphisms is commutative
	\[
	\xymatrix{%
	F(k) \ar@{=>}[d] \ar@{=>}[r] & F(k)\circ X(v) \ar@{=>}[d] \\
	F(k)\circ X(v\circ u) \ar@{=>}[r] & F(k)\circ
	X(v)\circ X(u). 
	}
	\]
 \end{enumerate}
\end{lemma}

\begin{proof}
 A direct translation of definition.
\end{proof}

\begin{example}
 \label{Gamma-invariant_functor}
 Let $G$ be a group. Consider the case $\bm{C}=\categories{k}$. 
 A strict functor
 \[
  X : G \longrightarrow \categories{k}
 \]
 is given by a $k$-linear category $X$ equipped with an action of $G$.
 For a $k$-linear category $Y$, a morphism of oplax functors
 \[
  F : X \longrightarrow \Delta(Y)
 \]
 is given by a functor
 \[
  F(\ast) : X(\ast)=X \longrightarrow \Delta(Y)(\ast) = Y
 \]
 and a a family of natural transformations
 \[
  \varphi(\alpha) : F(\ast) \Longrightarrow F(\ast)\circ X(\alpha)
 \]
 indexed by $\alpha\in G$ making the following diagram commutative
 \[
 \xymatrix{%
 F(\ast) \ar@{=>}[d] \ar@{=>}[r] & F(\ast)\circ X(\beta) \ar@{=>}[d] \\
 F(\ast)\circ X(\alpha\beta) \ar@{=>}[r] & F(\ast)\circ
 X(\alpha)\circ X(\beta). 
 }
 \]
 This is nothing but the definition of right $G$-invariant functor
 in \cite{0807.4706}.
\end{example}

\subsection{Comodule Categories}
\label{comodule_category}

We have seen in \S\ref{comodule} that we often obtain a coproduct
decomposition from a comodule structure, extending the idea of a
characterization of group graded algebras by Cohen and
Montgomery \cite{Cohen-Montgomery84}.

The notion of group graded algebras has been extended to group graded
$k$-linear categories and to $k$-linear categories graded by a small
category. See \cite{math/0312214,0807.4706,Lowen08}, for 
example. In order to extend their definitions to $\bm{V}$-categories
graded by a small category $I$, we introduce and investigate ``many
objectifications'' of comodules.

Recall from Lemma \ref{V-categories_is_monoidal} that the
category of $\bm{V}$-categories has a symmetric monoidal structure.

\begin{definition}
 A coalgebra $\bm{V}$-category is a comonoid object in the monoidal
 category of $\bm{V}$-categories.
\end{definition}

\begin{remark}
 In other words, a coalgebra $\bm{V}$-category is a
 $\bm{V}$-category $C$ equipped with a family of morphisms
 \begin{eqnarray*}
  \Delta_{a,b} : C(a,b) \longrightarrow C(a,b)\otimes C(a,b) \\
  \varepsilon_{a,b} : C(a,b) \longrightarrow 1
 \end{eqnarray*}
 indexed by pairs of objects $a,b\in C_0$ satisfying the following
 conditions:
 \begin{enumerate}
  \item $\Delta_{a,b}$ is compatible with compositions, i.e.\ 
	the following diagram is commutative
	\[
	 \begin{diagram}
	  \node{C(b,c)\otimes C(a,b)} \arrow{e,t}{\circ}
	  \arrow{s,l}{\Delta_{b,c}\otimes\Delta_{a,b}} \node{C(a,c)} 
	  \arrow[2]{s,r}{\Delta_{a,c}} \\ 
	  \node{(C(b,c)\otimes C(b,c))\otimes (C(a,b)\otimes C(a,b))} 
	  \arrow{s,l}{T} \\
	  \node{(C(b,c)\otimes C(a,b))\otimes (C(b,c) \otimes
	  C(a,b))} \arrow{e,t}{\circ \otimes \circ} \node{C(a,c)\otimes
	  C(a,c),} 
	 \end{diagram}
	\]
	where $T$ is an appropriate composition of associators and a
	symmetry operator. (Recall that we assume $\bm{V}$ is symmetric
	monoidal.) 
  \item $\Delta_{a,a}$ preserves identity morphisms, i.e. the following
	diagram is commutative for each object $a\in C_0$ 
	\[
	 \begin{diagram}
	  \node{1} \arrow{s,l}{1_a} \node{1\otimes 1} \arrow{w,t}{\cong}
	  \arrow{s,r}{1_a\otimes 1_a} \\
	  \node{C(a,a)} \arrow{e,t}{\Delta_{a,a}} \node{C(a,a)\otimes
	  C(a,a),} 
	 \end{diagram}
	\]
  \item $\varepsilon_{a,b}$ is a counit for $\Delta_{a,b}$, i.e.\
	the following diagram is commutative
	\[
	 \begin{diagram}
	  \node{} \node{C(a,b)} \arrow{s,r}{\Delta_{a,b}}
	  \arrow{sw,t}{\cong} \arrow{se,t}{\cong}
	  \\ 
	  \node{1\otimes C(a,b)}
	  \node{C(a,b)\otimes  C(a,b)}
	   \arrow{w,t}{\varepsilon_{a,b}\otimes 1}
	  \arrow{e,t}{1\otimes \varepsilon_{a,b}}
	  \node{C(a,b)\otimes 1}
	 \end{diagram}
	\]
  \item $\Delta_{a,b}$ is coassociative,  i.e.\ the following
	diagram is commutative
	\[
	 \begin{diagram}
	  \node{C(a,b)} \arrow{e,t}{\Delta_{a,b}} \arrow[2]{s,l}{\Delta_{a,b}}
	  \node{C(a,b)\otimes C(a,b)}
	  \arrow{s,r}{1\otimes\Delta_{a,b}} \\
	  \node{} \node{C(a,b)\otimes (C(a,b)\otimes
	  C(a,b))} \arrow{s,r}{\cong} \\
	  \node{C(a,b)\otimes C(a,b)}
	  \arrow{e,t}{\Delta_{a,b}\otimes 1}
	  \node{(C(a,b)\otimes C(a,b))\otimes C(a,b).}
	 \end{diagram}
	\]
 \end{enumerate}
 Another way of saying this is that a coalgebra $\bm{V}$-category is a
 category enriched over the category of comonoid objects in $\bm{V}$.
\end{remark}

\begin{example}
 Let $I$ be a small category. Then we have a $\bm{V}$-category
 $I\otimes 1$ by Lemma \ref{free_category}. Each $(I\otimes 1)(i,j)$ has
 a structure of comonoid by Example \ref{coproduct_on_free_object}.
 The comonoid structure is compatible with the compositions of morphisms
 and preserves identities. Thus $I\otimes 1$ is a coalgebra category.
\end{example}

\begin{example}
 \label{coalgebra_in_T}
 When $\bm{V}$ is of product type, any object $C$ has a canonical
 comonoid structure by Example \ref{everybody_is_comonoid}.
 The assumption that the unit object $1$ is terminal guarantees
 the existence of counit morphisms
 \[
  \varepsilon_{a,b} : C(a,b) \longrightarrow 1
 \]
 and we have a coalgebra structure on $C$.
\end{example}

\begin{definition}
 Let $C$ be a coalgebra $\bm{V}$-category. When a
 $\bm{V}$-category $X$ is equipped with a right comodule structure over
 $C$
 \[
  \mu : X \longrightarrow X\otimes C,
 \]
 it is called a right comodule category over $C$.
 Left comodule categories are defined analogously.
\end{definition}

It is convenient to have a more concrete description.

\begin{lemma}
 \label{concrete_comodule_category}
 Let $C$ be a coalgebra $\bm{V}$-category and $X$ be a
 $\bm{V}$-category. A right comodule structure on $X$ consists of 
 \begin{itemize}
  \item a map
	\[
	p : X_0 \longrightarrow C_0
	\]
	between objects, and
  \item a collection of morphisms
	\[
	\mu_{x,y} : X(x,y) \longrightarrow X(x,y)\otimes C(p(x),p(y))
	\]
	indexed by pairs of objects in $X$
 \end{itemize}
 satisfying the following conditions:
 \begin{enumerate}
  \item The following diagram is commutative
	\[
	 \begin{diagram}
	  \node{X(y,z)\otimes X(x,y)} \arrow{e,t}{\circ}
	  \arrow{s,l}{\mu_{y,z}\otimes\mu_{x,y}} \node{X(x,z)} 
	  \arrow[2]{s,r}{\mu_{x,z}} \\ 
	  \node{X(y,z)\otimes (C(p(y),p(z)))\otimes
	  (X(x,y)\otimes C(p(x),p(y)))}  
	  \arrow{s,l}{T} \\
	  \node{(X(y,z) \otimes X(x,y))\otimes
	  (C(p(y),p(z))\otimes C(p(x),p(y)))} \arrow{e,t}{\circ
	  \otimes \circ} 
	  \node{X(x,z)\otimes C(p(x),p(z)),} 
	 \end{diagram}
	\]
	where $T$ is an appropriate composition of associators and a
	symmetry operator.
  \item The following diagram is commutative for each object $x\in X_0$
	\[
	 \begin{diagram}
	  \node{1} \arrow{s,l}{1_{x}}\node{1\otimes 1} \arrow{w,t}{\cong}
	  \arrow{s,r}{1_{x}\otimes 1_p(x)} \\
	  \node{X(x,x)} \arrow{e,t}{\mu_{x,x}} \node{X(x,x)\otimes
	  C(p(x),p(x)).} 
	 \end{diagram}
	\]
  \item The following diagram is commutative
	\[
	 \begin{diagram}
	  \node{} \node{X(x,y)} \arrow{s,r}{\mu_{x,y}}
	  \arrow{sw,t}{\cong} \\ 
	  \node{X(x,y)\otimes 1}
	  \node{X(x,y)\otimes  C(p(x),p(y))}
	   \arrow{w,t}{1\otimes\varepsilon_{x,y}.}
	 \end{diagram}
	\]
  \item $\mu_{x,y}$ is coassociative in the sense that the following
	diagram is commutative
	\[
	 \begin{diagram}
	  \node{X(x,y)} \arrow{e,t}{\mu_{x,y}} \arrow[2]{s,l}{\mu_{x,y}}
	  \node{X(x,y)\otimes C(p(x),p(y))}
	  \arrow{s,r}{\mu_{x,y}\otimes 1} \\
	  \node{} \node{(X(x,y)\otimes
	  C(p(x),p(y))\otimes C(p(x),p(y)} \arrow{s,r}{\cong} \\
	  \node{X(x,y)\otimes C(p(x),p(y))}
	  \arrow{e,t}{1\otimes \Delta_{x,y}}
	  \node{X(x,y)\otimes (C(p(x),p(y))\otimes C(p(x),p(y))).}
	 \end{diagram}
	\]
 \end{enumerate}
\end{lemma}

\begin{proof}
 A functor $\mu : X \to X\otimes C $ defines a map
 \[
  \mu : X_0 \longrightarrow X_0\times C_0.
 \]
 Define $p=\pr_2\circ\mu : X_0\to C_0$. Then the counit condition
 implies 
 \[
  \mu(x) = (x,p(x))
 \]
 for $x \in C_0$ and for each pair of objects $x,y\in X_0$ we obtain
 \[
  \mu_{x,y} : X(x,y) \longrightarrow X(x,y)\otimes C(p(x),p(y)).
 \]
 It is easy to verify that the conditions on $\mu$ for comodule
 structure corresponds to conditions in this Lemma.
\end{proof}

\begin{example}
 \label{comodule_from_comma_category}
 Suppose $\bm{V}$ is of product type. By Example \ref{coalgebra_in_T}, any
 $\bm{V}$-category has a canonical coalgebra structure and any
 $\bm{V}$-functor 
 \[
  p : X \longrightarrow C
 \]
 is compatible with coproducts and counits. The composition
 \[
  X(x,y) \rarrow{\Delta} X(x,y)\otimes X(x,y) \rarrow{1\times p}
 X(x,y)\otimes  C(p(x),p(y))
 \]
 makes $X$ into a comodule over $C$.
\end{example}

Note that right comodules over a comonoid object in a
symmetric monoidal category, in general, form a category in an obvious
way (Definition \ref{morphism_of_comodules}). Namely by requiring a
morphism to make the diagram commutative 
\[
 \begin{diagram}
  \node{M} \arrow{e} \arrow{s} \node{N} \arrow{s} \\
  \node{M\otimes C} \arrow{e} \node{N\otimes C.}
 \end{diagram}
\]
In the case of comodule categories, we should relax the commutativity of
this diagram in the following way.

\begin{definition}
 \label{2-category_of_comodules}
 Let $C$ be a coalgebra $\bm{V}$-category and $(X,\mu)$ and $(X',\mu')$
 be right comodule categories over $C$. A left morphism of right comodule
 categories from $(X,\mu)$ to $(X',\mu')$ is a pair $(F,\varphi)$ of a
 $\bm{V}$-functor 
 \[
  F : X \longrightarrow X'
 \]
 and a $\bm{V}$-natural transformation
 \[
 \xymatrix{
 X \ar[d]_{\mu} \ar[r]^{F} & X' \ar[d]^{\mu'} \ar@{=>}[ld]_{\varphi} \\
 X\otimes C \ar[r]_{F\otimes 1} & X'\otimes C.
 }
 \]
 The composition of left morphisms are defined by the following diagram
 \[
 \xymatrix{%
 X \ar[d]^{\mu} \ar[r]^{F} & X' \ar[d]_{\mu'} \ar@{=>}[ld] \ar[r]^{F'}
 & X'' \ar[d]^{\mu''} \ar@{=>}[ld] \\
 X\otimes C \ar[r]_{F\otimes 1}  & X'\otimes C
 \ar[r]_{F'\otimes 1}  & X''\otimes C. 
 }
 \]
 
 For left morphisms
 \[
  (F,\varphi), (G,\psi) : (X,\mu) \longrightarrow (X',\mu'),
 \]
 a $2$-morphism from $(F,\varphi)$ to $(G,\psi)$ is a $\bm{V}$-natural
 transformation 
 \[
  \xi : F \Longrightarrow G
 \]
 making the following diagram commutative
 \[
  \xymatrix{%
  \mu'\circ F \ar@{=>}^{\varphi}[r] \ar@{=>}_{\mu'\circ \xi}[d] &
 (F\otimes 1)\circ \mu \ar@{=>}^{(\xi\otimes 1)\circ \mu}[d] \\
 \mu'\circ G \ar@{=>}^{\psi}[r] & (G\otimes 1)\circ \mu
 }
 \]
 The composition of $2$-morphisms are given by the composition of
 $\bm{V}$-natural transformations.

 Right morphisms for right comodules are defined analogously by
 reversing the direction of $\Rightarrow$.
 For left comodules, we also define left and right morphisms
 analogously. 

 The $2$-categories of right comodule categories over $C$ and left
 morphisms, of right comodules categories and right morphisms, of left
 comodule categories and left morphisms, and of left comodule
 categories and right morphisms are denoted by $\rlComod{C}$,
 $\rrComod{C}$, $\llComod{C}$, and $\lrComod{C}$,
 respectively. $2$-categories of bicomodule categories $\bilComod{C}{C}$
 and $\birComod{C}{C}$ are defined similarly.
\end{definition}

\begin{remark}
 The above definition can be generalized to comodules over a comonoid
 object in a symmetric monoidal $2$-category.
\end{remark}

\begin{example}
 \label{comma_category_as_comodule}

 Suppose $\bm{V}$ is of product type. Recall from Example
 \ref{comodule_from_comma_category} that any $\bm{V}$-category $C$ can
 be regarded as a coalgebra category and any functor
 \[
  p : X \longrightarrow C
 \]
 defines a right comodule structure
 \[
  X \rarrow{\Delta} X\otimes X \rarrow{1\otimes p}
 X \otimes C 
 \]
 on $X$. It is easy to see that this
 correspondence defines a $2$-functor
 \[
  \lcats{\bm{V}}\downarrow C \longrightarrow \rlComod{C},
 \]
 where the structure of $2$-category on $\lcats{\bm{V}}\downarrow C$ is
 defined in Example \ref{2-category_of_comma_categories}. Since
 $\otimes$ is the direct product, we can recover $p$ from
 $(1\otimes p)\circ \Delta$. However, there are differences in
 $1$-morphisms. For example, a morphism
 \[
  (F,\varphi) : (X,\mu) \longrightarrow (X',\mu')
 \]
 of right $C$-comodule categories is given by a $\bm{V}$-functor
 \[
  F : X \longrightarrow X'
 \]
 and a family of morphisms
 \[
  \varphi(x) : (F(x), p(F(x))) \longrightarrow (F(x), p(x))
 \]
 in $X'\otimes C$. Since $\otimes$ is the product in $\bm{V}$,
 $\varphi(x)$ is of the form
 \[
  \varphi(x) = \varphi_1(x)\otimes \varphi_2(x) \in
 (X\otimes C)((F(x),p(F(x))),(F(x),p(x))) = X(F(x),F(x))\otimes
 C(p(F(x)),p(x)). 
 \]
 The $1$-morphisms coming from $\lcats{\bm{V}}\downarrow C$ are
 exactly those morphisms having natural transformations whose first
 component is the identity.
 
 Thus we can regard $\lcats{\bm{V}}\downarrow C$ as a
 $2$-subcategory of $\rlComod{C}$.
\end{example}

When $\bm{V}$ is Abelian, we cannot expect such a simple
characterization of comodule categories as above.
We will see in \S\ref{graded_category} that, in this case,
the notion of comodules over a free $\bm{V}$-category $I\otimes 1$
generated by a small category $I$ corresponds to the notion of graded
categories.

\section{The Grothendieck Construction}
\label{definitions}

In this section, we define the Grothendieck construction as a
$2$-functor from the category of oplax functors from $I$ to the category
of comodules over $I\otimes 1$. The construction works for lax functors
by reversing arrows appropriately.

\subsection{The Grothendieck Construction for Oplax and Lax Functors}
\label{definition_for_oplax_functor}

The following is our definition of the Grothendieck construction. Recall
that $\bm{V}$ is a symmetric monoidal category satisfying Assumption
\ref{condition_on_V}. In particular it is closed under arbitrary
coproducts. 

\begin{definition}
 Let $I$ be a small category. For an oplax functor
 \[
  X : I \longrightarrow \categories{\bm{V}},
 \]
 define a $\bm{V}$-category $\Gr(X)$ as follows: Objects are
 given by
 \[
  \Gr(X)_0 = \coprod_{i\in I_0} X(i)_0\times\{i\}.
 \]
 For $(x,i),(y,j) \in \Gr(X)_0$, define
 \[
  \Gr(X)((x,i),(y,j)) = \bigoplus_{u:i\to j} X(j)(X(u)(x),y).
 \]
 The composition
 \[
 \circ : \Gr(X)((y,j),(z,k))\otimes\Gr(X)((x,i),(y,j)) \longrightarrow
 \Gr(X)((x,i),(z,k)) 
 \]
 is given, on each component, by
 \begin{eqnarray*}
  X(k)(X(v)(y),z)\otimes X(j)(X(u)(x),y) & \rarrow{1\otimes X(v)} &
   X(k)(X(v)(y),z)\otimes X(k)(X(v)(X(u)(x)),X(v)(y)) \\
  & \rarrow{\theta_{v,u}^*} & X(k)(X(v)(y),z)\otimes X(k)(X(v\circ
   u)(x)),X(v)(y)) \\ 
 & \rarrow{\circ} & X(k)(X(v\circ u)(x),z).
 \end{eqnarray*}

\end{definition}

One of the simplest examples is the semidirect product construction for
groups. 

\begin{example}
 \label{translation_groupoid}
 Suppose a group $G$ acts on another group $H$.
 We regard $G$ as a category with a
 single object $\ast$. Then the action defines a functor
 \[
  H : G \longrightarrow \category{Groups} \subset \Categories
 \]
 by
 \[
  H(\ast) = H
 \]
 and
 \[
  H(g) = g\cdot : H \longrightarrow H.
 \]
 The Grothendieck construction of this functor is nothing but the
 semidirect product.
\end{example}

This example can be extended as follows.

\begin{example}
 Let $A$ be a $\bm{V}$-category. Suppose a group $G$ acts on $A$ from
 the left via $\bm{V}$-functors. We regard it as a functor
 \[
  A : G \longrightarrow \bm{V}\category{-Categories}.
 \]

 The Grothendieck construction of $A$ is called the orbit
 category by Cibils and Marcos in \cite{math/0312214} when $\bm{V}$ is
 the category of $k$-modules.

 By definition, objects of $\Gr(A)$ can be identified with
 objects in $A$
 \[
  \Gr(A)_0 = A_0\times \{\ast\} \cong A_0.
 \]
 For $x,y\in A_0$, morphisms are given by
 \[
  \Gr(A)(x,y) = \bigoplus_{g\in G} {A}(gx,y)
 \]
 and compositions are given, on each component, by
 \[
  A(gy,z)\otimes A(hx,y) \rarrow{1\otimes A(g)} A(gy,z)\otimes A(ghx,gy)
 \rarrow{\circ} A(ghx,z).
 \]
\end{example}

It is not difficult to define $\Gr$ as a $2$-functor
\[
 \Gr : \lOplax(I,\categories{\bm{V}}) \longrightarrow
 \categories{\bm{V}}. 
\]

\begin{definition}
 \label{Gr(F)_definition}
 For a left transformation of oplax functors
 \[
  (F,\varphi) : X \longrightarrow Y,
 \]
 define
 \[
  \Gr(F,\varphi) : \Gr(X) \longrightarrow \Gr(Y)
 \]
 by
 \[
  \Gr(F,\varphi)(x,i) = (F(i)(x),i)
 \]
 for objects and 
 \[
 \Gr(F,\varphi) : \Gr(X)((x,i),(y,j)) \longrightarrow
 \Gr(Y)((F(i)(x),i), (F(j)(y),j))
 \]
 by the composition
 \[
  X(j)(X(u)(x),y) \rarrow{F(j)} Y(j)(F(j)(X(u)(x),F(j)(y)))
 \rarrow{\varphi(u)^*} Y(j)(Y(u)(F(i)(x)),F(j)(y))
 \]
 on each component.
\end{definition}

\begin{definition}
 \label{Gr(theta)}
 For a $2$-morphism
 \[
  \theta : (F,\varphi) \Longrightarrow (G,\psi)
 \]
 in $\lOplax(I,\categories{\bm{V}})$, define a $2$-morphism
\[
 \Gr(\theta) : \Gr(F,\varphi) \Longrightarrow
 \Gr(G,\psi)
\]
in $\categories{\bm{V}}$ by 
 \[
  \Gr(\theta)(x,i) = (\theta(i)(x), 1_i) : \Gr(F,\varphi)(x,i) =
 (F(i)(x),i) \longrightarrow (G(i)(x),i) = \Gr(G,\psi)(x,i).
 \]
\end{definition}

\begin{proposition}
 The above constructions define a $2$-functor
 \[
  \Gr : \lOplax(I,\categories{\bm{V}}) \longrightarrow
 \categories{\bm{V}}. 
 \]
\end{proposition}

We may dualize the above construction to obtain 
\[
 \Gr : \rLax(I^{\op},\categories{\bm{V}}) \longrightarrow
 \categories{\bm{V}}. 
\]
We briefly describe the construction.

\begin{definition}
 For a lax functor
 \[
  X : I^{\op} \longrightarrow \categories{\bm{V}},
 \]
 define a $\bm{V}$-category $\Gr(X)$ by
 \[
  \Gr(X)_0 = \coprod_{i\in I_0}\{i\}\times X(i)_0
 \]
 and 
 \[
  \Gr(X)((i,x),(j,y)) = \bigoplus_{u : i\to j} X(i)(x,X(u)(y)).
 \]
 The composition is given on each component by 
 \begin{eqnarray*}
  X(j)(y,X(v)(z))\otimes X(i)(x,X(u)(y)) & \rarrow{X(u)\otimes 1} &
   X(i)(X(u)(y),X(u)\circ X(v)(z))\otimes X(i)(x,X(u)(y)) \\
  & \rarrow{\theta_*} & X(i)(X(u)(y),X(v\circ u)(z))\otimes
   X(i)(x,X(u)(y)) \\ 
  & \rarrow{\circ} & X(i)(x,X(v\circ u)(z)).
 \end{eqnarray*}

 For a right transformation $(F,\varphi)$ of lax functors
 \[
  \xymatrix{
  X(j) \ar[r]^{F(j)} \ar[d]_{X(u)} & Y(j) \ar[d]^{Y(u)}
 \\ 
 X(i) \ar@{=>}[ur]^{\varphi(u)} \ar[r]_{F(i)} & Y(i),
 }
 \]
 define
 \[
  \Gr(F,\varphi) : \Gr(X) \longrightarrow \Gr(Y)
 \]
 by
 \[
  \Gr(F,\varphi)(i,x) = (i,F(i)(x))
 \]
 for objects and 
 \[
  \Gr(F,\varphi) : \Gr(X)((i,x),(j,y)) \longrightarrow
 \Gr(Y)((i,F(i)(x)),(j,F(j)(y))) 
 \]
 by the composition
 \[
  X(i)(x,X(u)(y)) \rarrow{F(i)} Y(i)(F(i)(x),F(i)(X(u)(y)))
 \rarrow{\varphi(u)_*} Y(i)(F(i)(x),Y(u)(F(j)(y))).
 \]
\end{definition}

\begin{proposition}
 The above constructions define a $2$-functor
 \[
 \Gr : \rLax(I^{\op},\categories{\bm{V}}) \longrightarrow
 \categories{\bm{V}}. 
 \]
\end{proposition}

The above propositions say the Grothendieck construction is a process to glue
$\bm{V}$-categories in a given diagram of $\bm{V}$-categories together
to form a single $\bm{V}$-category. We show in
\S\ref{left_adjoint_to_diagonal} that this $\Gr$ is left adjoint to the
diagonal functor and $\Gr(X)$ can be regarded as a $2$-colimit of
$X$. If we regard the Grothendieck 
construction in this way, there is no hope to recover the original
diagram. Note that the Grothendieck construction is defined by
coproducts and $\Gr(X)$ is a $\bm{V}$-category whose morphism objects
have coproduct decompositions. It turns out that these coproduct
decompositions allow us to construct a diagram of $\bm{V}$-categories
which is very closely related to the original diagram $X$.

We use the notion of graded category, which is the subject of the next
section \S\ref{graded_category}, in order to make the above
statement precise.
\subsection{Graded Categories}
\label{graded_category}

Given a group $G$, the notion of $G$-graded category was introduced and
has been studied for $k$-linear categories by extending the notion of
$G$-graded algebras. More generally, $k$-linear categories graded by a
small category have been defined and used in the deformation theory of
$k$-linear prestacks \cite{Lowen08,0905.2354}. In these approaches
graded categories are defined by using coproducts.

We have seen in \S\ref{comodule} that
coproduct decompositions are intimately related to comodule structures
over a ``free coalgebra''. This observation allows us to
define graded categories without coproduct decompositions. This
viewpoint also suggests us to extend the notion of morphisms by
incorporating appropriate natural transformations. Our
definition is also suggested by the definition of degree preserving
functor by Asashiba \cite{0905.3884}. Note that morphisms between graded
categories defined in \cite{Lowen08} correspond to strictly degree
preserving functors in \cite{0905.3884} and our category of $I$-graded
$\bm{V}$-categories is larger than Lowen's.

The following is our definition of $\bm{V}$-categories graded by a small
category $I$.

\begin{definition}
 \label{graded_category_definition}
 Let $I$ be a small category and $X$ be a $\bm{V}$-category. An
 $I$-grading on $X$ is a right comodule structure
 \[
  \mu : X \longrightarrow X\otimes (I\otimes 1)
 \]
 on $X$ over the coalgebra category $I\otimes 1$.
\end{definition}

\begin{remark}
 We may also use left comodules. Since $I\otimes 1$ is a cocommutative
 coalgebra category, it does not make any essential difference.
\end{remark}

A more concrete description of $I$-grading can be obtained in many cases.

\begin{lemma}
 Suppose $\bm{V}$ is the category $\lMod{k}$ for a commutative ring
 $k$. Then an $I$-grading on a 
 $\bm{V}$-category $X$ consists of a map 
 \[
  p : X_0 \longrightarrow I_0
 \]
 and a family of coproduct decompositions
 \[
 X(x,y) \rarrow{\cong} \bigoplus_{u:p(x)\to p(y) \text{ in }
 I} X^u(x,y) 
 \]
 indexed by pairs of objects in $X$ satisfying the following conditions:
 \begin{enumerate}
  \item For any object $x\in X_0$, we have the following diagram
	\[
	\begin{diagram}
	 \node{1} \arrow{e,t}{1_x} \arrow{se,..} \node{X(x,x)} \\
	\node{} \node{X^{1_{p(x)}}(x,x)} \arrow{n,J}
	\end{diagram}
	\]
  \item For objects $x,y,z\in X_0$, the following diagram is commutative
	\[
	 \begin{diagram}
	  \node{X(y,z)\otimes X(x,y)} \arrow{e,t}{\circ} \node{X(x,z)} \\
	  \node{X^u(y,z)\otimes X^v(x,y)} \arrow{n,J} \arrow{e,..}
	  \node{X^{uv}(x,z).} \arrow{n,J}
	 \end{diagram}
	\]
 \end{enumerate}
\end{lemma}

\begin{proof}
 This is a direct translation of Lemma \ref{decomposition_and_comodule}
 and Lemma \ref{concrete_comodule_category}.
\end{proof}

\begin{example}
 \label{G-graded_category}
 Let $G$ be a group and consider the case $\bm{V} = \lMod{k}$. The
 unit object of $\lMod{k}$ is the ground ring $k$ and the
 $k$-linear category $G\otimes k$ is nothing but the group algebra
 $k[G]$ over $k$ regarded as a $k$-linear category with a single
 object. 

 Since $G\otimes 1=k[G]$ has the only object $\ast$, a $G$-grading on a
 $k$-linear category $A$ is given by 
 a coproduct decomposition
 \[
  {A}(x,y) = \bigoplus_{g\in G} A^g(x,y)
 \]
 satisfying the condition that the functor $p$ induces 
 \[
  \circ : A^g(y,z)\otimes A^h(x,y) \longrightarrow A^{gh}(x,z).
 \]
 and
 \[
  \begin{diagram}
   \node{1} \arrow{e,t}{1_x} \arrow{se,..} \node{A(x,x)} \\
   \node{} \node{A^1(x,x).} \arrow{n}
  \end{diagram}
 \]
 This is the definition of a $G$-graded
 $k$-linear category in \cite{math/0312214,0807.4706,0905.3884}. 
\end{example}

Even if $\bm{V}$ is not Abelian, we often obtain an analogous coproduct
decomposition.

\begin{example}
 \label{nonenriched_graded_category}
 Let $I$ be a small category. Consider the case that $\bm{V}$ is the
 category of sets. Suppose a small category $X$ has an $I$-grading. As
 we have seen in Example \ref{comma_category_as_comodule}, it is
 determined by a functor
 \[
  p : X \longrightarrow I.
 \]
  We have a map of morphism sets
 \[
  p_{x,y} : \Mor_X(x,y) \longrightarrow \Mor_I(p(x),p(y))
 \]
 for each pair of objects $x,y \in X_0$. Denote
 \[
  \Mor^u_X(x,y) = p_{x,y}^{-1}\{u\}
 \]
 and we have a coproduct decomposition
 \[
  \Mor_X(x,y) = \coprod_{u:p(x)\to p(y)} \Mor^u_X(x,y).
 \]
 The composition of morphisms in $X$ induces a map
 \[
  \Mor^u_X(y,z)\times \Mor^v_X(x,y) \longrightarrow \Mor^{uv}(x,z).
 \]

 Conversely, we can recover $p$ from the above coproduct decomposition
 by using the following composition
 \[
  \Mor_X(x,y)= X(x,y) \cong \coprod_{u:p(x)\to p(y)} X^u(x,y)
 \rarrow{} 
 \coprod_{u:p(x) \to p(y)} \{u\} = \Mor_I(p(x),p(y)).
 \]

 Note that this argument works without a change when $\bm{V}$ is the
 category of topological spaces, simplicial sets, and small categories.
\end{example}

\begin{example}
 \label{graded_spectral_category}
 Suppose $\bm{V}$ is a stable model category and $X$ is an $I$-graded
 $\bm{V}$-category. By Remark 
 \ref{decompostion_for_stable_model_category}, we have a coproduct
 decomposition of each $X(x,y)$ up to weak equivalences.
% \[
%  X(x,y) \simeq \bigoplus_{u:p(x)\top(y)} X^u(x,y).
% \]
 In particular, we have
 \[
  X(x,y) \simeq \bigvee_{u: p(x)\top(y)} X^u(x,y)
 \]
 when $X$ is an $I$-graded spectral category in our sense. See
 \cite{0801.4524,0802.3938} for spectral categories. 
\end{example}

We have also seen in Example \ref{comma_category_as_comodule} that, when
$\bm{V}$ is of product type, the comma $2$-category
$\lcats{\bm{V}}\downarrow (I\otimes 1)$ can be regarded as a 
$2$-subcategory of the $2$-category $\rlComod{(I\otimes 1)}$ of
comodules over $I\otimes 1$. 
Thus we would like to define a $2$-category of $I$-graded
categories which generalizes the $2$-category of the comma category
$\lcats{\bm{V}}\downarrow (I\otimes 1)$ when $\bm{V}$ is of product
type and the $2$-category of $G$-graded categories defined by Asashiba
\cite{0905.3884}. Note that, although we have defined a $2$-category of
comodule categories in Definition \ref{2-category_of_comodules},
$1$-morphisms in this $2$-category is less restrictive than $1$-morphisms
of $G$-graded categories.

\begin{example}
 Let $G$ be a group and 
 \begin{eqnarray*}
  \mu & : & A \longrightarrow A\otimes k[G] \\
  \mu' & : & B \longrightarrow B\otimes k[G]
 \end{eqnarray*}
 be $G$-graded $k$-linear categories regarded as comodules over
 $k[G]$. Let 
 \[
  (F,\varphi) : (A,\mu) \longrightarrow (B,\mu')
 \]
 be a left morphism of $k[G]$-comodules. $F$ is a $k$-linear
 functor
 \[
  F : A \longrightarrow B
 \]
 and $\varphi$ is a $k$-linear natural transformation
 \[
 \xymatrix{
 A \ar[r]^{F} \ar[d]_{\mu} & B \ar[d]^{\mu'} \ar@{=>}[dl]_{\varphi} \\
 A\otimes k[G] \ar[r]_{F\otimes 1} & B\otimes k[G].
 }
 \]
 Note that $\mu$ and $\mu'$ are identity on objects and, for each object
 $a\in A_0$, $\varphi(a)$ can be regarded as an element
 \[
  \varphi(a) \in (B\otimes k[G])(\mu'(F(a)),(1\otimes
 F)(\mu(a)))  = B(F(a),F(a))\otimes k[G].
 \]

 The condition that $\varphi$ is a $k$-linear natural
 transformation implies that the following diagram is commutative
 \begin{center}
  \scalebox{.8}{%
  $\begin{diagram}
   \node{} \node{A(a,a')} \arrow{se,t}{\ell^{-1}} 
   \arrow{sw,t}{r^{-1}} \\
   \node{A(a,a')\otimes 1} \arrow{s,l}{\mu\otimes 1}
   \node{} \node{1\otimes A(a,a')}
   \arrow{s,r}{1\otimes F} \\  
   \node{(A(a,a')\otimes k[G])\otimes 1} \arrow{s,l}{(F\otimes
   1)\otimes\varphi(a)} \node{} 
   \node{1\otimes B(F(a),F(a'))} \arrow{s,r}{\varphi(a')\otimes \mu'} \\
   \node{(B(F(a),F(a'))\otimes k[G])\otimes (B(F(a),F(a))\otimes k[G])}
   \arrow{se,b}{\circ} \node{}
   \node{(B(F(a'),F(a')\otimes k[G])\otimes (B(F(a),F(a'))\otimes k[G])}
   \arrow{sw,b}{\circ} \\
   \node{}\node{B(F(a),F(a'))\otimes k[G].} \node{}
  \end{diagram}$}
 \end{center}

 Suppose $\varphi(a)$ is of the form
 \[
  \varphi(a) = 1_{F(a)}\otimes\varphi_2(a)
 \]
 for some $\varphi_2(a)\in G$ and suppose $f\in A^{h}(a,a')$ and
 $F(f)\in B^{h'}(F(a),F(a'))$. Then we have 
 \begin{eqnarray*}
  \mu(f) & = & f\otimes h \\
  \mu'(F(f)) & = & F(f)\otimes h'
 \end{eqnarray*}
 and the commutativity of the diagram implies 
 \[
  F(f)\otimes (h\varphi_2(a)) = F(f)\otimes (\varphi_2(a')h'). 
 \]

In other words, $F$ restricts to 
 \[
  F : A^{\varphi_2(a')g}(a,a') \longrightarrow
 B^{g\varphi_2(a)}(F(a),F(a')). 
 \]
 This is the definition of degree preserving functor in Definition
 4.1(2) in \cite{0905.3884}.
\end{example}

\begin{definition}
 Let $I$ be a small category. A left morphism
 \[
  (F,\varphi) : (X,\mu) \longrightarrow (X',\mu')
 \]
 in $\rlComod{(I\otimes 1)}$ is called degree-preserving if there exists 
 a family of morphisms in $I$ 
 \[
  \varphi_2(x) \in \Mor_I(p'(F(x)),p(x)) \subset
 \Mor_{\underline{I\otimes 1}}(p'(F(x)), p(x)) =\Mor_{\bm{V}}(1,
 (I\otimes 1)(p'(F(x)),p(x))) 
 \]
 with
 \[
  \varphi(x) = 1\otimes\varphi_2(x).
 \]

 We define degree preserving right morphisms in the same way. 
\end{definition}

\begin{remark}
 We required that $\varphi_2(x)$ is homogeneous in the sense that it
 belongs to $I$ not in $I\otimes 1$.
\end{remark}

\begin{definition}
 \label{2-category_of_graded_category_definition}
 Let $I$ be a small category. The subcategories of the $2$-categories
 $\rlComod{(I\otimes 1)}$ and $\rrComod{(I\otimes 1)}$consisting of
 degree preserving (left or right) morphisms  
 are denoted by $\lcats{\bm{V}}_{I}$ and $\rcats{\bm{V}}_I$ and called
 the $2$-categories of left or right $I$-graded $\bm{V}$-categories,
 respectively. 
\end{definition}

\begin{remark}
 The definition of morphisms of graded categories in \cite{Lowen08}
 requires $\varphi$ to be the identity.
\end{remark}
\subsection{The Grothendieck Construction as a Graded Category}
\label{Gr_as_2-functor}

In this section, we extend the Grothendieck construction for oplax
functors to a $2$-functor
\[
 \Gr : \lOplax(I,\categories{\bm{V}}) \longrightarrow
 \lcats{\bm{V}}_I.
\]

Let us begin with objects.

\begin{definition}
 For an oplax functor
 \[
 X : I \longrightarrow \categories{\bm{V}},
 \]
 define a $\bm{V}$-functor
 \[
 \mu_X : \Gr(X) \longrightarrow \Gr(X)\otimes (I\otimes 1)
 \]
 as follows. On objects,
 \[
 \mu_X : \Gr(X)_0 = \coprod_{i\in I_0} X(i)_0\times\{i\} \longrightarrow
 \Gr(X)_0\otimes I_0
 \]
 is defined by
 \[
  \mu_X(x,i) = (x,i,i).
 \]
 Define, on each component of morphisms, by
 \[
  \begin{diagram}
   \node{\Gr(X)((x,i),(y,j))} \arrow[2]{e,t}{\mu_X} \node{}
   \node{(\Gr(X)\otimes (I\otimes 1))((x,i),(y,j))} \\
   \node{X(j)(X(u)(x),y)}  \arrow{n,J} \arrow{e,t}{\cong}
   \node{X(j)(X(u)(x),y)\otimes (1\otimes 1)} 
   \arrow{e,t}{1\otimes u} \node{X(j)(X(u)(x),y)\otimes (u\otimes 1).}
   \arrow{n,J}
  \end{diagram}
 \]
\end{definition}

Let
\[
 (F,\varphi) : X \longrightarrow X'
\]
be a left morphism of oplax functors. Recall that in Definition
\ref{Gr(F)_definition}, we defined a $\bm{V}$-functor 
\[
 \Gr(F,\varphi) : \Gr(X) \longrightarrow \Gr(X').
\]
In order to obtain a left morphism in $\rlComod{(I\otimes 1)}$, we need a 
$\bm{V}$-natural transformation
\[
 \Gr(\varphi) : \mu_{X'}\circ \Gr(F,\varphi) \Longrightarrow
 (\Gr(F,\varphi)\otimes 1)\circ \mu_{X}. 
\]

\begin{definition}
 For a left morphism of oplax functors
 \[
 (F,\varphi) : X \longrightarrow X',
 \]
 define a $\bm{V}$-natural transformation
 \[
 \Gr(\varphi) : \mu_{X'}\circ \Gr(F) \Longrightarrow
 (\Gr(F)\otimes 1)\circ \mu_{X},
 \]
 namely a morphism
 \begin{multline*}
 \Gr(\varphi)(x,i) : 1 \to (\Gr(X')\otimes(I\otimes
  1))((\mu_{X'}\circ\Gr(F))(x,i), 
  (\Gr(F)\otimes 1)\circ\mu_{X}(x,i)) \\
  = X'(i)(F(x),F(x))\otimes (\Mor_I(i,i)\otimes 1)
 \end{multline*}
 by the identity morphism
 \[
  \Gr(\varphi)(x,i) = 1_{F(x)}\otimes 1_i.
 \]
\end{definition}

\begin{lemma}
 The pair $(\Gr(F,\varphi),\Gr(\varphi))$ defines a left morphism
 \[
  (\Gr(F,\varphi),\Gr(\varphi)) : (\Gr(X),\mu_{X}) \longrightarrow
 (\Gr(X'),\mu_{X'}) 
 \]
 in $\lcats{\bm{V}}_I$.
\end{lemma}

\begin{proof}
 It is easy to check that $(\Gr(F,\varphi),\Gr(\varphi))$ is a
 left morphism in $\rlComod{(I\otimes 1)}$.

 For an object $(x,i)$ in $\Gr(X)$, define
 \[
  \Gr(\varphi)_2(x,i) = 1_i : 1 \longrightarrow (I\otimes 1)(i,i).
 \]
 Then $\Gr(\varphi)(x,i) = 1_{F(x)}\otimes \Gr(\varphi)_2(x,i)$ and we
 have a degree preserving left morphism.
\end{proof}

 Let
 \[
  (F,\varphi),(G,\psi) : X \longrightarrow X'
 \]
 be morphisms of oplax functors.
 For a $2$-morphism
 \[
  \theta : (F,\varphi) \Longrightarrow (G,\psi)
 \]
 in $\lOplax(I,\categories{\bm{V}})$, we have defined a
 $\bm{V}$-natural transformation
 \[
  \Gr(\theta) : \Gr(F,\varphi) \Longrightarrow \Gr(G,\psi)
 \]
in Definition \ref{Gr(theta)}.

\begin{lemma}
 \label{Gr_as_functor_to_comma_categories}
  The above constructions define a functor
 \[
 \Gr : \lOplax(I,\categories{\bm{V}})(X,X')
 \longrightarrow (\lcats{\bm{V}}_I)(\Gr(X),\Gr(X')).   
 \]
\end{lemma}

\begin{proof}
 We need to check the following:
 \begin{enumerate}
  \item  For a $2$-morphism $\theta$ in
	 $\lOplax(I,\categories{\bm{V}})$, $\Gr(\theta)$ is a
	 $2$-morphism in $\lComod{(I\otimes 1)}$, 
	 i.e.\, it makes the following diagram commutative 
	 \[
	 \xymatrix{%
	 \mu_{X'}\circ \Gr(F,\varphi) \ar@{=>}^{\Gr(\varphi)}[r]
	 \ar@{=>}_{\mu_{X'}\circ \Gr(\theta)}[d] &
	 (F\otimes 1)\circ \mu_X \ar@{=>}^{(\Gr(\theta)\otimes 1)\circ
	 \mu_X}[d] \\ 
	 \mu_{X'}\circ \Gr(G,\psi) \ar@{=>}^{\Gr(\psi)}[r] & (G\otimes
	 1)\circ \mu.
	 }
	 \]
  \item $\Gr(1_{(F,\varphi)})=1_{(\Gr(F,\varphi),\Gr(\varphi))}$ for a
	$1$-morphism $(F,\varphi)$ in
	$\lOplax(I,\categories{\bm{V}})$. 
  \item For $2$-morphisms
	\[
	 \xymatrix{%
	(F_0,\varphi_0) \ar@{=>}^{\theta_1}[r] & (F_1,\varphi_1)
	\ar@{=>}^{\theta_2}[r] & (F_2,\varphi_2), 
	}
	\]
	we have
	\[
	 \Gr(\theta_2\circ\theta_1) = \Gr(\theta_2)\circ \Gr(\theta_1).
	\]
 \end{enumerate}

 The first and second parts are obvious from the definition of
 $\Gr(\theta)$. For the 
 third part, recall that both compositions $\theta_2\circ \theta_1$ and
 $\Gr(\theta_2)\circ\Gr(\theta_1)$ are compositions of natural
 transformations. The equality follows from the definition of the
 composition of natural transformations.
\end{proof}

\begin{proposition}
 \label{Gr_as_graded_category}
 The above constructions make the Grothendieck construction into a
 $2$-functor 
 \[
 \Gr : \lOplax(I,\categories{\bm{V}}) \longrightarrow
 \lcats{\bm{V}}_I \subset \rlComod{(I\otimes 1)}.
 \]
\end{proposition}

\begin{proof}
 It remains to prove the following:
 \begin{enumerate}
  \item The following diagram is commutative
	\begin{center}
	 \scalebox{.85}{%
	 $\begin{diagram}
	   \node{\lOplax(I,\categories{\bm{V}})(X',X'')\times
	   \lOplax(I,\categories{\bm{V}})(X,X')} \arrow{e,t}{\circ}
	   \arrow{s,l}{\Gr\times\Gr}
	   \node{\lOplax(I,\categories{\bm{V}})(X,X'')} \arrow{s,r}{\Gr}
	   \\  
	   \node{(\lcats{\bm{V}}_I)(\Gr(X'),\Gr(X''))
	 \times (\lcats{\bm{V}}_I)(\Gr(X),\Gr(X'))} \arrow{e,t}{\circ}
	   \node{(\lcats{\bm{V}}_I)(\Gr(X),\Gr(X'')). } 
	  \end{diagram}$
	 }
	\end{center}
  \item The following diagram is commutative
	\[
	 \begin{diagram}
	  \node{1} \arrow{e} \arrow{se}
	  \node{\lOplax(I,\categories{\bm{V}})(X,X)}
	  \arrow{s,r}{\Gr} \\
	  \node{} \node{(\lcats{\bm{V}}_I)(\Gr(X),\Gr(X)),}  
	 \end{diagram}
	\]
	where $1$ in the diagram is the trivial category.
 \end{enumerate}

 For morphisms of oplax functors
 \[
  X \rarrow{(F,\varphi)} X' \rarrow{(F',\varphi')} X'',
 \]
 the composition $(F',\varphi')\circ(F,\varphi)$ is given by
 \[
  (F',\varphi')\circ(F,\varphi) = (F'\circ F, \varphi'\circ\varphi)
 \]
 and
 \[
  \Gr((F',\varphi')\circ(F,\varphi)) = (\Gr(F'\circ F,
   \varphi'\circ\varphi), \Gr(\varphi'\circ\varphi)).
 \]
 For $(x,i)\in \Gr(X)_0$,
 \[
  \Gr(F'\circ F,\varphi'\circ \varphi)(x,i) = ((F'\circ F)(i)(x),i) =
 (F'(i)(F(i)(x)),i) = \Gr(F')(\Gr(F)(x,i)).
 \]
 For morphisms in $\Gr(X)$, the compatibility of the compositions and
 $\Gr$ follows from the definition of morphisms of oplax functors.

 The second part is obvious.
\end{proof}

We also obtain a $2$-functor
\[
 \Gr : \rLax(I^{\op},\categories{\bm{V}}) \longrightarrow
 \rcats{\bm{V}}_I. 
\]
The details are omitted.

\section{Comma Categories and the Smash Product Construction}
\label{comma_category}

Given a functor
\[
 p : E \longrightarrow I
\]
of small categories, the inverse image $p^{-1}(x)$ of $x\in I_0$
is a natural candidate for the fiber over $x$. There are two more ways
to take fibers, which should be regarded as ``homotopy fibers'', i.e.\
comma categories $p\downarrow x$ and $x\downarrow p$, thanks to the
work of Quillen \cite{Quillen73}.

In this section, we extend definitions of these fibers to enriched
contexts. 

\subsection{Fibers of Gradings}
\label{enriched_fiber}

%Recall that taking fibers is a quasi-inverse to the Grothendieck
%construction for ordinary categories.

%The Grothendieck construction
%takes values in graded categories. Thus we should start with graded
%categories. 

As we have seen in \S\ref{graded_category}, a correct enriched analogue
of a functor 
\[
 p : E \longrightarrow I
\]
is a grading
\[
 \mu : E \longrightarrow E\otimes (I\otimes 1).
\]
In this section, we define fibers of a grading.

\begin{definition}
 Let $I$ be a small category and
 \[
  \mu : E \longrightarrow E\otimes (I\otimes 1)
 \]
 be an $I$-grading. For each object $i\in I_0$, define a
 $\bm{V}$-category $E|_i$ by
 \[
  (E|_i)_0 = \lset{e\in E_0}{\mu(e)=(e,i)}
 \]
 and, for $e,e'\in (E|_i)_0$, $(E|_i)(e,e')$ is defined by the
 following pullback diagram 
 \[
  \begin{diagram}
   \node{(E|_i)(e,e')} \arrow{e} \arrow{s} \node{E(e,e')}
   \arrow{s,r}{p} \\  
   \node{E(e,e')\otimes 1} \arrow{e,t}{1\otimes 1_i}
   \node{E(e,e')\otimes(I(i,i)\otimes 1).} 
  \end{diagram}
 \]
 The composition is defined by the following diagram
 \begin{center}
  \scalebox{.7}{$
  \begin{diagram}
   \node{(E|_i)(e',e'')\otimes (E|_i)(e,e')} \arrow[3]{e}
   \arrow[3]{s} \arrow{se,..} \node{} 
   \node{} \node{E(e',e'')\otimes E(e,e')} \arrow[3]{s} \arrow{sw} \\
   \node{} \node{(E|_i)(e,e'')}\arrow{e} \arrow{s} \node{E(e,e'')}
   \arrow{s} \node{} \\ 
   \node{} \node{E(e,e'')\otimes 1} \arrow{e} \node{E(e,e'')\otimes
   I(i,i)\otimes 1} \node{} \\ 
   \node{E(e',e'')\otimes 1\otimes E(e,e')\otimes 1} \arrow[3]{e}
   \arrow{ne} 
   \node{} \node{} \node{E(e',e'')\otimes I(i,i)\otimes 1 \otimes
   E(e,e')\otimes I(i,i)\otimes 1.} \arrow{nw}
  \end{diagram}$}
 \end{center}
\end{definition}

The following is an analogue of $p\downarrow i$.

\begin{definition}
 Let
 \[
  \mu : E \longrightarrow E\otimes(I\otimes 1)
 \]
 be an $I$-grading and define
 \[
  p = \pr_2\circ \mu_0 : E_0 \longrightarrow E_0\times I_0
 \longrightarrow I_0.
 \]

 For each object $i\in I_0$, define a $\bm{V}$-category
 $\mu\downarrow i$ as follows. Objects are defined by
 \[
  (\mu\downarrow i)_0 = \coprod_{e\in E_0} \{e\}\times \Mor_{I}(p(e),i).
 \]

 For $(e,u), (e',u')\in (\mu\downarrow i)_0$, define an object
 $(\mu\downarrow i)((e,u),(e',u'))$ in $\bm{V}$ by the following pullback
 diagram in $\bm{V}$
 \[
  \begin{diagram}
   \node{(\mu\downarrow i)((e,u),(e',u'))} \arrow{e} \arrow[3]{s}
   \node{E(e,e')} \arrow{s,r}{\mu} \\
   \node{} \node{(E\otimes(I\otimes 1))((e,p(e)),(e',p(e')))} 
   \arrow{s,=} \\
   \node{} \node{E(e,e')\otimes (I\otimes 1)(p(e),p(e'))}
   \arrow{s,r}{1\otimes u'_*} \\
   \node{E(e,e')\otimes 1} \arrow{e,t}{1\otimes u^*}
   \node{E(e,e')\otimes 
   (I\otimes 1)(p(e),i).}
  \end{diagram}
 \]

 In $\mu\downarrow i$, the identity morphisms are defined by
 \[
  \begin{diagram}
   \node{1} \arrow{se,..} \arrow{ese,t}{1} \arrow{sse,b}{1_e} \node{}
   \node{} \\ 
   \node{} \node{(\mu\downarrow i)((e,u),(e,u))} \arrow{e} \arrow{s}
   \node{E(e,e)} \arrow{s,r}{u_*\circ\mu}  \\ 
   \node{} \node{E(e,e)\otimes 1} \arrow{e,t}{u} \node{E(e,e)\otimes
   (I\otimes 1)(p(e),i).} 
   \end{diagram}
 \]

 The composition
 \[
  (\mu\downarrow i)((e',u'),(e'',u''))\otimes (\mu\downarrow
 i)((e,u),(e',u')) 
 \longrightarrow (\mu\downarrow i)((e,u),(e'',u''))
 \]
 is given by the following diagram
  \begin{center}
   \rotatebox{270}{%
  \scalebox{.6}{%
  $\begin{diagram}
   \node{(\mu\downarrow i)((e',u'),(e'',u''))\otimes
    (\mu\downarrow i)((e,u),(e',u'))} \arrow[2]{s} \arrow{se} \arrow[3]{e}
    \node{} 
    \node{} 
   \node{E(e',e'')\otimes E(e,e')} \arrow{sw,t}{\mu\otimes\mu}
    \arrow{s,r}{\circ} \\  
   \node{} \node{(\mu\downarrow i)((e',u'),(e'',u''))\otimes
    (E\otimes (I\otimes 1))((e,p(e)),(e',p(e')))} \arrow{e}
    \arrow{s}  
    \node{(E\otimes (I\otimes 1))((e',p(e')),(e'',p(e'')))
    \otimes (E\otimes (I\otimes 1))((e,p(e)),(e',p(e')))}
    \arrow{se,t}{\circ} \arrow{s,r}{u''_*\otimes 1}
    \node{E(e,e'')} \arrow{s,r}{\mu} \\ 
    \node{(E(e',e'')\otimes 1)\otimes (\mu\downarrow i)((e,u),(e',u'))}
    \arrow{s} \arrow{e}
    \node{(E(e',e'')\otimes 1)\otimes (E\otimes (I\otimes
    1))((e,p(e)),(e',p(e')))} 
    \arrow{se,t}{1\otimes u'_*} 
    \arrow{e,t}{(1\otimes u')\otimes 1}
    \node{(E(e',e'')\otimes (I\otimes 1)(p(e'),i))\otimes
    (E\otimes (I\otimes 1))((e,p(e)),(e',p(e')))} \arrow{sse,t}{\circ}
    \node{E(e,e'')\otimes (I\otimes 1)(p(e),p(e''))}
    \arrow[2]{s,r}{1\otimes u''_*} \\
    \node{(E(e',e'')\otimes 1)\otimes (E(e,e')\otimes 1)}
    \arrow{s,l}{\circ} \arrow[2]{e,t}{1\otimes (1\otimes u)}
    \node{} 
    \node{(E(e',e'')\otimes 1)\otimes (E(e,e')\otimes
    (I\otimes 1)(p(e),i))} \arrow{se,t}{\circ} 
    \node{} \\ 
    \node{E(e,e'')\otimes 1} \arrow[3]{e,t}{1\otimes u} \node{} \node{}
    \node{E(e,e'')\otimes (I\otimes 1)(p(e),i).}
  \end{diagram}$}}
 \end{center}
\end{definition}

\begin{definition}
 Let
 \[
  \mu : E \longrightarrow E\otimes (I\otimes 1)
 \]
 be an $I$-grading and define
 \[
  p = \pr_2\circ\mu : E_0 \longrightarrow E_0\times I_0 \longrightarrow
 I_0. 
 \]
 For each $i\in I_0$, define a $\bm{V}$-category
 $i\downarrow \mu$ as follows. Objects are defined by
 \[
 (i\downarrow \mu)_0 = \coprod_{e\in E_0} \Mor_{I}(i,p(e))\times \{e\}.  
 \]

 Define $(i\downarrow \mu)((e,u),(e',u'))$ by the following
 pullback diagram
 \[
  \begin{diagram}
   \node{(i\downarrow \mu)((u,e),(u',e'))} \arrow{e} \arrow[3]{s}
   \node{E(e,e')} \arrow{s,r}{\mu} \\
   \node{} \node{(E\otimes(I\otimes 1))((e,p(e)),(e',p(e')))} 
   \arrow{s,=} \\
   \node{} \node{E(e,e')\otimes (I\otimes 1)(p(e),p(e'))}
   \arrow{s,r}{1\otimes u^*} \\
   \node{E(e,e')\otimes 1} \arrow{e,t}{1\otimes u'_*}
   \node{E(e,e')\otimes (I\otimes 1)(i,p(e')).}
  \end{diagram}
 \]

 The identities and compositions are defined analogously to the case of
 $\mu\downarrow i$.
\end{definition}

It follows from the fact that $\mu$ is a functor satisfying
coassociativity and counitality that $\mu\downarrow i$ is a
$\bm{V}$-category for each $i\in I_0$. 

\begin{remark}
 $(\mu\downarrow i)((e,u),(e',u'))$ can be also defined by an equalizer 
 \[
  \xymatrix{
 (\mu\downarrow i)((e,u),(e',u')) \ar[r] & E(e,e')
 \ar@<1ex>[r]^(.34){u_*\circ\mu} 
 \ar@<-1ex>[r]_(.34){u} & 
 E(e,e')\otimes (I\otimes 1)(p(e),i)}
 \]
\end{remark}

These ``homotopy fibers'' define functors on $I$.

\begin{definition}
 Let $E$ be a right $I$-graded category. Then define
 \[
  \overleftarrow{\Gamma}(\mu) : I \longrightarrow \categories{\bm{V}}
 \]
 as follows. For $i\in I_0$
 \[
  \overleftarrow{\Gamma}(\mu)(i) = \mu\downarrow i.
 \]
 For a morphism $u : i\to i'$ in $I$, define a
 $\bm{V}$-functor 
 \[
 \overleftarrow{\Gamma}(\mu)(u) : \overleftarrow{\Gamma}(\mu)(i) \longrightarrow
 \overleftarrow{\Gamma}(\mu)(i') 
 \]
 as follows. For an object $(e,v)$ in
 $\overleftarrow{\Gamma}(\mu)(i)=\mu\downarrow i$, define
 \[
  \overleftarrow{\Gamma}(\mu)(u)(e,v) = (e,u\circ v).
 \]
 The morphism
 \[
  \overleftarrow{\Gamma}(\mu)(u) : \overleftarrow{\Gamma}(i)((e,v),(e',v'))
 \longrightarrow \overleftarrow{\Gamma}(\mu)(i')((e,u\circ v),(e',u\circ v'))
 \]
 is defined by the following commutative diagram
 \[
 \begin{diagram}
  \node{\overleftarrow{\Gamma}(\mu)(i)((e,v),(e',v'))} \arrow{e} \arrow[2]{s}
  \node{E(e,e')} \arrow{s,r}{\mu} \\
  \node{} \node{E(e,e')\otimes (I\otimes 1)(p(e),p(e'))}
  \arrow{s,r}{1\otimes v'_*}
  \arrow{sse,t}{1\otimes (u\circ v')_*} \\
  \node{E(e,e')\otimes 1} \arrow{ese,b}{1\otimes (u\circ v)}
  \arrow{e,t}{1\otimes v}
  \node{E(e,e')\otimes (I\otimes 1)(p(e),i)} 
  \arrow{se,t}{1\otimes u_*} \\
  \node{} \node{} \node{E(e,e')\otimes (I\otimes 1)(p(e),i').}
 \end{diagram}
 \]

 Dually, for an $I$-graded category
 \[
  \mu : E \longrightarrow E\otimes (I\otimes 1),
 \]
 define
 \[
  \overrightarrow{\Gamma}(\mu) : I^{\op} \longrightarrow \categories{\bm{V}}
 \]
 by
 \[
  \overrightarrow{\Gamma}(\mu)(i) = i\downarrow\mu
 \]
 on objects and 
 \[
  \overrightarrow{\Gamma}(\mu)(u)(e,v) = (e,v\circ u).
 \]
 for a morphism $u : i'\to i$.
\end{definition}

\begin{lemma}
 $\overleftarrow{\Gamma}(\mu)$ and $\overrightarrow{\Gamma}(\mu)$ are functors.
\end{lemma}

\begin{proof}
 Let us check $\overleftarrow{\Gamma}(\mu)$ is a functor. The case of
 $\overrightarrow{\Gamma}(\mu)$ is analogous. For an identity morphism
 \[
  1_i : i \longrightarrow i
 \]
 in $I$,  the morphism
 \[
  \overleftarrow{\Gamma}(\mu)(1_i) : \overleftarrow{\Gamma}(\mu)(i)((e,v),(e',v'))
 \longrightarrow \overleftarrow{\Gamma}(\mu)(i)((e,v),(e',v')) 
 \]
 is defined by the following diagram
 \[
 \begin{diagram}
  \node{\overleftarrow{\Gamma}(\mu)(i)((e,v),(e',v'))} \arrow{e} \arrow[2]{s}
  \node{E(e,e')} \arrow{s,r}{\mu} \\
  \node{} \node{E(e,e')\otimes (I\otimes 1)(p(e),p(e'))}
  \arrow{s,r}{1\otimes v'_*}
  \arrow{sse,t}{1\otimes (1_i\circ v')_*} \\
  \node{E(e,e')\otimes 1} \arrow{ese,b}{1\otimes (1_i\circ v)}
  \arrow{e,t}{1\otimes v}
  \node{E(e,e')\otimes (I\otimes 1)(p(e),i)} 
  \arrow{se,t,=}{1\otimes (1_i)_*=1} \\
  \node{} \node{} \node{E(e,e')\otimes (I\otimes 1)(p(e),i).}
 \end{diagram}
 \]
 It follows that $\overleftarrow{\Gamma}(\mu)(1_i)$ is the identity morphism.

 For a composable morphisms
 \[
  i_0 \rarrow{u_1} i_1 \rarrow{u_2} i_2
 \]
 in $I$, the composition
 \[
  \overleftarrow{\Gamma}(\mu)(u_2)\circ \overleftarrow{\Gamma}(\mu)(u_1) :
 \overleftarrow{\Gamma}(\mu)(i_0)((i,v),(i',v')) 
 \longrightarrow \overleftarrow{\Gamma}(\mu)(i_2)((e,u_2\circ u_1\circ v),
 (e',u_2\circ u_1\circ v'))
 \]
 satisfies the universality of the pullback and we have
 \[
  \overleftarrow{\Gamma}(\mu)(u_2)\circ \overleftarrow{\Gamma}(\mu)(u_1) =
 \overleftarrow{\Gamma}(\mu)(u_2\circ u_1). 
 \]
\end{proof}

\begin{example}
 Consider the case $I$ is a group $G$ and $\bm{V}=\lMod{k}$. For a
 $G$-graded category
 \[
  \mu : A \longrightarrow A\otimes k[G]
 \]
 we obtain a functor
 \[
  \overleftarrow{\Gamma}(\mu) : G \longrightarrow \categories{k}.
 \]
 Since $G$ has a single object $\ast$, $\overleftarrow{\Gamma}(\mu)$ is
 determined by the $k$-linear category
 $\overleftarrow{\Gamma}(\mu)(\ast)$ and an 
 action of $G$ on it. 

 The $k$-linear category $\overleftarrow{\Gamma}(\mu)(\ast)$ has objects 
 \[
 \overleftarrow{\Gamma}(\mu)(\ast)_0 = \lset{(x,g)}{x\in A_0, g \in
 \Mor_{G}(\ast,\ast)} =  A_0\times G. 
 \]
 The $k$-module of morphisms $\overleftarrow{\Gamma}(\mu)(\ast)((x,g),(y,h))$ is
 given by the pullback diagram
 \[
 \begin{diagram}
 \node{} \node{\overleftarrow{\Gamma}(\mu)(\ast)((x,g),(y,h))}
  \arrow[2]{s} \arrow{e} 
 \node{A(x,y)} \arrow{s,r}{\mu} \\
 \node{} \node{} \node{A(x,y)\otimes k[G]} \arrow{s,r}{h_*} \\
 \node{} \node{A(x,y)\otimes 1} \arrow{e,t}{1\otimes g}
  \node{A(x,y)\otimes k[G].}
 \end{diagram}
 \]
 Since $\mu$ is $G$-grading, we have a coproduct decomposition
 \[
  A(x,y) = \bigoplus_{g\in G} A^g(x,y)
 \]
 and, for $f\in A^g(x,y)$, $\mu(f)$ is given by
 \[
  \mu(f) = f\otimes g.
 \]
 Thus we have an identification
 \[
  \overleftarrow{\Gamma}(\mu)(\ast)((x,g),(y,h)) \cong A^{h^{-1}g}(x,y)
 \]
 for $g,h\in G$. The is the smash product construction in
 \cite{math/0312214,0807.4706,0905.3884}. 
\end{example}

Let us extend $\overleftarrow{\Gamma}$ and $\overrightarrow{\Gamma}$ as
$2$-functors 
\begin{eqnarray*}
 \overleftarrow{\Gamma} & : & \lcats{\bm{V}}_I \longrightarrow
  \lOplax(I,\categories{\bm{V}}) \\
 \overrightarrow{\Gamma} & : & \rcats{\bm{V}}_I \longrightarrow
  \rLax(I^{\op},\categories{\bm{V}}).
\end{eqnarray*}

\begin{definition}
 Let 
 \[
 (F,\varphi) : (E,\mu) \longrightarrow (E',\mu').
 \]
 be a morphism in $\lcats{\bm{V}}_I$, i.e.\ a degree-preserving
 left morphism in $\rlComod{(I\otimes 1)}$. $\varphi$ can be written as
 \[
  \varphi(e) = 1\otimes \varphi_2(e)
 \]
 for $\varphi_2(e)\in \Mor_I(p'(F(e)),p(e))$.

 Define a left morphism of oplax functors  
 \[
  (\Gamma(F,\varphi),\Gamma(\varphi)) : \overleftarrow{\Gamma}(\mu)
 \longrightarrow \overleftarrow{\Gamma}(\mu') 
 \]
 as follows. For each $i\in I_0$, define a $\bm{V}$-functor
 \[
 \overleftarrow{\Gamma}(F,\varphi)(i) : \overleftarrow{\Gamma}(\mu)(i) \longrightarrow
 \overleftarrow{\Gamma}(\mu')(i) 
 \]
 by
 \[
  \overleftarrow{\Gamma}(F,\varphi)(i)(e,v) = (F(e),v\circ\varphi_2(e))
 \]
 for an object $(e,v)\in \overleftarrow{\Gamma}(\mu)(i)_0$ and 
 \[
  \overleftarrow{\Gamma}(F,\varphi)(i) : \overleftarrow{\Gamma}(\mu)(i)((e,v),(e',v'))
 \longrightarrow 
 \overleftarrow{\Gamma}(\mu')(i)((F(e),v\circ\varphi_2(e)),
 (F(e'),v'\circ\varphi_2(e'))) 
 \]
 is defined by the commutativity of the following diagram
 \begin{center}
  \scalebox{.9}{%
  $\begin{diagram}
   \node{\overleftarrow{\Gamma}(\mu)(i)((e,v),(e',v'))} \arrow{e} \arrow[2]{s}
    \node{E(e,e')} \arrow{s,r}{\mu}
   \arrow{e,t}{F} \node{E'(F(e),F(e'))} \arrow{s,r}{\mu'} \\ 
    \node{} \node{E(e,e')\otimes (I\otimes 1)(p(e),p(e'))}
    \arrow{s,l}{1\otimes v'_*}
    \arrow{se,t}{F\otimes \varphi(i)^*}
    \node{E'(F(e),F(e'))\otimes (I\otimes 1)(p'(F(e)),p'(F(e')))} 
   \arrow{s,r}{\varphi(i')_*} \\
   \node{E(e,e')\otimes 1} \arrow{s,l}{F} \arrow{e,t}{1\otimes v}
    \node{E(e,e')\otimes 
   (I\otimes 1)(p(e),i)} \arrow{se,b}{F\otimes \varphi_1(i)^*}
   \node{E'(F(e),F(e'))\otimes (I\otimes 
   1)(p'(F(e)),p(e'))}  \arrow{s,r}{1\otimes v'_*} \\
   \node{E'(F(e),F(e'))\otimes 1} \arrow[2]{e,t}{1\otimes
    v\circ\varphi_1(e)} \node{} 
   \node{E'(F(e),F(e'))\otimes (I\otimes 1)(p'(F(e)),i),}
  \end{diagram}$}
 \end{center}
 where the commutativity of the top right hexagon follows from the
 naturality of $\varphi$.

 For each $u : i\to i'$ in $I$ and $(e,v)\in \overleftarrow{\Gamma}(\mu)(i)_0$,
 define 
 \begin{multline*}
  \overleftarrow{\Gamma}(\varphi)(u)(e,v) : 1 \longrightarrow
 \overleftarrow{\Gamma}(\mu')(i')((\overleftarrow{\Gamma}(\mu')(u)\circ
  \overleftarrow{\Gamma}(F,\varphi))(e,v), 
 (\overleftarrow{\Gamma}(F,\varphi)(i')\circ \overleftarrow{\Gamma}(\mu)(u))(e,v)) \\
  = \overleftarrow{\Gamma}(\mu')(i')((F(e),u\circ
  (v\circ\varphi_2(e))),(F(e),(u\circ 
  v)\circ\varphi_2(e))) 
 \end{multline*}
 to be the identity.
\end{definition}

\begin{definition}
 Let
 \[
  (F,\varphi), (G,\psi) : (E,\mu) \longrightarrow (E',\mu')
 \]
 be left morphisms of $I$-graded $\bm{V}$-categories. For a $2$-morphism 
 \[
  \xi : (F,\varphi) \Longrightarrow (G,\psi)
 \]
 in $\lcats{\bm{V}}_I$, define a $2$-morphism in
 $\lOplax(I,\categories{\bm{V}})$ 
 \[
  \overleftarrow{\Gamma}(\xi) : (\overleftarrow{\Gamma}(F,\varphi),\overleftarrow{\Gamma}(\varphi))
 \Longrightarrow 
 (\overleftarrow{\Gamma}(G,\psi),\overleftarrow{\Gamma}(\psi)), 
 \]
 i.e.\ a $\bm{V}$-natural transformation
 \[
  \overleftarrow{\Gamma}(\xi) : \overleftarrow{\Gamma}(F,\varphi) \Longrightarrow
 \overleftarrow{\Gamma}(G,\psi) 
 \]
 by the following diagram
 \[
 \begin{diagram}
  \node{1} \arrow{se,t,..}{\overleftarrow{\Gamma}(\xi)(e,v)}
  \arrow[2]{e,t}{\xi(e)} 
  \arrow[3]{s,b}{\xi(e)} \node{} \node{E'(F(e),G(e))}
  \arrow[2]{s,r}{\mu} \\ 
  \node{}
  \node{\overleftarrow{\Gamma}(\mu')((F(e),v\circ\varphi_2(e)),
  (G(e),v\circ\psi_2(e)))}  
  \arrow{ne} \arrow{ssw} \node{} \\ 
  \node{} \node{} \node{E'(F(e),G(e))\otimes (I\otimes
  1)(p'(F(e)),p'(G(e)))} \arrow{s,r}{1\otimes (v\circ \psi_2(e))_*} \\ 
  \node{E'(F(e),G(e))\otimes 1} \arrow[2]{e,t}{1\otimes v\circ
  \varphi_2(e)} \node{}
  \node{E'(F(e),G(e))\otimes (I\otimes 1)(p'(F(e)),i).}
 \end{diagram}
 \]
\end{definition}

\begin{lemma}
 \label{Gamma_as_functor}
 The above constructions define a functor
 \[
  \overleftarrow{\Gamma} : \lcats{\bm{V}}_I((\mu),(\mu'))
 \longrightarrow \lOplax(I,
 \categories{\bm{V}})(\overleftarrow{\Gamma}(\mu),\overleftarrow{\Gamma}(\mu')).   
 \]
\end{lemma}

\begin{proof}
 By the universality of pullback, the identity natural transformation
 induces the identity. For composable $2$-morphisms
 \[
  \xymatrix{%
 (F_0,\varphi_0) \ar@{=>}[r]^{\xi_1} & (F_1,\varphi_1)
 \ar@{=>}[r]^{\xi_2} & (F_2,\varphi_2),
 }
 \]
consider the following diagram
\begin{center}
 \scalebox{.8}{
$\begin{diagram}
   \node{1} \arrow[2]{e,t}{(\xi_2\circ\xi_1)(e)}
  \arrow[3]{s,l}{(\xi_2\circ\xi_1)(e)} 
  \arrow{se,t}{\overleftarrow{\Gamma}(\xi_2)(e,v)\otimes\Gamma(\xi_1)(e,v)} 
   \node{} \node{E'(F_0(e),F_2(e))} \arrow[2]{s,=} \\ 
   \node{} \node{\overleftarrow{\Gamma}(E')((F_1(e),v\circ\varphi_{12}(e)),
   (F_2(e),v\circ\varphi_{22}(e)))
  \otimes\overleftarrow{\Gamma}(E')((F_0(e),v\circ\varphi_{02}(e)),
  (F_1(e),v\circ\varphi_{12}(e)))}
   \arrow{s,r}{} \node{} \\
   \node{}
   \node{\Gamma(E')((F_0(e),v\circ\varphi_{02}(e)),
  (F_2(e),v\circ\varphi_{02}(e)))} \arrow{sw} \arrow{e}
   \node{E'(F_0(e),F_2(e))} \arrow{s} \\  
   \node{E'(F_0(e),F_2(e))\otimes 1} \arrow[2]{e} \node{}
   \node{E'(F_0(e),F_2(e))\otimes (I\otimes 1)(p'(F_0(e)),i).} 
  \end{diagram}$}
\end{center}
 The commutativity of the diagram and the universality of the pullback
 implies that 
 \[
  (\Gamma(\xi_2)\circ\Gamma(\xi_1))(e,v) =
 \Gamma(\xi_2)(e,v)\circ\Gamma(\xi_1)(e,v). 
 \]
\end{proof}

It is not hard to see that $\overleftarrow{\Gamma}$ preserves the identity
left morphisms and compositions of left morphisms.

\begin{proposition}
 \label{Gamma_as_2-functor}
 We obtain a $2$-functor
 \[
  \overleftarrow{\Gamma} : \lcats{\bm{V}}_I \longrightarrow
 \lOplax(I,\categories{\bm{V}}). 
 \]
\end{proposition}

\begin{proof}
 Strict $2$-categories are categories enriched over $\Categories$. By
 definition of enriched functors (Definition \ref{V-functor}), it
 suffices to prove that the following diagrams are commutative in 
 $\Categories$. 
 \[
  \begin{diagram}
   \node{1} \arrow{e,t}{1} \arrow{se,b}{1}
   \node{\lcats{\bm{V}}_I((\mu),(\mu))}
   \arrow{s,r}{\overleftarrow{\Gamma}} \\ 
   \node{}
   \node{\lOplax(I,\categories{\bm{V}})(\overleftarrow{\Gamma}(\mu),
   \overleftarrow{\Gamma}(\mu))}   
  \end{diagram}
 \]
 \begin{center}
  \scalebox{.7}{
  $\begin{diagram}
   \node{\lcats{\bm{V}}_I(\mu',\mu'')\times
   \lcats{\bm{V}}_I(\mu,\mu')} \arrow{e,t}{\circ}
    \arrow{s,l}{\overleftarrow{\Gamma}\times\overleftarrow{\Gamma}} 
   \node{\lcats{\bm{V}}_I(\mu,\mu'')}
    \arrow{s,r}{\overleftarrow{\Gamma}} \\  
   \node{\lOplax(I,\categories{\bm{V}})(\overleftarrow{\Gamma}(\mu'),
    \overleftarrow{\Gamma}(\mu''))\times  
   \lOplax(I,\categories{\bm{V}})(\overleftarrow{\Gamma}(\mu),
    \overleftarrow{\Gamma}(\mu'))}  
   \arrow{e,t}{\circ}
   \node{\lOplax(I,\categories{\bm{V}})(\overleftarrow{\Gamma}(\mu),
    \overleftarrow{\Gamma}(\mu''))}  
  \end{diagram}$}
 \end{center}

 The commutativity of the first diagram is obvious. It is tedious but
 straightforward to check the commutativity of
 the second diagram.
\end{proof}

Dually we have 

\begin{proposition}
 We obtain a $2$-functor
 \[
  \overrightarrow{\Gamma} : \rcats{\bm{V}}_I \longrightarrow
 \rLax(I^{\op},\categories{\bm{V}}). 
 \]
\end{proposition}

\subsection{Fibered and Cofibered Categories}
\label{fibered_category}

We have seen in \S\ref{enriched_fiber} that there are three ways to take
a fiber over an object $i\in I_0$ for an $I$-graded category
\[
 \mu : E \longrightarrow E\otimes (I\otimes 1).
\]
We have also seen that two of them, $\mu\downarrow i$ and
$i\downarrow \mu$ can be extended to $2$-functors
\begin{eqnarray*}
 \overleftarrow{\Gamma} & : & \lcats{\bm{V}}_I \longrightarrow
 \lOplax(I,\categories{\bm{V}}), \\ 
 \overrightarrow{\Gamma} & : & \rcats{\bm{V}}_I \longrightarrow
 \rLax(I^{\op},\categories{\bm{V}}). 
\end{eqnarray*}

In the case of non-enriched categories, we need the notions of fibered
and cofibered categories introduced by 
Grothendieck in order to extend the
remaining construction $E|_i$ to a $2$-functor. A definition of fibered
$I$-graded category is introduced for $k$-linear categories
by Lowen \cite{Lowen08} recently. We reformulate Lowen's definition in
order to incorporate it into our definition of graded categories.

The idea of Grothendieck is to compare $\Gamma_{\cof}(\mu)(i)$ and
$E|_i$ for a given $I$-graded category 
\[
 \mu : E \longrightarrow E\otimes (I\otimes 1).
\]

\begin{definition}
 Let $\mu : E \to E\otimes (I\otimes 1)$ be an $I$-graded
 category. Define 
 \[
  i_i : E|_i \longrightarrow \mu\downarrow i
 \]
 as follows. For $e\in (E|_i)_0$, 
 \[
  i_i(e) = (e,1_i).
 \]
 For $e,e'\in (E|_i)_0$, 
 \[
  i_i : (E|_i)(e,e') \longrightarrow
 (\mu\downarrow i)((e,1_i),(e',1_i)) 
 \]
 is defined by the following diagram
 \[
  \begin{diagram}
   \node{(E|_i)(e,e')} \arrow{ese} \arrow{sse} \arrow{se,..} \node{}
   \node{} \\ 
   \node{} \node{(\mu\downarrow i)((e,1_i),(e',1_i))} \arrow{e}
   \arrow{s} 
   \node{E(e,e')} \arrow{s,r}{\mu} \\
   \node{} \node{E(e,e')\otimes 1} \arrow{e,t}{1\otimes 1_i}
   \node{E(e,e')\otimes I(i,i)\otimes 1.}
  \end{diagram}
 \]

 Similarly we define
 \[
  j_i : E|_i \longrightarrow i\downarrow \mu
 \]
 by
 \[
  j_i(e) = (e,1_i)
 \]
 on objects.
\end{definition}

\begin{lemma}
 For each $i\in I_0$,
 \begin{eqnarray*}
  i_i & : & E|_i \longrightarrow \mu\downarrow i, \\
  j_i & : & E|_i \longrightarrow i\downarrow \mu
 \end{eqnarray*}
 are $\bm{V}$-functors.
\end{lemma}

We define prefibered and precofibered categories in terms of these
functors. 

\begin{definition}
 When $i_i$ has a left adjoint
 \[
 s_i :   \mu\downarrow i \longrightarrow E|_i
 \]
 for each $i\in I_0$, $\mu$ is called
 precofibered. The collection $\{s_i\}_{i\in I_0}$ is called a
 precofibered structure on $\mu$.

 When $j_i$ has a right adjoint 
 \[
  t_i : i\downarrow \mu \longrightarrow E|_i
 \]
 for each $i\in I_0$, $\mu$ is called
 prefibered. The collection $\{t_i\}_{i\in I_0}$ is called a
 prefibered structure on $\mu$.
\end{definition}

We obtain a lax functor from a prefibered category.

\begin{definition}
 Let 
 \[
  \mu : E \longrightarrow E\otimes (I\otimes 1)
 \]
 be a prefibered category. For $i\in I_0$, define
 \[
  \Gamma_{\fib}(\mu)(i) = E|_{i}.
 \]
 For a morphism $u : i \to i'$ in $I$, define
 \[
  \Gamma_{\fib}(\mu)(u) : E|_{i'} \longrightarrow E|_{i}
 \]
 by the composition
 \[
  E|_{i'} \rarrow{j_{i'}} i'\downarrow \mu
 \rarrow{\overrightarrow{\Gamma}(\mu)(u)} i\downarrow \mu 
 \rarrow{t_{i}} E|_{i}.
 \]

 Dually, for a precofibered category
 \[
  \mu : E\longrightarrow E\otimes (I\otimes 1),
 \]
 define
 \[
  \Gamma_{\cof}(\mu)(i) = E|_i
 \]
 for each $i\in I_0$ and define
 \[
  \Gamma_{\cof}(\mu)(u) : E|_i \longrightarrow E|_{i'}
 \]
 for each $u : i \to i'$ by the composition
 \[
  E|_{i} \rarrow{i_{i}} \mu\downarrow i
 \rarrow{\overleftarrow{\Gamma}(\mu)(u)} \mu\downarrow i'
 \rarrow{s_{i'}} E|_{i'}.
 \]
\end{definition}

\begin{lemma}
 When $\mu$ is prefibered
 \[
 \Gamma_{\fib}(\mu) : I^{\op} \longrightarrow \categories{\bm{V}}
 \]
 is a lax functor. When $\mu$ is precofibered
 \[
  \Gamma_{\cof}(\mu) : I \longrightarrow \categories{\bm{V}}
 \]
 is an oplax functor.
\end{lemma}

\begin{proof}
 For a composable morphisms
 $i \rarrow{u} i' \rarrow{u'} i''$, we have the following diagram
 \[
 \xymatrix{
 E|_i & E|_{i'} \ar[l]_{\Gamma_{\fib}(\mu)(u)}
 \ar@<-.4ex>[d]_{j_{i'}} & E|_{i''}
 \ar[l]_{\Gamma_{\fib}(\mu)(u')} 
 \ar[d]^{j_{i''}} \\
 i\downarrow\mu \ar[u]^{t_i} & i'\downarrow\mu
 \ar[l]_{\overrightarrow{\Gamma}(\mu)(u)} 
 \ar@<-.4ex>[u]_{t_{i'}} &
 i''\downarrow\mu. \ar[l]_{\overrightarrow{\Gamma}(\mu)(u')}
 }
 \]
 By definition, we have a natural transformation
 \[
 \varepsilon_{i'} : j_{i'}\circ t_{i'} \Longrightarrow
 1_{i'\downarrow\mu} 
 \]
 for each $i\in I_0$. Define
 \[
  \theta_{u',u} : \Gamma_{\fib}(\mu)(u)\circ \Gamma_{\fib}(\mu)(u')
 \Longrightarrow \Gamma_{\fib}(\mu)(u'\circ u)
 \]
 by using $\varepsilon_{i'}$
 \begin{multline*}
 \Gamma_{\fib}(\mu)(u)\circ \Gamma_{\fib}(\mu)(u') =
 (t_i\circ\overrightarrow{\Gamma}(\mu)(u)\circ j_{i'})\circ (t_{i'}\circ
 \overrightarrow{\Gamma}(\mu)(u)\circ j_{i''}) \\
  \Longrightarrow t_i\circ\overrightarrow{\Gamma}(\mu)(u)\circ 
 \overrightarrow{\Gamma}(\mu)(u')\circ j_{i''} =
  t_i\circ\overrightarrow{\Gamma}(\mu)(u'\circ u)\circ j_{i''} =
  \Gamma_{\fib}(\mu)(u'\circ u).
 \end{multline*}
 Then we obtain a lax functor.

 Similarly, if $\mu$ is precofibered, define
 \[
  \theta_{u',u} : \Gamma_{\cof}(\mu)(u'\circ u) \Longrightarrow
 \Gamma_{\cof}(\mu)(u')\circ\Gamma_{\cof}(\mu)(u) 
 \]
 by using 
 \[
  \eta_{i'} : 1_{E|_{i'}} \longrightarrow i_{i'}\circ s_{i'}.
 \]
 And we obtain an oplax functor.
\end{proof}

\begin{definition}
 \label{fibered_category_definition}
 We say a prefibered $I$-graded category
 $\mu : E \to E\otimes (I\otimes 1)$ is fibered if the above natural
 transformations $\varepsilon_{i'}$, hence $\theta_{u',u}$, are
 isomorphisms. A precofibered 
 category is called cofibered if $\eta_{i'}$ are all natural
 isomorphisms. 
\end{definition}

We would like to define morphisms of prefibered and precofibered
categories corresponding to morphisms of lax and oplax functors under
$\Gamma_{\fib}$ and $\Gamma_{\cof}$.

\begin{definition}
Let
 \begin{eqnarray*}
  \mu & : & E \longrightarrow E\otimes(I\otimes 1) \\
  \mu' & : & E' \longrightarrow E'\otimes(I\otimes 1) \\
 \end{eqnarray*}
 be prefibered graded categories over $I$. A morphism of prefibered
 graded categories from $\mu$ to $\mu'$ is a right morphism
 \[
  (F,\varphi) : (E,\mu) \longrightarrow (E',\mu')
 \]
 of $I$-graded categories. $2$-morphisms are also $2$-morphisms in
 $\rcats{\bm{V}}_I$. 
\end{definition}

\begin{definition}
 Let
\[
  (F,\varphi) : (E,\mu) \longrightarrow (E',\mu')
 \]
 be a morphism of prefibered $I$-graded categories.
 Define a right morphism of lax functors
 \[
  (\Gamma_{\fib}(F,\varphi),\Gamma_{\fib}(\varphi)) : \Gamma_{\fib}(\mu)
 \longrightarrow \Gamma_{\fib}(\mu')
 \]
 as follows. For an object $i\in I_0$, we need to define a functor
 \[
  \Gamma_{\fib}(F,\varphi)(i) : E|_{i} \longrightarrow E'|_{i}.
 \]
 For $e \in (E|_i)_0$, we have
 \[
  \varphi(e) : i = p(e) \longrightarrow p'(F(e)).
 \]
 In other words, $(F(e),\varphi(e))\in (i\downarrow\mu')_0$. By applying
 \[
  t'_i : i\downarrow\mu' \longrightarrow E'|_{i}
 \]
 we obtain
 \[
  \Gamma_{\fib}(F,\varphi)(i)(e) = t'_i(F(e),\varphi(e)) \in
 (E'|_i)_0.
 \]
 For a pair of objects 
 $e,e' \in (E|_i)_0$, define
 \[
 (E|_i)(e,e') \longrightarrow
 (E'|_i)(\Gamma_{\fib}(F,\varphi)(i)(e),\Gamma_{\fib}(F,\varphi)(i)(e')) 
 \]
 by the composition
 \[
  (E|_i)(e,e') \rarrow{\tilde{F}} (i\downarrow
 E')((F(e),\varphi(e)),(F(e'),\varphi(e'))) \rarrow{t'_i}
 (E'|_i)(t'_i(F(e),\varphi(e)), t'_i(F(e'),\varphi(e'))),
 \]
 where $\tilde{F}$ is defined by the following diagram
 \[
 \begin{diagram}
 \node{E(e,e)} \arrow{se,..} \arrow{ese,t}{F} \arrow{sse,b}{F} \node{} \\
 \node{} \node{(i\downarrow
 E')((F(e),\varphi(e)),(F(e'),\varphi(e')))} \arrow{e} \arrow{s} 
 \node{E'(F(e),F(e'))} \arrow{s} \\
 \node{} \node{E'(F(e),F(e'))} \arrow{e} \node{E'(F(e),F(e'))\otimes
  (I\otimes 1)(i,p'(F(e'))).}
 \end{diagram}
 \]
 The commutativity of the outside square comes from the naturality of
 $\varphi$. 

 We define $\Gamma_{\fib}(\varphi)$ for
 \[
 \xymatrix{%
 E|_{i'} \ar[rr]^{\Gamma_{\fib}(F,\varphi)(i')}
 \ar[d]_{\Gamma_{\fib}(\mu)(u)} & & E'|_{i'} 
 \ar[d]^{\Gamma_{\fib}(\mu)(u)} \\ 
 E|_{i} \ar[rr]_{\Gamma_{\fib}(F,\varphi)(i)}
 \ar@{=>}[urr]^{\Gamma_{\fib}(\varphi)} & & E'|_{i} 
 }
 \]
 as follows. The adjunction $\varepsilon_i$ induces
 \[
  F(t_i(e,u)) \longrightarrow F(e),
 \]
 which in turn induces
 \[
  j_i(F(t_i(e,u))) \longrightarrow (F(e),\varphi(e))
 \]
 or
 \[
  F(t_i(e,u)) \longrightarrow t_i(F(e),\varphi(e))
 \]
 and we obtain a morphism
 \[
  (F(t_i(e,u)),\varphi(t_i(e,u))) \longrightarrow
 (t'_i(F(e),\varphi(e)),u) 
 \]
 in $i\downarrow\mu'$. By applying $t'_{i}$, we obtain
 $\Gamma_{\fib}(\varphi)$. 
\end{definition}

For $2$-morphisms we define as follows.

\begin{definition}
 Let
 \[
  (F,\varphi), (G,\psi) : (E,\mu) \longrightarrow (E',\mu')
 \]
 be right morphisms of prefibered $I$-graded categories. For a
 $2$-morphism 
 \[
  \xi : (F,\varphi) \Longrightarrow (G,\psi),
 \]
 define a morphism of right transformations
 \[
 \Gamma_{\fib}(\xi) : (\Gamma_{\fib}(F,\varphi),\Gamma_{\fib}(\varphi))
 \Longrightarrow 
 (\Gamma_{\fib}(G,\psi), \Gamma_{\fib}(\psi))
 \]
 as follows. For $i\in X_0$, we need to define a $\bm{V}$-natural
 transformation
 \[
  \Gamma_{\fib}(\xi)(i) : \Gamma_{\fib}(F,\varphi)(i) \Longrightarrow
 \Gamma_{\fib}(G,\psi)(i). 
 \]
 For each object $e$ in $E|_i$, we have
 \[
  \xi(e) : F(e) \longrightarrow G(e).
 \]
 We also have a commutative diagram
 \[
  \begin{diagram}
   \node{x} \arrow{s,=} \arrow{e,=} \node{p(e)} \arrow{e,t}{\varphi(e)}
   \node{p'(F(e))} \arrow{s,r}{p'(\xi(e))} \\
   \node{x} \arrow{e,=} \node{p(e)} \arrow{e,b}{\psi(e)} \node{p'(G(e))}
  \end{diagram}
 \]
 and we have a morphism
 \[
  \xi(e) : (F(e),\varphi(e)) \longrightarrow (G(e),\psi(e))
 \]
 in $i\downarrow \mu'$ and we obtain
 \[
  \Gamma_{\fib}(\xi)(i)(e) \in
 E'(t'_i(F(e),\varphi(e)),t'_i(G(e),\psi(e))) =
 (E'|_{i})(\Gamma_{\fib}(F,\varphi)(i)(e),\Gamma_{\fib}(G,\psi)(i)(e)). 
 \]
 It is not difficult to check that $\Gamma_{\fib}(\xi)$ satisfies the
 condition for a morphism of right transformation.
\end{definition}

The above definitions can be dualized and we obtain a $2$-functor
$\Gamma_{\cof}$. 

\begin{proposition}
 We have $2$-functors
 \begin{eqnarray*}
  \Gamma_{\fib} & : & \category{Prefibered}(I) \longrightarrow
 \rLax(I^{\op},\categories{\bm{V}}),  \\
  \Gamma_{\cof} & : & \category{Precofibered}(I) \longrightarrow
   \lOplax(I,\categories{\bm{V}}), 
 \end{eqnarray*}
 whose behaviors on objects are given by $\Gamma_{\fib}(\mu)$ and
 $\Gamma_{\cof}(\mu)$, respectively.
\end{proposition}

\section{The Grothendieck Construction as Left Adjoints}
\label{adjunctions}

In \S\ref{Gr_as_2-functor}, we have seen that the Grothendieck
construction defines a $2$-functor
\[
 \Gr : \lOplax(I,\categories{\bm{V}}) \longrightarrow
 \lcats{\bm{V}}_I \subset \rlComod{(I\otimes 1)}.
\]
We have constructed $2$-functors
\[
 \lGamma : \lcats{\bm{V}}_I \longrightarrow
 \lFunct(I,\categories{\bm{V}}) \subset \lOplax(I,\categories{\bm{V}}) 
\]
and
\[
 \Gamma_{\cof} : \category{Cofibered}(I) \longrightarrow
 \lOplax(I,\categories{\bm{V}}) 
\]
in \S\ref{comma_category}.

The aim of this section is to show that $\Gr$ is left adjoint to
$\overleftarrow{\Gamma}$. We also show that a restriction of $\Gr$ to
a certain $2$-subcategory is left adjoint to $\Gamma_{\cof}$ in a weak
sense. We, of course, have ``lax versions'' of these adjunctions.

By composing the forgetful functor, we obtain
\[
 \Gr : \lOplax(I,\categories{\bm{V}}) \longrightarrow
 \categories{\bm{V}}, 
\]
which is the $2$-functor $\Gr$ defined in
\S\ref{definition_for_oplax_functor}. We show that this $2$-functor also
has a right adjoint in \S\ref{left_adjoint_to_diagonal}.

\subsection{The Grothendieck Construction and Graded Categories}
\label{Gr_and_pi}

In this section, we prove that $\lGamma$ is right adjoint to $\Gr$.

\begin{theorem}
 \label{Gr_and_Gamma_are_adjoint}
 For an oplax functor $X: I \to \categories{\bm{V}}$ and an $I$-graded
 category $\mu : E \to E\otimes (I\otimes 1)$, we
 have an isomorphism of categories 
 \[
 (\lcats{\bm{V}}_I)((\Gr(X),\mu_X),(E,\mu))\cong
 \lOplax(I,\categories{\bm{V}})(X,\lGamma(\mu)). 
 \]
\end{theorem}

We need to define morphisms
\begin{eqnarray*}
 \eta_X & : & X \longrightarrow \lGamma(\mu_X), \\
 \varepsilon_{\mu} & : & \Gr(\lGamma(\mu)) \longrightarrow \mu,
\end{eqnarray*}
in $\lOplax(I,\categories{\bm{V}})$ and
$\lcats{\bm{V}}_I$, respectively.

Let us first define $\eta_X$. We need to define a family of functors
\[
 \eta_X(i) : X(i) \longrightarrow \lGamma(\mu_X)(i)
\]
indexed by $i\in I_0$ and a family of natural transformations
\[
\eta_X(u)  : \lGamma(\mu_X)(u)\circ \eta_X(i) \Longrightarrow
\eta_X(j)\circ X(u) 
\]
indexed by morphisms $u : i\to j$ in $I$.

Objects of $\lGamma(\mu_X)(i)$ are given by
\[
 \lGamma(\mu_X)(i)_0 = \coprod_{(x,j)\in \Gr(X)_0} \{(x,j)\}\times
 \Mor_{I}(j,i)
\]
and we have a canonical inclusion
\[
 \eta_X(i) : X(i)_0 \hookrightarrow X(i)_0\times\{i\}\times \{1_i\}
 \subset \lGamma(\mu_X)(i)_0. 
\]
For $x,x'\in X(i)_0$, 
\[
 \lGamma(\mu_X)(i)(\eta_X(i)(x),\eta_X(i)(x')) = 
 \lGamma(\mu_X)((x,i,1_i),(x',i,1_i))
\]
is defined by the following pullback diagram
\[
 \begin{diagram}
  \node{\lGamma(\mu_X)((x,i,1_i),(x',i,1_i))} \arrow{e} \arrow{s}
  \node{\Gr(X)((x,i),(x',i))} 
  \arrow{e,=} \node{\bigoplus_{u: 
  i\to i} X(i)(X(u)(x),x')\otimes(\{u\}\otimes 1)} \arrow{s,r}{\mu} \\ 
  \node{\Gr(X)((x,i),(x',i))\otimes 1} \arrow{s,=}
  \node{} \node{\bigoplus_{u:i\to i} X(i)(X(u)(x),x')\otimes(\{u\}\otimes
  1)\otimes(\{u\}\otimes 1)} \arrow{s} \\
  \node{\bigoplus_{u: i\to i} X(i)(X(u)(x),x')\otimes(\{u\}\otimes
  1)\otimes 1}
  \arrow[2]{e,t}{1\otimes 1_i} 
  \node{} 
  \node{\Gr(X)((x,i),(x',i))\otimes (I\otimes 1)(i,i)} 
 \end{diagram}
\]
and we have 
\[
 \lGamma(\mu_X)((x,i,1_i),(x',i,1_i)) \cong X(i)(x,x')\times\{1_i\}.
\]
This identification defines a functor
\[
 \eta_X(i) : X(i) \longrightarrow \lGamma(\mu_X)(i)
\]
by the canonical inclusion.

For a morphism $u : i \to j$ and an object $x\in X(i)$, we have
\begin{eqnarray*}
 \lGamma(\mu_X)(u)\circ\eta_X(i)(x) & = & (x,i,u) \\
 \eta_X(j)\circ X(u)(x) & = & (X(u)(x),j,1_j).
\end{eqnarray*}
Define a morphism in $\bm{V}$
\[
 \eta_X(u)(x) : 1 \longrightarrow
 \lGamma(\mu_X)(j)(\lGamma(\mu_X)(u)\circ\eta_X(i)(x),
 \eta_X(j)\circ X(u)(x)) = \lGamma(\mu_X)(j)((x,i,u),(X(u)(x),j,1_j))
\]
by the following diagram
\begin{center}
 \scalebox{.75}{%
 $\begin{diagram}
  \node{1} \arrow[3]{e,t}{1_{X(u)(x)}} \arrow{sse,b}{1_{X(u)(x)}}
   \arrow{se,t,..}{\eta_{X}(u)(x)} \node{} \node{} 
   \node{X(j)(X(u)(x),X(u)(x))\otimes(\{u\}\otimes 1)} 
   \arrow{s} \\ 
  \node{} \node{\lGamma(\mu_X)(j)((x,i,u),(X(u)(x),j,1_j))} \arrow{e}
   \arrow{s} 
  \node{\Gr(X)((x,i),(X(u)(x),j))} \arrow{e,=} \node{\bigoplus_{v : i
  \to j} X(j)(X(u)(x),X(u)(x))\otimes (\{v\}\otimes 1)} \arrow{s,r}{\mu}
   \\  
  \node{} \node{\Gr(X)((x,i),(X(u)(x),j))\otimes 1}
   \arrow[2]{e,t}{1\otimes u}
   \node{} 
   \node{\Gr(X)((x,i),(X(u)(x),j))\otimes (I\otimes 1)(i,j)}
 \end{diagram}$
 }
\end{center}

\begin{lemma}
 The above construction defines a morphism of oplax functors
 \[
  \eta_X : X \longrightarrow \lGamma(\mu_X).
 \]
\end{lemma}

\begin{proof}
 We need to check the following diagrams are commutative.
 \[
 \xymatrix{%
 \lGamma(\mu_X)(1_i)\circ \eta_X(i) \ar@{=>}[r] \ar@{=>}[d] &
 \eta_X(i)\circ X(1_i) 
 \ar@{=>}[d] \\ 
 1_{\lGamma(\mu_X)(i)}\circ \eta_X(i) \ar@{=}[r] & \eta_X(i)\circ 1_{X(i)}
 }
 \]
 \[
 \xymatrix{%
 \lGamma(\mu_X)(v\circ u)\circ \eta_X(k) \ar@{=>}[r] \ar@{=>}[d] &  
 \lGamma(\mu_X)(u)\circ \lGamma(\mu_X)(v)\circ \eta_X(k)  \ar@{=>}[r] &
 \lGamma(\mu_X)(u)\circ \eta_X(j)\circ X(v) \ar@{=>}[d] \\
 \eta_X(k)\circ X(v\circ u) 
 \ar@{=>}[rr] & & 
 \eta_X(k)\circ X(v)\circ X(u).
 }
 \]
 The commutativity of these diagrams follows from the commutativity of
 diagrams of $2$-morphisms in Definition
 \ref{oplax_functor_definition}. The details are omitted. 
\end{proof}

Let us define $\varepsilon_{\mu}$ for an $I$-graded category
$(E,\mu)$. We need to define a $\bm{V}$-functor
\[
 \varepsilon'_{\mu} : \Gr(\lGamma(\mu)) \longrightarrow E
\]
and a $\bm{V}$-natural transformation
\[
\xymatrix{
\Gr(\lGamma(\mu)) \ar[r]^{\mu_{\lGamma(\mu)}}
\ar[d]_{\varepsilon'_{\mu}} & \Gr(\lGamma(\mu))\otimes (I\otimes 1) 
\ar[d]^{\varepsilon'_{\mu}\otimes 1} \\
E \ar[r]^{\mu} \ar@{=>}[ur]^{\varepsilon''_{\mu}} & E\otimes (I\otimes 1).
}
\]
Let us denote the map induced by $\mu$ on objects by
\[
 \mu_0 = 1_{E_0}\times p : E_0 \longrightarrow E_0\times I_0.
\]
$\varepsilon'_{\mu}$ is defines as follows. Objects of
$\Gr(\lGamma(\mu))$ are given by
\[
 \Gr(\lGamma(\mu))_0 = \coprod_{i\in I_0} \lGamma(\mu)(i)_0\times\{i\} =
 \coprod_{i\in I_0} \lset{(e,p(e)\rarrow{u} i,i)}{ e\in E_0, u\in
 I_1}.
\]
For objects $(e,u,i), (e',u',i')$ in $\Gr(\lGamma(\mu))$,
\[
 \Gr(\lGamma(\mu))((e,u,i),(e',u',i')) =
 \bigoplus_{v: i\to i'} \lGamma(\mu)(i')(\lGamma(\mu)(v)(e,u),(e'u'))
\]
and each component $\lGamma(\mu)(i')((e,v\circ u),(e',u'))$ is defined
by the following pullback diagram 
\[
 \begin{diagram}
  \node{\lGamma(\mu)(i')((e,v\circ u),(e',u'))} \arrow{e} \arrow[2]{s} 
  \node{E(e,e')} \arrow{s,r}{\mu} \\
  \node{} \node{E(e,e')\otimes (I\otimes 1)(p(e),p(e'))}
  \arrow{s,r}{u'_*} \\ 
  \node{E(e,e')\otimes 1} \arrow{e,t}{1\otimes (v\circ u)}
  \node{E(e,e')\otimes (I\otimes 1)(p(e),i').} 
 \end{diagram}
\]
$\varepsilon'_{\mu}$ is defined by the top morphism in the above
diagram.

For an object $(x,u,i) \in \Gr(\lGamma(\mu))$, define
\[
 \varepsilon''_{\mu}(x,u,i) : 1 \to (I\otimes 1)(p\circ
 \varepsilon'_{\mu} (x,u,i),p_{\lGamma(\mu)}(x,u,i)) = (I\otimes
 1)(p(e), i)
\]
by 
\[
 \varepsilon''_{\mu}(x,u,i)= u.
\]

It is tedious but elementary to check that $\varepsilon''_{\mu}$ is a
$\bm{V}$-natural transformation and we obtain a morphism
$\varepsilon_{\mu}= (\varepsilon_{\mu}',\varepsilon_{\mu}'')$ in
$\categories{\bm{V}}_I$. 

\begin{lemma}
 The pair
 $\varepsilon_{\mu}= (\varepsilon_{\mu}',\varepsilon_{\mu}'')$ is a
 morphism in $\lcats{\bm{V}}_I$.
\end{lemma}

\begin{proof}[Proof of Theorem \ref{Gr_and_Gamma_are_adjoint}]
 It remains to check the following diagrams are commutative
 \[
  \xymatrix{%
 \Gr \ar@{=>}_{\Gr\circ\eta}[d] \ar@{=}[dr] & &
 \lGamma \ar@{=}[dr] \ar@{=>}^{\eta\circ\lGamma}[r] & \lGamma\circ\Gr\circ
 \lGamma \ar@{=>}^{\lGamma\circ\varepsilon}[d] \\ 
  \Gr\circ\lGamma\circ\Gr \ar@{=>}_{\varepsilon\circ\Gr}[r] & \Gr & &
 \lGamma 
 }
 \]

 For an oplax functor $X$, consider the composition
 \[
  \Gr(X) \rarrow{\Gr(\eta_X)} \Gr(\lGamma(\mu_X))
 \rarrow{\varepsilon_{\Gr(X)}} \Gr(X). 
 \]
 For objects, we have
 \[
  (x,i) \longmapsto (x,1_i,i) \longmapsto (x,i).
 \]
 For objects $(x,i), (x',i')$ in $\Gr(X)$, we have
 \begin{eqnarray*}
  \Gr(\lGamma(\mu_X))(\eta_X(x,i),\eta_X(x',i')) & = &
   \bigoplus_{u:i\to i'}
   \lGamma(\mu_X)(i')(\lGamma(\mu_X)(u)(x,1_i),(x',1_{i'})) \\
  & = & \bigoplus_{u:i\to i'} \lGamma(\mu_X)(i')((x,u),(x',1_{i'})) 
 \end{eqnarray*}
 and each component $\lGamma(\mu_X)(i')((x,u),(x',1_{i'}))$ can be
 identified with $X(i')(X(u)(x),x')$. It follows that the composition
 \[
  \Gr(X)((x,i),(x',i')) \longrightarrow
 \Gr(\lGamma(\mu_X))(\eta_X(x,i),\eta_X(x',i')) \longrightarrow
 \Gr(X)((x,i),(x,i')) 
 \]
 is the identity.

 For an $I$-graded category $\mu : E \to E\otimes (I\otimes 1)$,
 consider the composition
 \[
  \lGamma(\mu) \rarrow{\eta_{\lGamma(\mu)}} \lGamma(\Gr({\lGamma(\mu)}))
 \rarrow{\lGamma(\varepsilon_{\mu})} \lGamma(\mu)
 \]
 of morphisms of oplax functors. Note that $\lGamma$ takes values in
 strict functors and $\eta_{\lGamma(\mu)}$ and $\lGamma(\varepsilon_{\mu})$
 are ordinary natural transformations. Thus
 it suffices to consider the composition 
 \[
  \lGamma(\mu)(i) \rarrow{\eta_{\lGamma(\mu)}} \lGamma(\Gr(\lGamma(\mu)))(i)
 \rarrow{\lGamma(\varepsilon_{\mu})} \lGamma(\mu)(i)
 \]
 of $\bm{V}$-functors for each object $i\in I_0$.

 Objects of $\lGamma(\Gr(\lGamma(\mu)))(i)$ are
 \[
  \lGamma(\Gr(\lGamma(\mu)))(i)_0 = \coprod_{(e,p(e)\rarrow{u}i,i)\in
 \Gr(\lGamma(\mu))_0} \{(e,p(e)\rarrow{u}i,i)\}\times\Mor_{I}(i,i)
 \]
 and, for $(e,u)\in \lGamma(\mu)(i)_0$,
 \begin{eqnarray*}
  (\lGamma(\varepsilon_{\mu})\circ\eta_{\lGamma(\mu)})(e,u) & = &
   \lGamma(\varepsilon_{\mu})((e,u,i),1_i) \\
  & = & (\varepsilon_{\mu}'(e,u,i),1_i\circ\varepsilon_{\mu}''(e,u,i)) \\
  & = & (e,u).
 \end{eqnarray*}
 The composition
 \[
  \lGamma(\mu)(i)((e,u),(e',u')) \rarrow{\eta_{\lGamma(\mu)}}
 \lGamma(\Gr(\lGamma(\mu)))(i)((e,u,i), (e',u',i))
 \rarrow{\lGamma(\varepsilon_{\mu})} 
 \lGamma(\mu)(i)((e,u),(e',u')) 
 \]
 is easily seen to be the identity, since $\eta$ is given by the
 canonical inclusion and $\varepsilon'$ is given by the projection. 
\end{proof}
\subsection{The Smash Product Construction for Precofibered and
  Prefibered Categories} 
\label{fibered}

In this section, we study the relations between $\lGamma$ and
$\Gamma_{\cof}$ and $\rGamma$ and $\Gamma_{\fib}$ for precofibered and
prefibered categories, respectively. We concentrate on the
case of precofibered categories. The case of prefibered categories is
analogous and is left to the reader.

\begin{lemma}
 Let $\mu : E \to E\otimes(I\otimes 1)$ be a precofibered $I$-graded
 category. The family of functors
 \[
  i_i : \Gamma_{\cof}(\mu)(i) \longrightarrow \lGamma(\mu)(i)
 \]
 indexed by $i\in I_0$ defines a left transformation of oplax functors 
 \[
  i : \Gamma_{\cof}(\mu) \longrightarrow \lGamma(\mu).
 \]
\end{lemma}

\begin{proof}
 For $u : i \to i'$ in $I$, we have the following diagram
 \[
  \begin{diagram}
   \node{\Gamma_{\cof}(\mu)(i)} \arrow[3]{s,l}{\Gamma_{\cof}(\mu)(u)}
   \arrow[2]{e,t}{i_{i}} \arrow{se,t}{i_i} \node{} 
   \node{\lGamma(\mu)(i)} 
   \arrow[3]{s,r}{\lGamma(\mu)(u)} \arrow{sw,=} \\
   \node{} \node{\lGamma(\mu)(i)} \arrow{s,r}{\lGamma(\mu)(u)} \node{} \\
   \node{} \node{\lGamma(\mu)(i')} \arrow{sw,t}{s_{i}} \arrow{se,=} \node{} \\
   \node{\Gamma_{\cof}(\mu)(i')} \arrow[2]{e,t}{i_{i'}} \node{}
   \node{\lGamma(\mu)(i')} 
  \end{diagram}
 \]
 Since $s_{i'}$ is left adjoint to $i_{i'}$, we have a $\bm{V}$-natural
 transformation
 \[
  1_{\lGamma(\mu)(i')} \Longrightarrow i_{i'}\circ s_{i'},
 \]
 which induces a left transformation.
\end{proof}

We only need the fact that $s_{i'}$ is left adjoint to $i_{i'}$ in the
above proof and thus
a similar argument implies that we obtain a right transformation from
$\{s_i\}$. 

\begin{lemma}
 The functors
 \[
  s_i : \lGamma(\mu)(i) \longrightarrow \Gamma_{\cof}(\mu)(i)
 \]
 defines a right transformation of oplax functors from $\lGamma(\mu)$ to
 $\Gamma_{\cof}(\mu)$. 
\end{lemma}

When the natural transformations
$\varepsilon_i : s_i\circ i_i \Rightarrow 1_{\Gamma_{\cof}(\mu)(i)}$ are
natural isomorphisms, we obtain a left transformation.

\begin{corollary}
 When $\mu$ is precofibered and $\{\varepsilon_i\}$ are natural
 isomorphisms, we obtain a left transformation 
 \[
  s : \lGamma(\mu) \longrightarrow \Gamma_{\cof}(\mu).
 \]
\end{corollary}

\begin{corollary}
 Let $\mu$ be a precofibered $I$-graded $\bm{V}$-category. If the
 natural transformations
 \begin{eqnarray*}
  \varepsilon_i & : & s_i\circ i_i \Longrightarrow
   1_{\Gamma_{\cof}(\mu)(i)} \\
  \eta_i & : & 1_{\lGamma(\mu)(i)} \Longrightarrow i_i\circ s_i
 \end{eqnarray*}
 are natural isomorphisms, for any oplax
 functor 
 $X : I \to \categories{\bm{V}}$, the above left transformations induce
 an equivalence of categories 
 \[
  i : \lOplax(I,\categories{\bm{V}})(X,\lGamma(\mu)) \longleftrightarrow
 \lOplax(I,\categories{\bm{V}})(X,\Gamma_{\cof}(\mu)) : s. 
 \]
\end{corollary}

\begin{corollary}
 Under the assumptions as above, we have an equivalence of categories
 \[
  \lcats{\bm{V}}_I(\Gr(X),\mu) \simeq
 \lOplax(I,\categories{\bm{V}})(X,\Gamma_{\cof}(\mu)). 
 \]
\end{corollary}

In other words, under these assumptions, $\Gamma_{\cof}$ can be regarded
as a right adjoint functor to $\Gr$ in a weak sense.

\appendix

\section{Appendices}

\subsection{Enriched Categories by Comodules}
\label{enrichment_by_comodule}

Given a set $S$, the category of quivers with the set of vertices is the
comma category
\[
 \category{Quivers}(S) = \Sets\downarrow (S\times S).
\]
This category can be made into a monoidal category under the pullback
\[
 \begin{diagram}
  \node{Y\times_S X} \arrow[2]{e} \arrow[2]{s} \node{} \node{X}
  \arrow{s} \\ 
  \node{} \node{} \node{S\times S} \arrow{s,r}{\pr_1} \\
  \node{Y} \arrow{e} \node{S\times S} \arrow{e,t}{\pr_2}  \node{S.}
 \end{diagram}
\]
It is a well-known fact that a category $X$ with the set of objects $S$ is
a monoid object in this monoidal category. See \cite{math/0301209}, for
example. This characterization can be generalized to a definition of
small enriched categories
by using the notion of bicomodules, which is the subject of this
appendix. For the standard definition of 
enriched categories, see \S\ref{enriched_category}.

Our starting point is to regard $C$-$C$-bimodules as ``quivers enriched
over $\bm{V}$ with vertices given by $C$''.

\begin{definition}
 When $C=S\otimes 1$ for a set $S$, $C$-$C$-bicomodules are called
 $\bm{V}$-quivers with the set of vertices $S$. The category of
 $\bm{V}$-quivers with the set of vertices $S$ is denoted by
 $\bm{V}\text{-}\category{Quivers}(S)$. 
\end{definition}

Lemma \ref{monoidal_category_under_cotensor} allows us to define
enriched categories without referring to 
each object.

\begin{definition}
 \label{enriched_category_definition2}
 Let $C$ be a flat comonoid object in $\bm{V}$.
 Then a monoid object in the monoidal category 
 $(\biComod{C}{C}, \Box_C, C)$ is called a category object in
 $\bm{V}$ with objects $C$.
 
 For a category object $M$ in $\bm{V}$ with objects $C$, the right and
 left coactions are denoted by
 \begin{eqnarray*}
  s & : & M \longrightarrow M\otimes C \\
  t & : & M \longrightarrow C\otimes M
 \end{eqnarray*}
 and called the source and the target.

 When $C=S\otimes 1$ for a set $S$, a category object in
 $\bm{V}$ with objects $S\otimes 1$ is called a category enriched over
 $\bm{V}$, or simply $\bm{V}$-category, with the set of objects 
 $S$.
\end{definition}

We use the following convention for simplicity.

\begin{convention}
 For a $\bm{V}$-category $X$, the set of objects is denoted by $X_0$. As
 an object of $\bm{V}$, $X$ is denoted by $X_1$.

 We also use traditional notations described in
 \S\ref{enriched_category}, when we have a coproduct decomposition 
 \[
  X_1 = \bigoplus_{x,y\in X_0} X(x,y).
 \]
\end{convention}

\begin{example}
 Let $\bm{V}=\lMod{k}$. The canonical natural transformation
 \[
  \theta_{S,T} : (S\times T)\otimes 1 \longrightarrow (S\otimes 1)\otimes
 (T\otimes 1)
 \]
 is an isomorphism for any $S$ and $T$.

 A category $A$ enriched over $\lMod{k}$ with the set of objects $A_0$
 is a $k$-module $A_1$ equipped with a bicomodule structure 
 \begin{eqnarray*}
  s & : & A_1 \longrightarrow A_1\otimes (A_0\otimes 1), \\
  t & : & A_1 \longrightarrow (A_0\otimes 1)\otimes A_1
 \end{eqnarray*}
 and a monoid structure
 \[
  \circ : A_1\Box_{A_0\otimes 1} A_1 \longrightarrow A_1.
 \]

 Lemma \ref{decomposition_and_comodule} implies that we have a coproduct
 decomposition
 \[
  A_1 \cong \bigoplus_{(a,b)\in A_0\times A_0} A(a,b)
 \]
 and
\[
 A_1\Box_{A_0\otimes 1}A_1 = \bigoplus_{a,b,c\in A_0} A(b,c)\otimes
 A(a,b). 
\]
 The monoid structure $\circ$ induces
 \[
  \circ : A(b,c)\otimes A(a,b) \longrightarrow A(a,c).
 \]
 The unit
 \[
  \iota : A_0\otimes 1 \longrightarrow A_1
 \]
 induces a morphism
 \[
  1 \rarrow{s} \{s\}\otimes 1 \rarrow{\iota} A(a,a)
 \]
 serving as identity morphisms.
 And we obtain the standard definition of $k$-linear category.

 Note that we the inclusion
 \[
  A_1\Box_{A_0\otimes 1} A_1 \hookrightarrow A_1\otimes A_1
 \]
 has a canonical retraction
 \[
  r : A_1 \otimes A_1 \longrightarrow A_1\Box_{A_0\otimes 1} A_1
 \]
 and the composition
 \[
  A_1\otimes A_1 \rarrow{r} A_1\Box_{A_0\otimes 1} A_1 \rarrow{\circ} A_1  
 \]
 makes $A_1$ into an algebra (possibly without a unit). This is the
 algebra associated with a $k$-linear category $A$, which is often
 denoted by $\Lambda(A)$. See
 \cite{Gerstenhaber-Schack83,math.RA/0305218}, for 
 example. 
\end{example}

\begin{example}
 \label{free_quiver}
 Let $Q$ be a quiver with the set of vertices $Q_0$, i.e.\ a diagram
 \[
  \xymatrix{Q_1 \ar@<1ex>[r]^{s} \ar@<-1ex>[r]_{t} & Q_0.}
 \]
 We obtain a $(Q_0\otimes 1)$-$(Q_0\otimes 1)$-bimodule structure on
 $Q_1\otimes 1$ by
 \[
  \begin{diagram}
   \node{} \node{Q_1\otimes 1} \arrow{s,r}{\Delta\otimes 1}
   \arrow{sse,l}{t} 
   \arrow{ssw,l}{s} 
   \node{} \\
   \node{} \node{(Q_1\times Q_1)\times 1} \arrow{s,r}{\theta} \node{} \\
   \node{(Q_1\otimes 1) \otimes (Q_0\otimes 1)} \node{(Q_1\otimes
   1)\otimes (Q_1\otimes 1)} \arrow{e,t}{t\otimes 1}
   \arrow{w,t}{s\otimes 1} \node{(Q_0\otimes
   1)\otimes (Q_1\otimes 1).}
  \end{diagram}
 \]
 We obtain a functor
 \[
  (-)\otimes 1 : \category{Quivers}(Q_0) \longrightarrow
 \biComod{(Q_0\otimes 1)}{(Q_0\otimes 1)} =
 {\bm{V}}\text{-}\category{Quivers}(Q_0).  
 \]
\end{example}

\begin{example}
 Let $\bm{V}$ be the category $\Spaces$ of topological spaces.
 As we have seen in Example \ref{decomposition_in_product_type}, any
 object $C$ is a comonoid and a coaction of $C$ on another object $M$ is
 determined by a morphism.
 \[
  \pi : M \longrightarrow C.
 \]
 Thus the category $\biComod{C}{C}$ of bicomodules over $C$
 can be identified with the comma category
 $\Spaces\downarrow C\otimes C$, since the monoidal structure is given
 by the product. The corresponding 
 monoidal structure on $\Spaces\downarrow C\otimes C$ is the monoidal
 structure described at the beginning of this section. A monoid object
 $M$ in this monoidal category is, therefore, a topological category with
 the space of objects $C$.

 When $C$ has a discrete topology, i.e.\ $C=S\otimes 1$ for a set $S$,
 we have a coproduct decomposition 
 \[
  M \cong \coprod_{(s,t)\in S} M(s,t)
 \]
 as we have seen in Example \ref{decomposition_in_product_type} and we
 obtain the standard definition of a category enriched over $\Spaces$
 with the set of objects $S$.
\end{example}

\begin{definition}
 Let $A$ and $B$ be $\bm{V}$-categories. A $\bm{V}$-functor from $A$ to
 $B$ is a morphism of bicomodules
 \[
  f = (f_0,f_1) : A \longrightarrow B
 \]
 making the following diagrams commutative
 \[
  \begin{diagram}
   \node{A_0\otimes 1} \arrow{s,l}{f_0\otimes 1} \arrow{e,t}{\iota}
   \node{A_1} \arrow{s,r}{f_1} \node{A_1\Box_{A_0} A_1}
   \arrow{s,l}{f_1\Box f_1} \arrow{e,t}{\circ} \node{A_1}
   \arrow{s,r}{f_1} \\ 
   \node{B_0\otimes 1} \arrow{e,t}{\iota} \node{B_1}
   \node{B_1\Box_{B_0} B_1}
   \arrow{e,t}{\circ} \node{B_1.} 
  \end{diagram}
 \]

 The category of $\bm{V}$-categories and $\bm{V}$-functors is denoted by
 $\categories{\bm{V}}$. 
\end{definition}

A monoidal structure on $\categories{\bm{V}}$ is induced from the
following monoidal structure on the category of bicomodules.

\begin{definition}
 For bicomodules $M$ and $N$ over $C$ and $D$, respectively,
 define left and right coactions of $C\otimes D$ on
 $M\otimes N$ by
 \begin{eqnarray*}
  M\otimes N & \rarrow{\mu^R_M\otimes \mu^R_N} & (M\otimes C)\otimes
   (N\otimes D) 
 \cong (M\otimes N)\otimes (C\otimes D), \\
  M\otimes N & \rarrow{\mu^L_M\otimes \mu^L_N} & (C\otimes M)\otimes
   (D\otimes N) \cong (C\otimes D)\otimes (M\otimes N).
 \end{eqnarray*}
\end{definition}

\begin{lemma}
 \label{monoidal_category_of_comodules}
 The above operation makes the category $\category{Bicomodules}(\bm{V})$
 of bicomodules in $\bm{V}$ into a
 symmetric monoidal category. The unit is given by 
 \[
  1 \longrightarrow 1\otimes 1.
 \]
\end{lemma}

For $\bm{V}$-categories $A$ and $B$, we would like to define a morphism
\[
 (A\otimes A)\Box_{(A_0\otimes 1)\otimes(B_0\otimes 1)} (B\otimes B)
 \longrightarrow A\otimes B.
\]
We have the following diagram
\begin{equation}
 \xymatrix{
 (A\otimes B)\Box_{(A_0\otimes 1)\otimes(B_0\otimes 1)} (A\otimes B)
 \ar[r] & (A\otimes B)\otimes (A\otimes B) \ar[d] \ar@<1ex>[r]^(.35){s}
 \ar@<-1ex>[r]_(.35){t} & (A\otimes B)\otimes (A_0\otimes 1)\otimes (B_0\otimes
 1)\otimes (A\otimes B) \ar[d] \\ 
 (A\Box_{A_0\otimes 1}A)\otimes (B\Box_{B_0\otimes 1}B) \ar[r] &
 (A\otimes A)\otimes (B\otimes B) \ar@<1ex>[r]^(.35){s}
 \ar@<-1ex>[r]_(.35){t} & (A\otimes 
 (A_0\otimes 1)\otimes A) \otimes
 (B\otimes (B_0\otimes 1)\otimes B).
 }
 \label{product_of_V-categories} 
\end{equation}

\begin{lemma}
 \label{monoidal_category_of_V-categories}
 Suppose the bottom row in the diagram (\ref{product_of_V-categories})
 is an equalizer. Then the resulting morphism
 \[
  (A\otimes B)\Box_{(A_0\otimes 1)\times (B_0\otimes 1)} (A\otimes B)
 \rarrow{} (A\Box_{A_0\otimes 1}
 A)\otimes (B\Box_{B_0\otimes 1} B) \rarrow{\circ \otimes \circ} A\otimes B.
 \]
 defines a structure of monoid on $A\otimes B$.

 If the bottom row in (\ref{product_of_V-categories}) is an equalizer
 for any $\bm{V}$-categories $A$ and $B$, then we obtain a symmetric
 monoidal structure on $\categories{\bm{V}}$. 
\end{lemma}

\begin{remark}
 The condition for the above Lemma is satisfied when $\bm{V}$ is the
 category of $k$-modules, chain complexes, topological spaces,
 simplicial sets, and small categories.
\end{remark}

We also have $2$-morphisms in $\categories{\bm{V}}$.

\begin{definition}
 Let $A$ and $B$ be $\bm{V}$-categories and 
 \[
  f,g : A \longrightarrow B
 \]
 be $\bm{V}$-functors. A $\bm{V}$-natural transformation from $f$ to $g$
 \[
  \varphi : f \Longrightarrow g
 \]
 is a morphism
 \[
  \varphi : A_0\otimes 1 \longrightarrow B_1
 \]
 satisfying the following conditions:
 \begin{enumerate}
  \item The following diagram is commutative
	\[
	 \begin{diagram}
	  \node{(A_0\otimes 1)\otimes (A_0\otimes 1)}
	  \arrow{s,l}{\varphi\otimes f_0\otimes 1}
	  \node{A_0\otimes 1} \arrow{e,t}{\Delta} \arrow{s,r}{\varphi}
	  \arrow{w,t}{\Delta} 
	  \node{(A_0\otimes 1)\otimes  
	  (A_0\otimes 1)} \arrow{s,r}{\varphi\otimes g_0\otimes 1} \\
	  \node{B_1\otimes (B_0\otimes 1)} \node{B_1} \arrow{e,t}{t}
	  \arrow{w,t}{s} 
	  \node{B_1\otimes (B_0\otimes 1).}
	 \end{diagram}
	\]
  \item The following diagram is commutative
	\[
	 \begin{diagram}
	  \node{} \node{A_1} \arrow{se,t}{T\circ t} \arrow{sw,t}{s} \node{} \\
	  \node{A_1\Box_{A_0\otimes 1} (A_0\otimes 1)}
	  \arrow{s,l}{g\otimes\varphi} \node{} 
	  \node{(A_0\otimes 1)\Box_{A_0\otimes 1} A_1}
	  \arrow{s,r}{\varphi\otimes f} \\ 
	  \node{B_1\Box_{B_0\otimes 1} B_1} \arrow{se,b}{\circ} \node{}
	  \node{B_1\Box_{B_0\otimes 1} B_1} 
	  \arrow{sw,b}{\circ} \\ 
	  \node{} \node{B_1}
	 \end{diagram}
	\]
 \end{enumerate}
\end{definition}

We have seen in Lemma \ref{free_quiver} that the ``free object
functor''
\[
 (-)\otimes 1 : \Sets \longrightarrow \bm{V}
\]
can be extended to
\[
 (-)\otimes 1 : \category{Quivers} \longrightarrow
 \bm{V}\text{-}\category{Quivers}. 
\]
It induces a functor
\[
 (-)\otimes 1 : \Categories \longrightarrow \categories{\bm{V}}.
\]

\begin{lemma}
 \label{free_category}
 For a small category $I$  with the set of objects $I_0$ and the set
 of morphisms $I_1$, define a $\bm{V}$-category $I\otimes 1$ by
 $(I\otimes 1)_0=I_0$ and $(I\otimes 1)_1=I_1\otimes 1$. The structure
 morphisms are induced by those of $I$.

 Then we obtain a $2$-functor
 \[
 (-)\otimes 1 : \Categories \longrightarrow \categories{\bm{V}}.
 \]
\end{lemma}
\subsection{The Grothendieck Construction as a Left Adjoint to the
  Diagonal Functor}
\label{left_adjoint_to_diagonal}

Recall from Definition \ref{2-diagonal} that we have the
diagonal functor
\[
  \Delta : \bm{C} \longrightarrow \lOplax(I,\bm{C}). 
\]

The following adjunction is well-known for non-enriched categories. For
example, it can be found in Thomason's \cite{Thomason79-1}. According to
Thomason, it is originally due to J.~Gray \cite{J.Gray69}. For the sake
of completeness, we give a proof of an enriched version of this
adjunction. 

\begin{theorem}
 \label{first_adjunction}
 For any oplax functor $X : I \to \categories{\bm{V}}$ and a
 $\bm{V}$-category $A$, we have the following natural isomorphism of
 categories
 \[
  \categories{\bm{V}}(\Gr(X),A) \cong
 \lOplax(I,\categories{\bm{V}})(X,\Delta(A)). 
 \]
\end{theorem}

\begin{proof}
 We need to define morphisms
 \begin{eqnarray*}
  \eta_X & : & X \longrightarrow \Delta(\Gr(X)) \\
  \varepsilon_A & : & \Gr(\Delta(A)) \longrightarrow A
 \end{eqnarray*}
 in $\lOplax(I,\categories{\bm{V}})$ and $\categories{\bm{V}}$,
 respectively. 

 For each $i\in I_0$, 
 \[
  \eta_X(i) : X(i) \hookrightarrow \Delta(\Gr(X))(i) = \Gr(X).
 \]
 is given by the canonical inclusion. For a morphism $u : i\to j$ in
 $I$, we need to define a $\bm{V}$-natural transformation
 \[
  \Delta(\Gr(X))(u)\circ\eta_X(i) \Longrightarrow \eta_X(j)\circ X(u). 
 \]
 It is given, for $x\in X(i)_0$, by the composition
 \[
  1 \rarrow{1_{X(u)(x)}} X(j)(X(u)(x),X(u)(x)) \hookrightarrow
 \Gr(X)((x,i),(X(u)(x),j)) = \Gr(X)(\eta_X(i)(x), \eta_X(j)(X(u)(x))). 
 \]
 It is left to the reader to check that these family of morphisms form
 a $\bm{V}$-natural transformation.
 
 The category $\Gr(\Delta(A))$ has objects
 \[
  \Gr(\Delta(A))_0 = \coprod_{i\in I_0} \Delta(A)(i)_0\times\{i\} =
 A_0\times I_0.
 \]
 For $(a,i), (a',i') \in \Gr(\Delta(A))_0$, we obtain a morphism
 \[
  \Gr(\Delta(A))((a,i),(a',i')) = \bigoplus_{u : i\to i'} A(a,a')
 \longrightarrow 
 A(a,a')\otimes (I\otimes 1)(i,i').
 \]
 or
 \[
  \Gr(\Delta(A)) \longrightarrow A\otimes (I\otimes 1).
 \]
 The counit $\varepsilon$ of $I\otimes 1$ induces
 \[
  \varepsilon_A : \Gr(\Delta(A)) \longrightarrow A\otimes (I\otimes 1)
 \rarrow{1\otimes\varepsilon} A\otimes 1 \cong A.
 \]

 It is elementary to check that $\eta_X$ and $\varepsilon_A$ give us
 adjunctions we want. The proof is omitted.
\end{proof}

\begin{example}
 Let $G$ be a group and consider the case $A=\lMod{k}$. $\Delta(A)$ 
 is $\lMod{k}$ equipped with the trivial $G$-action. Thus
 $\lOplax(G,\categories{k})(X,\Delta(A))$ is the
 category of right $G$-invariant functors from $X$ to
 $\lMod{k}$. (See Example \ref{Gamma-invariant_functor} for right
 $G$-invariant functors).

 On the other hand, the category $\categories{k}(\Gr(X),\lMod{k})$
 is the category of representations of the Grothendieck construction of
 $X$. Cibils and Marcos \cite{math/0312214} regard $\Gr(X)$ as a version
 of orbit category. Thus we can identify the category of right
 $G$-invariant functors with the category of representations of the
 Cibils-Marcos orbit category. This is observed by Asashiba and stated
 as Theorem 6.2 in \cite{0807.4706}.
\end{example}

\subsection{The Grothendieck Construction over Product Type Monoidal Categories}
\label{product_type}

In this and next sections, we specialize the constructions in this paper
to the case $\bm{V}$ is a product type symmetric monoidal category. In
this case, we do not need to use comodules.

Throughout this section, $\bm{V}$ is a product type symmetric monoidal
category. 
Let us modify the Grothendieck construction as a $2$-functor
\[
 \Gr : \lOplax(I,\categories{\bm{V}}) \longrightarrow
 \lcats{\bm{V}}\downarrow I\otimes 1.
\]

\begin{definition}
 For an oplax functor
 \[
  X : I \longrightarrow \categories{\bm{V}},
 \]
 define
 \[
  p_X : \Gr(X) \longrightarrow I\otimes 1
 \]
 by
 \[
  p_X(x,i) = i
 \]
 for objects and
 \[
  p_X : \Gr(X)((x,i),(y,j)) = \bigoplus_{u: i\to j} X(j)(X(u)(x),y)
 \longrightarrow (I\otimes 1)(i,j)
 \]
 is defined by
 \[
  X(j)(X(u)(x),y) \longrightarrow 1 \rarrow{\eta} \{u\}\otimes 1
 \hookrightarrow \Mor_{I}(i,j)\otimes 1 = (I\otimes 1)(i,j)
 \]
 on each component. Recall that we assume that $1$ is a terminal
 object in $\bm{V}$. The morphism $\eta$ is one of structure morphisms
 of the lax monoidal functor $(-)\otimes 1$.
\end{definition}

\begin{definition}
 \label{Gr(F)_for_type_T}
 For a morphism of oplax functors
 \[
  (F,\varphi) : X \longrightarrow Y,
 \]
 define
 \[
  \Gr(F,\varphi) : \Gr(X) \longrightarrow \Gr(Y)
 \]
 by Definition \ref{Gr(F)_definition}.
 
 Define a natural transformation
 \[
  \Gr(\varphi) : p_Y\circ\Gr(F,\varphi) \Longrightarrow p_X
 \]
 by the identity
 \[
 \Gr(\varphi)(i) : p_Y\circ\Gr(F,\varphi)(x,i)=i \rarrow{1_i}
 i=p_X(x,i). 
 \]
\end{definition}

\begin{lemma}
 The pair $(\Gr(F,\varphi),\Gr(\varphi))$ is a $1$-morphism in
 $\lcats{\bm{V}}\downarrow I\otimes 1$.
\end{lemma}

\begin{proof}
 Obvious from the definition.
\end{proof}

\begin{definition}
 For a $2$-morphism
 \[
  \theta : (F,\varphi) \Longrightarrow (G,\psi)
 \]
 in $\lOplax(I,\categories{\bm{V}})$, define
\[
 \Gr(\theta) : (\Gr(F,\varphi),\Gr(\varphi)) \Longrightarrow
 (\Gr(G,\psi),\Gr(\psi)) 
\]
in $\lcats{\bm{V}}\downarrow I$ by 
 \[
  \Gr(\theta)(x,i) = (\theta(i)(x), 1_i) : \Gr(F,\varphi)(x,i) =
 (F(i)(x),i) \longrightarrow (G(i)(x),i) = \Gr(G,\psi)(x,i).
 \]
\end{definition}

 This makes the following diagram of natural transformations commutative
 \[
 \xymatrix{%
 p_Y\circ \Gr(F,\varphi) \ar@{=>}_{\varphi}[dr] \ar@{=>}^{p_Y\circ
 \Gr(\theta)}[rr] & & 
 p_Y\circ \Gr(G,\psi) \ar@{=>}^{\psi}[dl] \\ 
  & p_X
  }
 \]
 since all morphisms in the diagram are identity. And we obtain a
 $2$-morphism
 \[
  \Gr(\theta) : (\Gr(F,\varphi),\Gr(\varphi)) \Longrightarrow
 (\Gr(G,\psi),\Gr(\psi)). 
 \]

Lemma \ref{Gr_as_functor_to_comma_categories} can be translated as
follows. 

\begin{lemma}
 The above constructions define a functor
 \[
 \Gr : \lOplax(I,\categories{\bm{V}})(X,X')
 \longrightarrow (\lcats{\bm{V}}\downarrow I\otimes
 1)(\Gr(X),\Gr(X')).   
 \]
\end{lemma}

And Proposition \ref{Gr_as_graded_category} becomes the following form.

\begin{proposition}
 The Grothendieck construction defines a functor of
 $\category{Categories}$-enriched categories, i.e.\ a $2$-functor
 \[
  \Gr : \lOplax(I,\categories{\bm{V}}) \longrightarrow
 \lcats{\bm{V}}\downarrow I\otimes 1.
 \]
\end{proposition}

\subsection{Comma Categories for Enriched Categories}
\label{comma_for_enriched}

When $\bm{V}$ is of product type, we can regard the Grothendieck
construction as
\[
 \Gr : \lOplax(I,\categories{\bm{V}}) \longrightarrow
 \lcats{\bm{V}}\downarrow (I\otimes 1).
\]
In order to define a right adjoint to the Grothendieck construction of
this form, we need comma categories for enriched categories. There seems
to be a general theory of comma objects in enriched categories. (See,
for example, \cite{R.J.Wood78}.) We choose, however, a more
down-to-earth approach, which is enough for our purposes.

\begin{definition}
 \label{comma_V-category}
 Let $\bm{V}$ be a monoidal category closed under pullbacks and 
 \[
 p : E \longrightarrow B
 \]
 be a $\bm{V}$-functor. For an object $x\in B_0$, define a
 $\bm{V}$-category $p\downarrow x$ as follows. Objects are given by
 \[
  (p\downarrow x)_0 = \coprod_{e\in E_0} \{e\}\times
 \Mor_{\underline{B}}(p(e),x). 
 \]
 For a pair $(e,f), (e',f') \in (p\downarrow x)_0$, define an object
 $(p\downarrow x)((e,f),(e',f'))$ in $\bm{V}$ by the following
 pullback diagram in $\bm{V}$
 \[
 \begin{diagram}
 \node{} \node{(p\downarrow x)((e,f),(e',f'))} \arrow[2]{s} \arrow{e}
 \node{E(e,e')} \arrow{s,r}{p} \\
 \node{} \node{} \node{B(p(e),p(e'))} \arrow{s,r}{f'_*} \\
 \node{} \node{1} \arrow{e,t}{f} \node{B(p(e),x).}
 \end{diagram}
 \]

 The composition
 \[
  (p\downarrow x)((e',f'),(e'',f'')) \otimes (p\downarrow
 x)((e,f),(e',f')) \longrightarrow (p\downarrow x)((e,f),(e'',f''))
 \]
 is defined by the commutativity of the following diagram
 \begin{center}
  \scalebox{.6}{%
  $\begin{diagram}
   \node{(p\downarrow x)((e',f'),(e'',f''))\otimes (p\downarrow
   x)((e,f),(e',f'))} \arrow[2]{s} \arrow{se} \arrow[3]{e} \node{}
    \node{} 
   \node{E(e',e'')\otimes E(e,e')} \arrow{sw,t}{p\otimes p}
    \arrow{s,r}{\circ} \\  
   \node{} \node{(p\downarrow x)((e',f'),(e'',f''))\otimes
    B(p(e),p(e'))} \arrow{e} \arrow{s}
    \node{B(p(e'),p(e''))\otimes B(p(e),p(e'))}
    \arrow{se,t}{\circ} \arrow{s,r}{f''_*\otimes 1}
    \node{E(e,e'')} \arrow{s,r}{p} \\ 
    \node{1\otimes (p\downarrow x)((e,f),(e',f'))} \arrow{s}
    \arrow{se} \arrow{e}
    \node{1\otimes B(p(e),p(e'))} \arrow{se}
    \arrow{e,t}{f'\otimes 1}
    \node{B(p(e'),x)\otimes B(p(e),p(e'))} \arrow{sse,t}{\circ}
    \node{B(p(e),p(e''))} \arrow[2]{s,r}{f''_*} \\
    \node{1\otimes 1} \arrow{s} \node{(p\downarrow x)((e,f),(e',f'))}
    \arrow{sw} \arrow{e} 
    \node{B(p(e),p(e'))} \arrow{se,t}{f'_*} \node{} \\
    \node{1} \arrow[3]{e,t}{f} \node{} \node{} \node{B(p(e),x).}
  \end{diagram}$}
 \end{center}

 The identity morphisms are defined by
 \[
  \begin{diagram}
   \node{1} \arrow{se,..} \arrow{ese,t}{1} \arrow{sse,=} \node{} \node{}
   \\ 
   \node{} \node{(p\downarrow x)((e,f),(e,f))} \arrow{e} \arrow{s}
   \node{B(p(e),p(e))} \arrow{s,r}{f_*} \\ 
   \node{} \node{1} \arrow{e,t}{f} \node{B(p(e),x).}
  \end{diagram}
 \]
\end{definition}

If we restrict our attention to $1$-morphisms in
$\lcats{\bm{V}}_I$, i.e.\ degree preserving morphisms, we have an
alternative construction for the smash product.

\begin{definition}
 Let $I$ be a small category and
 \[
 p : X \longrightarrow I\otimes 1
 \]
 a $\bm{V}$-functor. Define a functor
 \[
  \lGamma(p) : I \longrightarrow \categories{\bm{V}}
 \]
 as follows. For an object $i\in I_0$, 
 \[
  \lGamma(p)(i) = p\downarrow i.
 \]

 For a morphism $u : i \to j$ in $I$, define a $\bm{V}$-functor
 \[
  \lGamma(u)=p\downarrow u : p\downarrow i \longrightarrow p\downarrow j
 \]
 as follows. For an object $(x,v)$ in $p\downarrow i$, define
 \[
  (p\downarrow u)(x,f) = (x,u\circ f).
 \]
 For morphisms, 
 \[
  p\downarrow u : (p\downarrow i)((x,f),(x',f')) \longrightarrow
 (p\downarrow j)((x,u\circ f),(x,u\circ f'))
 \]
 is defined by the following commutative diagram
 \[
 \begin{diagram}
  \node{(p\downarrow i)((x,f),(x',f'))} \arrow{e} \arrow[2]{s}
  \node{X(x,x')} \arrow{s,r}{p} \\
  \node{} \node{(I\otimes 1)(p(x),p(x'))} \arrow{s,r}{f'_*}
  \arrow{sse,t}{(u\circ f')_*} \\
  \node{1} \arrow{ese,b}{u\circ f} \arrow{e,t}{f} \node{(I\otimes
  1)(p(x),i)} 
  \arrow{se,t}{u_*} \\
  \node{} \node{} \node{(I\otimes 1)(p(x),j).}
 \end{diagram}
 \]
\end{definition}

\begin{remark}
 When $p$ is a graded category defined by a coproduct decompositions, we
 can define $\lGamma$ without 
 assuming that $\bm{V}$ is closed under pullbacks.

 For an object $i\in I_0$, the definition of $\lGamma(p)(i)_0$ is the same. 
 For $(x,f), (x',f') \in \lGamma(i)_0$, define an object
 $\lGamma(p)(i)((x,f),(x',f'))$ in $\bm{V}$ by
 \[
  \lGamma(p)(i)((x,f),(x',f')) = \bigoplus_{u: f'\circ u=f}
 X^u(x,x'). 
 \]
 The composition
 \[
  \lGamma(p)(i)((x',f'),(x'',f''))\otimes \lGamma(p)(i)((x,f),(x',f'))
 \longrightarrow \lGamma(p)(i)((x,f),(x'',f'')) 
 \]
 is given by the composition
 \[
  X^{u'}(x',x'') \otimes X^u(x,x') \longrightarrow X^{u'\circ u}(x,x'')
 \]
 on each component.

 For a morphism $u : i \to j$ in $I$, define a $\bm{V}$-functor
 \[
  \lGamma(p)(u) : \lGamma(p)(i) \longrightarrow \lGamma(p)(j)
 \]
 as follows. For an object $(x,f)$ in $\lGamma(p)(i)$, define
 \[
  \lGamma(p)(u)(x,f) = (x,u\circ f).
 \]
 For morphisms, define
 \[
  \lGamma(p)(u) : \lGamma(p)(i)((x,f),(x',f')) \longrightarrow
 \lGamma(p)(j)((x,u\circ f),(x,u\circ f'))
 \]
 by the identity
 \[
  X^v(x,x') \longrightarrow X^{v}(x,x')
 \]
 on each component.
\end{remark}

It is straightforward to check that the above construction is compatible
with the compositions and the identities.

\begin{lemma}
 The above construction defines a functor
 \[
  \lGamma(p) = p\downarrow (-) : I \longrightarrow \categories{\bm{V}}.
 \]
\end{lemma}

We would like to extend $\Gamma$ to a $2$-functor
\[
 \lGamma : \lcats{\bm{V}}\downarrow (I\otimes 1) \longrightarrow
 \lOplax(I,\categories{\bm{V}}). 
\]

\begin{definition}
 Let 
 \begin{eqnarray*}
  p & : & X \longrightarrow I\otimes 1 \\
  p' & : & X' \longrightarrow I\otimes 1
 \end{eqnarray*}
 be objects in $\lcats{\bm{V}}\downarrow (I\otimes 1)$.
 For a morphism
 \[
  (F,\varphi) : p \longrightarrow p',
 \]
 define a morphism of oplax functors
 \[
  (\lGamma(F,\varphi),\Gamma(\varphi)) : \lGamma(p) \longrightarrow
 \lGamma(p') 
 \]
 as follows. For $i\in I_0$, a $\bm{V}$-functor
 \[
  \lGamma(F,\varphi)(i) : \lGamma(p)(i) = p\downarrow i
 \longrightarrow p'\downarrow i = \lGamma(p')(i)
 \]
 is defined on objects by 
 \[
  \lGamma(F,\varphi)(i)(x,f) = (F(x),f\circ \varphi(x))
 \]
 and on morphisms by the commutativity of the following diagram
 \begin{center}
  \scalebox{.7}{%
  $\begin{diagram}
   \node{(p\downarrow i)((x,f),(x',f'))} \arrow[3]{e} \arrow[4]{s}
    \arrow{se,..} \node{} \node{} 
   \node{X(x,x')} \arrow[2]{s} \arrow{sw,t}{F} \\
   \node{} \node{(p'\downarrow i)((F(x),f\circ
   \varphi(x)),(F(x'),f'\circ\varphi(x')))} \arrow{e} \arrow[2]{s}
   \node{X'(F(x),F(x'))} \arrow{s} \node{} \\ 
   \node{} \node{} \node{(I\otimes 1)(p'\circ F(x),p'\circ F(x'))}
    \arrow{s} 
    \node{(I\otimes 1)(p(x),p(x'))}  
    \arrow[2]{s} \\
    \node{} \node{1} \arrow{sw,=} \arrow{e} \node{(I\otimes 1)(p'\circ
    F(x),i)} 
    \arrow{se,t}{(\varphi(x)')^*}\node{} \\ 
    \node{1} \arrow[3]{e} \node{} \node{} \node{(I\otimes 1)(p(x),i)}
  \end{diagram}$
  }
 \end{center}

 For each $u : i \to j$ in $I$, define a $\bm{V}$-natural
 transformation 
 \[
 \lGamma(\varphi)(u) : \lGamma(p')(u)\circ \lGamma(F,\varphi)(i)
 \Longrightarrow \lGamma(F,\varphi)(j)\circ \lGamma(p)(u) 
 \]
 to be the identity.
\end{definition}

It is straightforward to check that the pair
$(\lGamma(F,\varphi),\lGamma(\varphi))$ is a morphism of oplax
functors. This correspondence extends as follows.

\begin{lemma}
 The above construction $\Gamma$ defines a functor 
 \[
  \lGamma : (\lcats{\bm{V}}\downarrow I\otimes 1)(p,p')
 \longrightarrow
 \lOplax(I,\categories{\bm{V}})(\Gamma(p),\Gamma(p')). 
 \]
\end{lemma}

\begin{proof}
 This follows from Lemma \ref{Gamma_as_functor}.
\end{proof}

Proposition \ref{Gamma_as_2-functor} specializes to the following.

\begin{proposition}
 $\lGamma$ defines a $2$-functor
 \[
 \lGamma : \lcats{\bm{V}}\downarrow(I\otimes 1) \longrightarrow
 \lOplax(I,\categories{\bm{V}}). 
 \]
\end{proposition}

We have seen in Example \ref{comma_category_as_comodule} that
$\lcats{\bm{V}}_I$ can be identified with
$\lcats{\bm{V}}\downarrow (I\otimes 1)$. By translating $\lGamma$
for comma categories into $\lGamma$ for comodules, we obtain the
following corollary to Theorem \ref{Gr_and_Gamma_are_adjoint}.

\begin{theorem}
 Let $\bm{V}$ be a product type symmetric monoidal category. Then the 
 $2$-functor 
 \[
  \lGamma : \lcats{\bm{V}}\downarrow(I\otimes 1) \longrightarrow
 \lOplax(I,\categories{\bm{V}})
 \]
 is right adjoint to
 \[
  \Gr : \lOplax(I,\categories{\bm{V}}) \longrightarrow
 \lcats{\bm{V}}\downarrow(I\otimes 1).
 \]
\end{theorem}

\addcontentsline{toc}{section}{References}
\bibliographystyle{halpha}
\bibliography{%
\bibdir/mathAb,%
\bibdir/mathAl,%
\bibdir/mathAn,%
\bibdir/mathA,%
\bibdir/mathB,%
\bibdir/mathBa,%
\bibdir/mathBe,%
\bibdir/mathBo,%
\bibdir/mathBr,%
\bibdir/mathBu,%
\bibdir/mathCh,%
\bibdir/mathCo,%
\bibdir/mathC,%
\bibdir/mathDa,%
\bibdir/mathDe,%
\bibdir/mathD,%
\bibdir/mathE,%
\bibdir/mathFa,%
\bibdir/mathFo,%
\bibdir/mathFr,%
\bibdir/mathF,%
\bibdir/mathGa,%
\bibdir/mathGo,%
\bibdir/mathGr,%
\bibdir/mathG,%
\bibdir/mathHa,%
\bibdir/mathHe,%
\bibdir/mathH,%
\bibdir/mathI,%
\bibdir/mathJ,%
\bibdir/mathKa,%
\bibdir/mathKo,%
\bibdir/mathK,%
\bibdir/mathLa,%
\bibdir/mathLe,%
\bibdir/mathLu,%
\bibdir/mathL,%
\bibdir/mathMa,%
\bibdir/mathMc,%
\bibdir/mathMi,%
\bibdir/mathM,%
\bibdir/mathN,%
\bibdir/mathO,%
\bibdir/mathP,%
\bibdir/mathQ,%
\bibdir/mathR,%
\bibdir/mathSa,%
\bibdir/mathSe,%
\bibdir/mathSt,%
\bibdir/mathS,%
\bibdir/mathT,%
\bibdir/mathU,%
\bibdir/mathV,%
\bibdir/mathW,%
\bibdir/mathX,%
\bibdir/mathY,%
\bibdir/mathZ,%
\bibdir/physics,%
\bibdir/personal,%
\bibdir/japanese}

\end{document}